\documentclass[]{amsart}

\usepackage{amsaddr}
\usepackage{amssymb}
\usepackage{amsthm}
\usepackage{thmtools}
\usepackage{graphicx}
\usepackage{bm} 
\usepackage{microtype}
\usepackage{dsfont}
\usepackage{mathtools}
\usepackage{tikz}
\usetikzlibrary{graphs}
\usepackage[backend=biber,bibencoding=utf-8]{biblatex}
\usepackage{csquotes}
\usepackage{comment}
\usepackage{hyperref}
\usepackage[capitalize]{cleveref}
\usepackage{booktabs,multirow}
\usepackage{pgfplotstable}
\usepackage{booktabs, siunitx}

\usetikzlibrary{decorations.markings}

\newcounter{mythm}
\crefalias{mythm}{section}
\numberwithin{mythm}{section}

\declaretheorem[style=plain, sibling=mythm]{theorem}
\declaretheorem[style=plain, sibling=mythm]{lemma}

\declaretheorem[style=plain, sibling=mythm]{corollary}
\declaretheorem[style=definition, sibling=mythm]{definition}
\declaretheorem[style=definition, sibling=mythm]{example}
\declaretheorem[style=remark, sibling=mythm]{remark}

\def\XXint#1#2#3{{\setbox0=\hbox{$#1{#2#3}{\int}$}
     \vcenter{\hbox{$#2#3$}}\kern-.5\wd0}}

    \renewcommand{\Re}[1]{\mathrm{Re}(#1)}
    \renewcommand{\Im}[1]{\mathrm{Im}(#1)}

\definecolor{mycor}{rgb}{1, 0.6, 0}
\definecolor{mycom}{rgb}{1, 0.4, 0}

\begin{document}
\title[Graph lattice sums]{Graph lattice sums and graph zeta functions for long-range interacting quantum lattice models}

\author{Andreas A. Buchheit}
\address{Saarland University, Department of Mathematics, 66123 Saarbrücken, Germany \\ 
ETH Zürich, Department of Mathematics, 8092 Zürich, Switzerland}

\author{Andreas Rupp}
\address{Saarland University, Department of Mathematics, 66123 Saarbrücken, Germany}

\begin{abstract}

Taming the exponential increase of the Hilbert space dimension with system size in the simulation of gapped quantum lattice models is of the highest relevance for understanding and designing exotic quantum materials, where nonlocal interactions are of particular interest. High-order linked-cluster expansions provide access the solution of the eigenvalue problem for the infinite system, yet rely on the computation of high-dimensional oscillatory lattice sums with a graph structure, only approachable with Monte Carlo methods so far. This work resolves this issue, rendering all required graph lattice sums, referred to as graph zeta functions for kernels involving power-laws, computable. The resulting method reduces the evaluation time for state-of-the art series expansions from tenthousands of core-hours to minutes on a single core with controlled numerical precision. After factorizing the lattice sum over blocks, each block is evaluated by the cheapest available strategy depending on its treewidth $\mathrm{tw}$, while exploiting redundancies through caching. Basic blocks, such as bridges and cycles, admit analytic forms in terms of generalized zeta functions. Series-parallel blocks with $\mathrm{tw}\le 2$ can be computed at linear cost in the number of graph nodes and log-linear cost in the size of the momentum grid using a semi-analytical algebra based on Epstein zeta functions and rapidly decaying Fourier series. Finally, for $\mathrm{tw}>2$, the method is combined with tensor-network bucket elimination yielding polynomial scaling of numerical work and memory in momentum grid size with exponents only growing with $\mathrm{tw}$ rather than with the number of vertices. Through use of FFT, the full momentum grid is recovered at the cost of a single momentum evaluation. We provide a detailed analysis of the precision and runtime of our method against analytic and numerical benchmarks. Going further, we reproduce published Monte Carlo data for the well-known transverse-field Ising model on different 1D, 2D, and 3D lattices, obtaining full agreement within the references' statistical errors. The evaluation of the dispersion relation for a 3D quantum lattice on a momentum grid with 4096 momentum points is achieved within ten minutes on a single core. Building on this mathematical foundation, the companion paper \cite{duft2026} develops a comprehensive perturbative framework for gapped quantum lattice models, with applications reproducing experimental observations. The method is implemented in the open-source, high-performance Graph Zeta Library (GZL).
\end{abstract}

\maketitle

\section{Introduction}

Long-range interactions that decay as a power-law $|\bm x|^{-\nu}$ of the distance $\bm x\in \mathds R^d$ appear at all scales in nature, ranging from the interactions binding together subatomic quarks, over electromagnetic interactions between particles in materials, to gravitational interaction between galaxies \cite{campa2014physics}. For the classical simulation of $N$ long-range interacting particles, however, the evaluation of large scale sums with the number of summands scaling as $N$ is required in the evaluation of forces and sums with $N^2$ summands for energies. As typical material samples include particle numbers on the order of the Avogadro constant, $N\approx 10^{23}$, a challenging numerical problem is encountered when simulating long-range interacting systems. Integral approximations are frequently used, as they don't scale in complexity with number of particles, yet they introduce systematic errors that can qualitatively alter results \cite{buchheit2023exact}.  The situation becomes even more severe when investigating long-range interacting quantum systems. Here, system properties are encoded in a Hamiltonian $H$, whose dimension increases exponentially with the number of lattice sites. 

The precise simulation of interacting quantum lattice models on a classical computer is of central interest in condensed matter physics and theoretical chemistry. Direct approaches such as exact diagonalization are, however, limited to small site numbers $N\approx 25$ as the associated Hilbert space dimension exhibits an exponential growth in $N$. Under usage of symmetries and with supercomputing ressources, $N\approx 50$ has been reached for spin-1/2 \cite{Wietek2018}. Despite these advances, the limit of infinite lattice points forms the focal point of interest. 

A promising route for extracting rigorously the properties of infinite quantum lattice models is based on high-order perturbation theory in the form of linked-cluster expansions. To this end, one decomposes the Hamilonian of a finite quantum lattice model with $N$ lattice sites as $H=H_{0,N}+\lambda H_{\mathrm{int},N}$ with $H_{0,N}$ an unperturbed Hamiltonian with known bounded and equidistant spectrum and $H_{\mathrm{int},N}$ a perturbation with the prefactor $\lambda \in \mathds R$. Subsequently, the problem of determining the low-lying spectrum of $H$ is transformed into computing high-order Taylor series expansions in the parameter $\lambda$. 
The associated perturbation theory in the limit $N\to \infty$ is made accessible by linked-cluster expansions  \cite{adelhardt2024monte,hormann2023projective}. 
They cast the analytic series coefficients exactly in terms of iterated high-dimensional lattice sums associated with graphs, whose dimension increases linearly with the expansion order. The approach is formally exact within the radius of convergence and has been applied to various long-range interacting quantum spin systems such as the famous transverse-field Ising \cite{koziol2024order} or the XXZ model \cite{adelhardt2025quantum}. It provides reliable results even in the case two-dimensional frustrated systems. In such systems, quantum Monte Carlo methods are often not effective due to potential sign problems \cite{sandvik2010computational} and density matrix renormalization group (DMRG) is not effective due to long-range interactions on a higher-dimensional lattice \cite{stoudenmire2012studying}. In addition, graph decomposition techniques allow for the study of both the lowest lying eigenvalue, associated with the ground state, as well eigenvalues of excited states. For this reason, they are currently being explored as impactful tools for studying highly relevant frustrated quantum spin systems, such as long-range interacting Kagome lattices \cite{semeghini2021probing,samajdar2021quantum}. The main obstacle that hinders widespread use of these methods lies in the computation of the resulting high-dimensional, slowly convergent sums where Monte Carlo sums form the current standard \cite{adelhardt2024monte}, suffering from cluster-scale runtimes and stochastic errors. When computing not only ground state properties but also dispersion relations, the resulting sums are rendered oscillatory, and the sums need to be recomputed for every choice of wavevector $\bm k$ in the unit cell of the reciprocal lattice, complicating the computational problem further.

In previous work \cite{buchheit2025epstein_method}, we have shown that iterated lattice sums with a circle graph structure can be computed at linear complexity with node number, avoiding the exponential increase of numerical work of standard approaches, with direct applications in the stability of crystal lattices \cite{robles2025exact}.
The zeta function for circular graphs is given as follows.
\begin{definition}[Circle zeta function]
Let $\Lambda$ be a $d$-dimensional lattice. For $n\in \mathds N$ and $\bm \nu\in \mathds C^n$, we define the circle zeta function as
\[
\zeta^{(n)}(\bm \nu) =  \,\sideset{}{'}\sum_{\bm x_1,\dots,\bm x_{n-1}\in \Lambda} \prod_{j=1}^n  \vert\bm x_j-\bm x_{j-1}\vert^{-\nu_j}, \quad \mathrm{Re}(\nu_j)>d,
\]
where the primed sum excludes summands with coincident positions $\bm x_j-\bm x_{j-1}$ and with an arbitrary $\bm x_0=\bm x_n\in \Lambda$, which can be set to zero due to translational invariance. The circle zeta function can be meromorphically continued to $\bm \nu\in \mathds C^n$.
\end{definition}
In our previous work, we have derived the following integral representation for circle zeta functions \cite[Theorem~2.6]{buchheit2025epstein_method}, which factorizes the sum, thereby providing linear complexity in the number of nodes. It is based on the Epstein zeta function $Z_{\Lambda,\nu}(\bm k)$ \cite{epstein1903theorieI,epstein1903theorieII}, made efficiently computable in \cite{buchheit2024epstein} with a high-performance open-source implementation in \href{https://github.com/epsteinlib/epsteinlib}{EpsteinLib}. Using this special function, the following integral representation for the circle zeta function is obtained.
\begin{theorem}
The circle zeta function admits the integral representation
\[
    \zeta^{(n)}(\bm \nu) = V_\Lambda \int_{\mathrm{BZ}} \prod_{i=1}^n Z_{\Lambda,\nu_i}(\bm k),\quad \mathrm{Re}(\nu_j)\ge d,\quad j=1,\dots,n.
\]
The meromorphic continuation of the function to $\bm \nu \in \mathds C^N$ can be explicitly constructed using the Hadamard integral \cite{gelfand1964generalizedI}.
\end{theorem}
The goal of this work is to carry this exponential reduction in complexity over to all graphs appearing in the simulation of long-rang interacting quantum lattice models. The resulting lattice sums and zeta functions then readily permit efficient simulations their properties with far-reaching impact across numerical analysis, theoretical physics and theoretical chemistry. 
This work achieves this goal, by combining techniques from graph theory with methods for generalized zeta functions and high-dimensional lattice sums. It makes both single graph sums and whole corpora of graphs sums, that form the perturbative series coefficients, efficiently computable. 

All numerical experiments reported in this work were carried out with the Graph Zeta Library (GZL)~\cite{gzl2026}, a high-performance open-source implementation of the framework developed in this article, which allows to reproduce all results on a standard laptop. This article is published in parallel with Ref.~\cite{duft2026}, in which the library and the methods developed here are applied at scale to the long-range transverse-field Ising model in dimensions $d=1,2,3$ on several lattice geometries. The present work provides the mathematical framework, the proofs, and the numerical analysis, while Ref.~\cite{duft2026} integrates the numerical method into linked-cluster perturbation theory, thereby creating a comprehensive high-order perturbative scheme applicable for general gapped quantum lattice models, and derives the physical results.

This work is structured as follows. In \cref{sec:main_result}, we define the graph lattice sums and graph zeta functions appearing in the perturbative treatment of quantum lattice models and analyze their basic properties. 
In \cref{sec:block_factorization}, we significantly reduce evaluation complexity by factorizing graph lattice sums over blocks for general interaction kernels. We further use block caching, making use of graph isomorphisms, to avoid recomputation of blocks. The following \cref{sec:basic_blocks} then discusses the efficient evaluation of basic blocks, such as bridges and cycles, which can be especially challenging for power-law kernels with exponents close to the lattice dimension. This problem is solved based on generalized zeta functions such as the Epstein zeta function and allows for evaluation of trees and simple cycles without momenta. For series-parallel blocks, we demonstrate that any graph can be generated by multiplication and convolution of base objects in \cref{sec:sp-construction}. For kernels with power-law tails, singular convolutions arise that are efficiently computed based on a representation in terms of Epstein zeta functions plus a rapidly decaying Fourier series in \cref{sec:algebra}. Highly-connected graphs are made computable through tensor-network bucket elimination in \cref{sec:tensor}. We benchmark our method against known results and exact summation in \cref{sec:benchmarks}, achieving stable convergence even for interaction exponents $\nu$ close to the system dimension. 
In \cref{sec:lrtfim}, we apply our method to the well-known long-range transverse field Ising model and demonstrate that it reproduces Monte Carlo results, obtained within a day on a cluster, within minutes on a laptop. We draw our conclusions and provide an outlook in Section~\ref{sec:outlook}. 

\section{Problem statement and graph lattice sums}
\label{sec:main_result}
We begin by introducing some basic notation, starting with the concept of a Bravais lattice.
\begin{definition}[Lattices]
    We denote a periodic point set $\Lambda=A\mathds Z^d$, with $A\in \mathds R^{d\times d}$ regular, a monoatomic or Bravais lattice. The elementary lattice cell is denoted by $E_\Lambda=A[-1/2,1/2]^d$ with volume $V_\Lambda=|\det A|$. We further define the reciprocal lattice $\Lambda^\ast=A^{-T}\mathds Z^d$ with reciprocal lattice cell $\mathrm{BZ}=A^{-T}\mathds T^d$ with the torus $\mathds T=[0,1)$.
\end{definition}

In the following, we will connect the high-dimensional lattice sums appearing in the perturbative treatment of quantum lattice models with graph theory. We therefore introduce some necessary notation for weighted multi-graphs with terminals, slightly extending the notation from Diestel \cite{diestel2025graph}.

\begin{definition}[Graphs]
We use the following graph-theoretic notation.
\begin{enumerate}
\item A multigraph $G=(V,E,\partial)$ consists of a node set $V\subset\mathds N$, an edge set $E$, and an incidence map
\[
\partial:E\to V\times V,\qquad \partial(e)=(e^-,e^+),\qquad
e^-\neq e^+,
\]
assigning to each oriented edge $e\in E$ its tail $e^-$ and head $e^+$. Two distinct edges are called parallel if their incidences coincide up to orientation, that is, if $\{e_1^-,e_1^+\}=\{e_2^-,e_2^+\}$. The multigraph is called finite if $V$ and $E$ are finite, and simple if it contains no parallel edges. It is called connected if every pair of distinct nodes can be joined by a path, disregarding edge orientations. A connected multigraph without cycles is a tree. A spanning tree of a connected multigraph is a subset $T\subseteq E$ of its edges forming a tree on all of $V$, and the edges outside $T$ are called chords. In the following, we suppress the incidence map and write $G=(V,E)$.
\item A weighted multigraph is a tuple $G=(V,E,K)$, where $(V,E)$ is a multigraph and $K:E\to X$ is a weight function into an arbitrary set $X$.
\item A two-terminal weighted multigraph is a tuple $G=(V,E,K,s,t)$, where $(V,E,K)$ is a weighted multigraph and $s,t\in V$ are distinguished nodes, called the terminals, with $s$ the source and $t$ the sink.
\end{enumerate}
\end{definition}

We then define the main object of interest of this work, namely the high-dimensional iterated lattice sums that form the series coefficients in high-order perturbative series of gapped quantum lattice systems with non-trivial interactions \cite{adelhardt2025quantum,fey2019quantum,adelhardt2024monte}. For general interaction kernels, we call these objects graph lattice sums. For the particularly important case of power-law kernels accompanied by compactly supported corrections, we refer to them as graph zeta functions. We begin our discussion with the general case. 

\begin{definition}[Graph lattice sums]
\label{def:graph-lattice-sum}
Consider a Bravais lattice $\Lambda\subseteq\mathds R^d$ and $G=(V,E,K,s,t)$ a finite connected two--terminal multigraph, where each edge $e\in E$ carries an orientation, written $e=(e^-,e^+)$, and an absolutely summable kernel $K_e\in\ell^1(\Lambda)$. For a wave vector $\bm k\in \mathrm{BZ}$, we define the graph lattice sum ${\mathcal Z}_{\Lambda,G}:\mathrm{BZ}\to\mathds C$ by
\[
{\mathcal Z}_{\Lambda,G}(\bm k)
\;=\;\sum_{\{\bm x_{v}\in\Lambda\}_{v\neq p}}
e^{-2\pi i\delta\bm x_{(s,t)}\cdot\bm k}
\prod_{e\in E}
{K}_{e}\big(\delta\bm x_e\big),
\]
with $\delta\bm x_e=\bm x_{e^+}-\bm x_{e^-}$, where the indexed family runs over $v\in V\setminus\{p\}$ with an arbitrary pinning
node $p\in V$ held fixed. In cases, where the choice of the lattice is immaterial, we use the shortened notation $\mathcal Z_G$.
\end{definition}
Graph lattice sums are used to define physical quantities per lattice site in an infinite quantum lattice system, examples being ground state energy densities or dispersion relations, under the assumption that all physically relevant correlation functions inherit translational invariance from the underlying lattice $\Lambda$ and from the physical model. These functions therefore only depend on relative distances of lattice sites, which results in the removal of the pinning vertex $p$ from summation. The proof of the following elementary properties crucally relies on translational invariance. 

\begin{lemma}[Elementary properties]
\label{lem:choice_of_x}
The following properties hold:
\begin{enumerate}
\item The graph lattice sum is absolutely convergent and hence well-defined. For any spanning tree $(V,T)$ with edges $T\subseteq E$, it is uniformly bounded in $\bm k$ by
\[
|{\mathcal Z}_{\Lambda,G}(\bm k)|\le
\prod_{e\in T} \Vert K_e\Vert_{\ell^1}
\prod_{e\in E\setminus T} \Vert K_e\Vert_{\ell^\infty}.
\]
\item The value of the sum does not depend on the choice of the pin $p\in V$ and the associated position $\bm x_p$.
\item For equal terminals $s=t$, the graph lattice sum is independent of $\bm k$ and of the common terminal. At $\bm k=\bm 0$, the choice of terminals does not alter the value of the sum. 
\end{enumerate}
\end{lemma}
\begin{proof}
(1) As the modulus of the phase factor is bounded by one, it suffices to establish a bound on
\[
\sum_{\{\bm x_{v}\in\Lambda\}_{v\neq p}}
\prod_{e\in E}
|{K}_{e}\big(\delta\bm x_e\big)|.
\]
First note that, as $K_e$ is summable, it is bounded, with 
\[
\Vert K_e\Vert_{\ell^\infty}\le \Vert K_e\Vert_{\ell^1}.
\]  
Since $G$ is connected, there exists a spanning tree $(V,T)$ with $|V|-1$ edges $T\subset E$ covering the node set $V$. We can then bound all kernels not associated to the spanning tree by their $\ell^{\infty}$ norms,
\[
\prod_{e\in E}
|{K}_{e}\big(\delta\bm x_e\big)| \le  \prod_{e\in T} \vert{K}_{e}\big(\delta\bm x_e\big)\vert \prod_{e\in E\setminus T} \Vert K_e\Vert_{\ell^\infty}.
\]
Root the spanning tree at the pin $p$. After reversing edges and replacing $K_e$ by the reflected kernel $K_e(-\,\cdot\,)$ where necessary, which changes neither the moduli above nor any of the norms, we may assume that every tree edge is oriented away from the root. Then every node $v\neq p$ is the head of exactly one tree edge $e(v)$. The coordinate transform
\[
\bm y_{v}=\bm x_{v}-\bm x_{e(v)^-}, \qquad v\in V\setminus\{p\},
\]
then maps $\Lambda^{\vert V\vert-1}$ bijectively onto itself. As $\Lambda$ is a group, the original positions are recovered recursively from the root via $\bm x_v=\bm x_{e(v)^-}+\bm y_v$, starting from the pinned position $\bm x_p$. After the variable transform, each kernel depends on a single variable. By Tonelli's theorem, we may interchange sum and product, and obtain
\[
\sum_{\{\bm x_{v}\in\Lambda\}_{v\neq p}}
\prod_{e\in T}
\vert{K}_{e}\big(\delta\bm x_e\big)\vert
= \prod_{v\in V\setminus \{p\}} \sum_{\bm y_{v}\in \Lambda}
\vert K_{e(v)}(\bm y_{v})\vert
= \prod_{e\in T} \Vert K_e\Vert_{\ell^{1}}.
\]   
For any spanning tree $(V,T)$, we therefore obtain
\[
\sum_{\{\bm x_{v}\in\Lambda\}_{v\neq p}}
\prod_{e\in E}
\vert {K}_{e}\big(\delta\bm x_e\big)\vert\le \prod_{e\in T} \Vert K_e\Vert_{\ell^1} \prod_{e\in E\setminus T} \Vert K_e\Vert_{\ell^\infty}.
\]
Hence, $\mathcal Z_{\Lambda,G}(\bm k)$ is well-defined and is uniformly bounded in $\bm k$ as above.

(2) We first show independence of the sum on the pin position $\bm x_p$. To this end, translate this position by $\bm \tau \in \Lambda$, obtaining $\bm x_p\to \bm x_p + \bm \tau$. The additional substitution $\bm x_v \to \bm x_v+\bm \tau$, $v\neq p$, then maps $\Lambda^{|V|-1}$ bijectively to $\Lambda^{|V|-1}$. Together with the shift of $\bm x_p$, it leaves all differences $\bm x_{e^+}-\bm x_{e^-}$ and $\bm x_{t}-\bm x_s$ unchanged. Hence the value of $\mathcal Z_{\Lambda,G}(\bm k)$ is unchanged, thus is does not depend on $\bm x_p$. 

Second, we show independence on the choice of pin $p\in V$. Choose a second pin $q\in V$ and, by the already shown independence of the pinned position, place both pins at the same position $\bm\xi\in\Lambda$. We want to show that
\[
\sum_{\{\bm x_v\in\Lambda\}_{v\neq p}} f(X)
=\sum_{\{{\bm x}_v\in\Lambda\}_{v\neq q}} f( X),
\]
where $f$ is the summand function, which depends only on the the differences
\[
X_{i,j}=\bm x_i-\bm x_j,\qquad i,j\in V,
\]
both in the kernels and in the phase factor. Consider the change of variables
\[
\bm x_i={\bm y}_i-{\bm y}_{p}+{\bm y}_{q},
\qquad i\in V.
\]
The transformed positions again lie in $\Lambda$, and the pin constraint is preserved, as $\bm x_p={\bm y}_q=\bm\xi$. The substitution is further bijective and is inverted by the same formula with $p$ and $q$ interchanged, with ${\bm y}_i=\bm x_i-\bm x_{q}+\bm x_{p}$. Moreover, all differences
are preserved,
\[
X_{i,j}=\bm x_i-\bm x_j={\bm y}_i-{\bm y}_j,
\]
so the summands coincide pairwise. Since both sums are absolutely convergent, the reindexed unordered sums agree. Changing notation from $\bm y_v$ to $\bm x_v$ yields the
identity.

(3) If either $s=t$ or $\bm k=\bm 0$, the phase factor equals one and the dependency on both $\bm k$ and on $\bm x_s$ and $\bm x_t$ is removed. 
\end{proof}

\begin{remark}
    The uniform bound in Lemma~\ref{lem:choice_of_x} shows that the summability requirement $K_e\in \ell^1(\Lambda)$ can be relaxed to $K_e\in \ell^\infty(\Lambda)$ with the additional condition that a spanning tree $(V,T)$ exists, such that
    \[
        \prod_{e\in T} \Vert K_e\Vert_{\ell^1}<\infty.
    \]
\end{remark}
Multi-edges naturally appear within perturbative series expansion. However, both for analytical and numerical purposes, it is often useful to study simple graphs instead. The following Lemma shows that multi-edges can be merged into a single edge where the new kernel function is given by the Hadamard product of the kernel functions of the constituents, while accounting for differing edge orientations through sign changes.

\begin{lemma}[Edge merging]
\label{lem:merge}
Let $G=(V,E,K,s,t)$ be a finite connected two-terminal weighted multigraph with kernels $K_e\in\ell^1(\Lambda)$, and let $\tilde G=(V,\tilde E,\tilde K,s,t)$ be the graph obtained by replacing, for each pair of nodes $u<v$ joined by at least one edge, all edges between $u$ and $v$ by a single edge $\tilde e$ with
$\tilde e^-=u$, $\tilde e^+=v$ and kernel
\[
\tilde K_{\tilde e}
=
\prod_{\substack{e\in E\\ \partial(e)=(u,v)}} K_e
\prod_{\substack{e\in E\\ \partial(e)=(v,u)}} K_e(-\,\cdot\,).
\]
Then $\tilde G$ is simple and connected,
$\tilde K_{\tilde e}\in\ell^1(\Lambda)$ for every $\tilde e\in\tilde E$, and
\[
\mathcal Z_{\Lambda,\tilde G}\;=\;\mathcal Z_{\Lambda,G}.
\]
\end{lemma}
\begin{proof}
    By construction, $\tilde G$ is connected, as merging edges does not change which nodes as joined, and simple. The associated graph lattice sum converges absolutely as the merged kernels are summable as finite products of summable kernels. For any edge $e$ connecting $u,v\in V$, we have that 
    $K_e(\delta\bm x_{e})=K_e\big(\pm(\bm x_{v}-\bm x_{u})\big)$ with the sign depending on the edge orientation. The factors of all edges connecting $u$ and $v$ then combine into $\tilde K_e$ and the lattice sums therefore coincide.
\end{proof}

A particularly important case of graph lattice sums occurs, if the kernel $K_e$ is formed by a sum of regularized power-laws, accompanied by a compactly supported short-range correction. We refer to the associated graph lattice sums as graph zeta functions.

\begin{definition}[Graph zeta functions]
    \label{def:graph_zeta_functions}
   Consider the graph lattice sum in \cref{def:graph-lattice-sum} for edge kernels $K_e$ of the form
    \[
    K_e(\bm x) = a_e(\bm x) +\sum_{j=1}^\ell b_{ej}\mathcal K_{\nu_{ej}}(\bm x),
    \]
    with $a_e:\Lambda\to \mathds C$ even and compactly supported, $b_{ej}\in \mathds C$, and $\mathcal{K}_{\nu}(\bm x) = \vert \bm x\vert^{-\nu}$, $\bm x\neq \bm 0$ and $\mathcal K_{\nu}(\bm 0)=0$.
    The summability condition $K_e\in\ell^1(\Lambda)$ then corresponds to $\min_{e,j}\mathrm{Re}(\nu_{ej})>d$. We denote the resulting lattice sum by $\zeta_{\Lambda,G}$, abbreviated $\zeta_G$, and
    refer to it as a graph zeta function.
\end{definition}

Graph zeta functions are closely connected to zeta functions from the classical literature, if all edge kernels are given by a simple power law $\mathcal K_\nu$. For $G$ a simple bridge, and $\mathds Z=\Lambda$, we have $\zeta_{\mathds Z,G}(0)=2\zeta(\nu)$, with $\zeta(\nu)$ the Riemann zeta function. For a simple bridge on a lattice $\Lambda$, the Epstein zeta function is recovered $\zeta_{\Lambda,G}(\bm k)=Z_{\Lambda,\nu}(\bm k)$ \cite{epstein1903theorieI,epstein1903theorieII}.

\begin{remark}
Graph zeta functions allow for several simplifications compared to general graph lattice sums. First, edge orientation becomes irrelevant, as each edge kernel is even. Second, exchanging source and terminal does not alter the lattice sum, and, for $K_e$ real, the graph zeta function is real as well. Finally, joining multi-edges reduces again to a graph zeta function as $a(\bm x) \mathcal K_\nu(\bm x)$ is even and compactly supported and as  $\mathcal K_{\nu} \mathcal K_{\nu'}=\mathcal K_{\nu+\nu'}$ is again a regularized power-law.
\end{remark}

In series expansions of gapped quantum lattice models, the series coefficients are given by combinatorial sums over different graphs. We call the collection of these graphs a corpus and the resulting sums corpus sums. If we are able to compute corpus sums up to high order, we can precisely extract the physical quantities of any gapped quantum lattice model \cite{adelhardt2024monte}.

We first define general interacting quantum lattice models in whose perturbative treatment graph lattice sums and graph zeta functions appear. 
\begin{definition}[Quantum lattice model]
\label{def:quantum-model}
On a Bravais lattice $\Lambda=A\mathds Z^d$, with $A\in\mathds R^{d\times d}$ regular, define the Hilbert space $\mathfrak h=\mathds C^m$, $m\in\mathds N$, called the local state space. The Hilbert space for a quantum system on the finite torus $\Lambda_n=\Lambda/(n\Lambda)$, with $|\Lambda_n|=n^d=N$ sites, is then
given by the tensor product
\[
\mathcal H_n = \bigotimes_{\bm x\in \Lambda_n} \mathfrak{h},\quad
\dim(\mathcal H_n) = m^{|\Lambda_n|}.
\]
Fix a self-adjoint operator $q$ on $\mathfrak h$ with equidistant spectrum bounded from below. We study the low-lying eigenvalues of self-adjoint operators $H_n$ on $\mathcal H_n$, called Hamiltonians, that admit the decomposition
\[
H_n= H_{0,n}+\lambda H_{\mathrm{int},n},\quad
H_{0,n}=\sum_{\bm x\in \Lambda_n} q_{\bm x},
\quad \lambda\in\mathds R,
\]
where $q_{\bm x}$ acts as $q$ on the $\bm x$-th tensor component and as the identity otherwise. The interaction operator $H_{\mathrm{int},n}$ is of the form
\[
H_{\mathrm{int},n} = \frac{1}{2}
\sum_{\substack{\bm x,\bm y\in \Lambda_N\\ \bm x\neq\bm y}}
K_n\big([\bm y-\bm x]\big)\, v_{\bm x\bm y},\quad
K_n([\bm z]) = \sum_{\bm w\in n\Lambda} K(\bm z+\bm w),
\]
with an even summable kernel $K\in \ell^1(\Lambda;\mathds R)$ and its image periodization $K_n$, whose defining sum converges absolutely. Moreover, $v_{\bm x\bm y}=v_{\bm y\bm x}$ is a self-adjoint operator acting only on the Hilbert spaces with positions $\bm x$ and $\bm y$, 
\[
v_{\bm x\bm y} = \sum_{p=1}^q c_p u^{(p)}_{\bm x} u^{(p)}_{\bm y},\quad q\ge 1,\quad c_p\in \mathds R,
\]
with $u_{\bm x}^{(p)}$ a collection of self-adjoint operators only acting on the local Hilbert space with position $\bm x$.
\end{definition}

The low-lying spectrum of the Hamiltonian in the limit $n\to \infty$ encodes the physically most relevant properties of the quantum system at zero temperature. We briefly discuss two such properties, namely the ground state energy density, related to the lowest lying eigenvalue, and the dispersion relation, describing excitations of the ground state. These and related properties will be rendered efficiently computable using the method developed in this work.  

\begin{definition}[Energy density and dispersion relation]
For a finite quantum lattice model as in \cref{def:quantum-model} of size per dimension $n\in \mathds N$, we define the finite  ground-state energy $E_0(\lambda,n)$ and the energy density $e_0(\lambda,n)$ as 
\[
E_0(\lambda,n) =  \min\mathrm{spec}\,H_n,\quad e_0(\lambda,n) = \frac{E_0(\lambda,n)}{n^d}.
\]
For the energy density, the thermodynamic limit of infinite sites $n\to \infty$ is well-defined and reads
\[
e_0(\lambda)=\lim_{n\to \infty} e_0(\lambda,n).
\]
As the interaction $K$ depends only on relative distances and $H_n$ is translation invariant, it is clear that the Hamiltonian is an orthogonal sum over momentum sectors on the discretized Brillouin zone $\mathrm{BZ}_n=(\Lambda^\ast/n)/\Lambda^\ast$,
\[
H_n= \bigoplus_{\bm k\in \mathrm{BZ}_n}  H_n^{(\bm k)}.
\]
The dispersion relation at momentum $\bm k$ is then defined as the energy difference of the first excited state of $H_n^{(\bm k)}$ to the ground state,
\[
\omega(\bm k,\lambda,n) = \min \bigg(\big(\mathrm{spec}~H_n^{(\bm k)}-E_0(\lambda,n)\big)\setminus\{0\}\bigg),\quad \bm k \in \mathrm{BZ_n}.
\]
As $n\to \infty$, it converges to $\omega(\bm k,\lambda)$ with $\bm k\in \mathrm{BZ}$.
\end{definition}

One of the most prototypical quantum lattice models is the long-range transverse-field Ising model, see i.e.\,\cite{adelhardt2024monte} for a review. In this work, it will serve as the primary benchmark for our method. 
\begin{definition}[Long-range transverse-field Ising model]
    \label{def:lrtfim}
    Consider a Bravais lattice $\Lambda\subseteq \mathds R^d$, a local Hilbert space $\mathfrak h=\mathds C^2$, and an even kernel $K\in \ell^1(\Lambda;\mathds R)$. In the sense of the periodization limit from \cref{def:quantum-model}, the Hamiltonian of the long-range transverse-field Ising model in units of the bare gap $2h$, with $h>0$ the transverse field, reads 
    \[
    H=\sum_{\bm x\in \Lambda} \frac{1}{2}\sigma^{x}_{\bm x}-\frac{\lambda}{2} \sum_{\substack{\bm x,\bm y\in \Lambda\\ \bm x\neq\bm y}} K(\bm x-\bm y)\sigma^z_{\bm x} \sigma^z_{\bm y},
    \]
    with the Pauli matrices
    \[
    \sigma^x=\begin{pmatrix}
        0 & 1\\
        1 & 0
    \end{pmatrix},\quad \sigma^z = \begin{pmatrix}
        1 & 0\\
        0 & -1
    \end{pmatrix}.
    \]
    Here, $\lambda=J/(2h)$, with $J\in \mathds R$ the coupling constant. The coupling is called ferromagnetic if $\lambda>0$ and antiferromagnetic if $\lambda<0$.
    The model is connected to \cref{def:quantum-model} via the replacements
    \[
    q_{\bm x}=\frac{1}{2}\sigma^x_{\bm x},\quad v_{\bm x\bm y}=-\sigma^z_{\bm x} \sigma^z_{\bm y}.
    \]
\end{definition}

The numerical work for extracting the spectrum of $H_n$ scales exponentially with the number of sites $N=n^d$. In practice, exact diagonalization can currently treat approximately 30 sites for long-range interacting spin-$1/2$ (local dimension $m=2$), with 50 sites reachable on supercomputing hardware for short-range interactions \cite{Wietek2018}, yet, from a physical point of view, system sizes of the order of the Avogadro constant $N_A\approx 10^{23}$ are the relevant scale. 
Graph decomposition techniques, offer representations of
the Taylor series coefficients of $e_0(\lambda)$ and $\omega(\bm k,\lambda)$ around $\lambda = 0$ in terms of graph lattice sums.

It has been shown in works by Knetter, Uhrig, and Schmidt~\cite{Knetter2000,Knetter2003}, building on the flow-equation method of Wegner~\cite{Wegner1994}, that the eigenvalue problem for the ground-state energy density and the low-lying energy bands can be transformed into the computation of iterated high-dimensional lattice sums with a graph structure, whose summation dimension increases linearly with the expansion order. One particular method for extracting the series coefficients is known as pCUT (perturbative continuous unitary transformations), which was extended to long-range interactions through white-graph expansions~\cite{Coester2015,Fey2016,fey2019quantum}, and has, by now, been applied to a broad range of long-range models (see the review~\cite{adelhardt2024monte} and references therein). Other variants of these perturbative schemes exist, all based on linked-cluster expansions, one important example being pCAT (projective cluster additive transformations) \cite{hormann2023projective}. Any perturbative method casts the Taylor series coefficients of observables as a collection of graphs with associated combinatorial prefactors, called a corpus in the following. This corpus is agnostic to the lattice and to the interaction kernel and needs to be precomputed only once for a specific model and observable, e.g.\ by the methods reviewed in~\cite{adelhardt2024monte}. It expresses the Taylor coefficients at $\lambda=0$ as finite linear combinations of graph lattice sums.

\begin{figure}
    \centering
    \includegraphics[width=\linewidth]{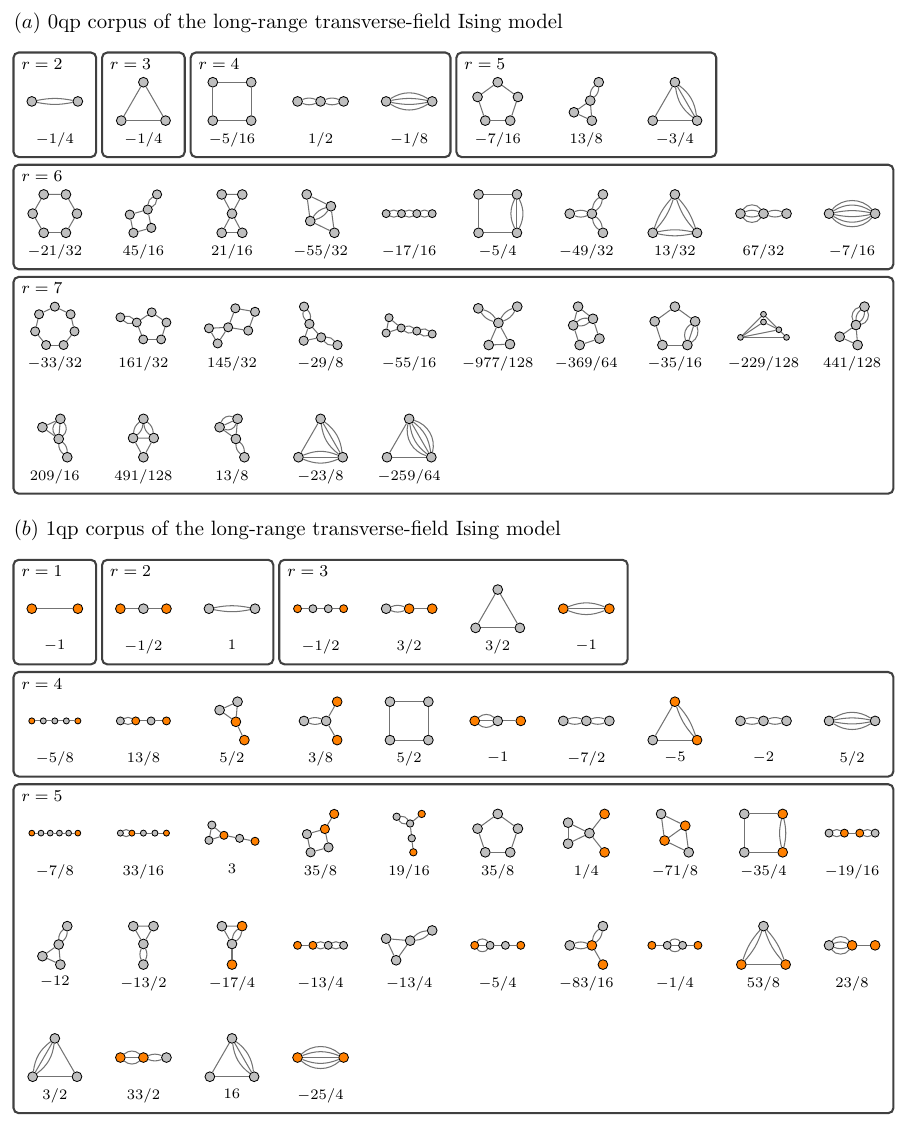}\\
    \caption{Low-order corpora of the  long-range transverse-field Ising model on an arbitrary $d$-dimensional lattice $\Lambda$, obtained via the pCUT method, where all edges $e\in E$ are associated with an identical kernel $K_e\in \ell^1(\Lambda)$. Panel (a) depicts the corpus for the ground state energy density $e_0(\lambda)$ (no terminals) up to expansion order 7  (up to $6d$-dimensional lattice sums), including the graphs $G$ and their prefactors $a_r(G)\in \mathds Q$. Panel (b) depicts the corpus for the single quasiparticle dispersion relation $\omega(\bm k,\lambda)$ (1qp) up to expansion order 5 with terminals colored in orange. In this work, we use state-of-the-art expansions of order $13$ for 0qp (8\,403 graphs, max.~dimension $(r-1) d$)  and order $11$ for 1qp (22\,677 graphs, max.~dimension $rd$), where all graphs are efficiently and precisely computable using our method.}
    \label{fig:tfim_o7}
\end{figure}

\begin{definition}[Corpus sums and perturbative series]
\label{def:corpus}
A corpus is a sequence $(\mathcal C_r)_{r\in\mathds N_0}$ of finite sets $\mathcal C_r$ of finite connected two-terminal multigraphs $G=(V,E,s,t)$, together with combinatorial prefactors $a_r(G)\in\mathds Q$ for $G\in\mathcal C_r$. For a kernel $K\in\ell^1(\Lambda)$, write $(G,K)$ for the two-terminal weighted multigraph obtained from $G$ by assigning the kernel $K$ to every edge. The $r$-th order corpus sum is
\[
c_{r}(\bm k)=\sum_{G\in\mathcal C_r}
a_r(G)\,\mathcal Z_{\Lambda,(G,K)}(\bm k),
\]
well-defined by \cref{lem:choice_of_x}, and the corpus sums assemble into the formal power series
\[
F(\bm k,\lambda)=\sum_{r=0}^{\infty} c_r(\bm k)\,\lambda^{r}
\]
in the expansion parameter $\lambda$. For graphs with equal terminals $s=t$, the corpus sums are independent of $\bm k$. 
\end{definition}

The corpora for the long-range transverse field Ising model are provided to in the Graph Zeta Library \cite{gzl2026}. As an example, we display the low-order corpora for this model in \cref{fig:tfim_o7}, which were obtained through the pCUT method, see \cite{duft2026} for details. The corpus for the energy density (0qp) is shown in panel (a), while the corpus for the single particle dispersion relation (1qp) is provided in (b).  For observables of quantum lattice models as in
\cref{def:quantum-model}, corpus sums yield the formal power series around $\lambda=0$ in the infinite-lattice limit. For the ground-state energy density, only graphs with equal terminals $s=t$ appear, and the corpus sums carry no momentum dependence, while the dispersion relation requires graphs with distinct terminals that are evaluated at momentum $\bm k$. 

\begin{remark}
The order $r=0$ contribution in the observables corresponds to a constant, typically excluded from the corpus, as it requires no sum evaluation. For the LRTFIM, as defined in \cref{def:lrtfim}, we have
\[
e_0(\lambda)=-\frac{1}{2}+\sum_{r=2}^\infty c_r \lambda^r,\quad \omega(\bm k,\lambda)= 1+\sum_{r=1}^\infty c_r(\bm k) \lambda^r.
\]
Note that the order $1$ term in the energy density vanishes.
\end{remark}

With precomputed corpora, the exact formal power series of obervables such as energy density or single-quasiparticle dispersion relation around $\lambda=0$ is determined through graph lattice sums or, in the case of combinations of short-range and power-law kernels, graph zeta functions.
In principle, this allows for an exact order-by-order expansion of the physically relevant observables of infinite quantum lattice systems, if the graph lattice sums can be reliable supplied, with vast application potential. The following three main problems arise in their evaluation.

\begin{enumerate}
    \item Evaluation complexity for each graph lattice sum increases exponentially with number of nodes for simple lattice truncations. 
    \item    For power-law kernels with exponent $\nu$, convergence becomes arbitrary slow as $\nu\searrow d$.
    \item Resolving the full Brillouin zone requires recomputation for every single momentum point.
    \item The number of graphs in the corpus grows exponentially with order.
\end{enumerate}

In the following sections, we solve issues $(1)-(3)$ and severely reduce the impact of $(4)$. As we will show later, this reduces the runtime for state-of-the-art perturbative expansions from tenthousands of core-hours for Monte Carlo sums for a single $\bm k$ value (i.e.~ in \cite{adelhardt2024monte}) to minutes on a standard laptop for the full momentum grid at improved precision. We begin with a factorization of the high-dimensional graph lattice sums into blocks.

\section{Block factorization of graph lattice sums}
\label{sec:block_factorization}

The numerical work of simple discretization approaches for evaluating graph lattice sums, replacing $\Lambda$ by a torus $\Lambda_N$ or a truncated lattice scale as $\mathcal O(N^{(|V|-1)d})$ and quickly become infeasible for large numbers of graph nodes $|V|$. We here show first, that for general absolutly summable interaction kernels $K_e\in \ell^1(\Lambda)$, translational invariance of the lattice $\Lambda$ can be leveraged to factorize the sums over blocks, thereby reducing the exponential scaling in the number of nodes $|V|$ to scaling with the largest number of nodes within a block. In particular, for $G$ a tree, scaling with $|V|$ reduces from exponential to linear in the node count. 

For the decomposition of graph lattice sums, we first define blocks, cutvertices and the block-cut tree \cite{diestel2025graph}.

\begin{definition}[Cutvertices, blocks, and the block--cut tree]
Let $G=(V,E,K,s,t)$ be connected. A cutvertex is a vertex whose removal disconnects $G$. A block is a maximal connected subgraph without a cutvertex (a maximal $2$-connected subgraph, or a single bridge). The blocks partition the edge set $E$, and a vertex is a cutvertex precisely when it is shared by more than one block. The blocks and cutvertices of $G$ form a tree, the block--cut tree $T(G)$, with a block $B$ joined to a cutvertex $c$ whenever $c\in B$.
\end{definition}

Based on translational invariance of the lattice model, we have established in \cref{lem:choice_of_x} (2) that graph lattice sums do not depend on the choice of the pinning vertex $p$, whose position is removed from summation. 
Altering the position of the pin provides us with the possibility of factorize a single graph lattice sum into a product of two sums on the respective blocks.

\begin{lemma}[Cutvertex splitting]
Let the graph $G=(V,E,K,s,t)$ be given by two subgraphs $G_1=(V_1,E_1,K_1,s_1,t_2)$ and $G_2=(V_2,E_2,K_2,s_2,t_2)$ connected at a single cut-vertex $p\in V$, with $K_1$ and $K_2$ the restriction of $K$ onto the reduced edge sets. The graph lattice sum then factorizes over the two subgraphs,
\[
\mathcal Z_G(\bm k)=\mathcal Z_{G_1}(\bm k)\,\mathcal Z_{G_2}(\bm k).
\]
For $s\in V_1$ and $p$ in $V_2$, we route the terminals via the cutvertex, $(s_1,t_1)=(s,p)$ and $(s_2,t_2)=(p,t)$. If, without restriction of generality, both $s,t\in V_1$ then $G_1$ keeps these terminals and $G_2$ is an equal momentum-independent attachment, i.e.,
 \[
\mathcal Z_G(\bm k)=\mathcal Z_{G_1}(\bm k)\,\mathcal Z_{G_2}(\bm 0).
\]
\end{lemma}
\begin{proof}
    By \cref{lem:choice_of_x}, set $p$ as the pinning vertex und fix its position to the origin, $\bm x_p=0$. 
    As $E_1$ and $E_2$ are disjoint, the product over edges factorizes over the subgraphs
    \[
    \prod_{e\in E} {K}_{e}\big(\delta\bm x_e\big)= 
    \bigg(\prod_{e\in E_1} {K}_{e}\big(\delta\bm x_e)\bigg)\bigg(\prod_{e\in E_2}{K}_{e}\big(\delta\bm x_e)\bigg)
    \]
    Define the reduced node sets $\tilde V=V\setminus \{p\}$, $\tilde V_1=V_1\setminus\{p\}$, and $\tilde V_2=V_2\setminus \{p\}$. As $p$ is a cutvertex, $\tilde V_1\cap\tilde V_2 =\varnothing$, and 
    \[
    \sum_{\{\bm x_{v}\in\Lambda\}_{v\in \tilde V}}=\sum_{\{\bm x_{v}\in\Lambda\}_{v\in \tilde V_1}} \sum_{\{\bm x_{v}\in\Lambda\}_{v\in \tilde V_2}}.
    \]
    For $s\in V_1$ and $t\in V_2$, rewrite the exponential as 
    \[
    e^{-2\pi i (\bm x_t-\bm x_s)\cdot \bm k} = e^{-2\pi i (\bm x_t-\bm x_p)\cdot \bm k}e^{-2\pi i (\bm x_p-\bm x_s)\cdot \bm k}.
    \]
    The only joint node position appearing in both edge products is $\bm x_{e^\pm}=\bm x_p$, set to zero by translational invariance. The sums over $v\in V_1$ can be therefore be moved inside the products over $e\in E_1$ and the sum over $v\in V_2$ inside the product over $e\in E_2$ including the corresponding exponential factors. The exchange of orders is here allowed by Fubini, see \cref{lem:choice_of_x} (1). If both terminals are associated with the same subgraph, choose $s_1=s$ and $t_2=t$ and choose arbitrary but equal terminals for $G_2$, which do not influence its value as the momentum dependency vanishes.
\end{proof}

The lemma above lets us decompose graph lattice sums at cutvertices. Applied recursively, it yields the block factorisation of graph lattice sums, for which we first define the spine and the attachments.

\begin{definition}[Spine and attachments]
Let $G=(V,E,K,s,t)$ be connected, with a set of blocks $\mathcal B(G)$ and a block--cut tree $T(G)$. The blocks on the unique path in $T(G)$ between a block containing $s$ and a block containing $t$, taken in order $B_1,\dots,B_m$, form the spine $\mathcal S(G)=\{B_1,\dots,B_m\}$. Consecutive spine blocks meet at the cutvertices $p_i=V(B_i)\cap V(B_{i+1})$ with $p_0=s$ and $p_m=t$, the spine block $B_i$ is assigned the terminals $(p_{i-1},p_i)$, routing the external momentum $s=p_0\to p_1\to\cdots\to p_m=t$. The remaining blocks form the attachments $\mathcal A(G)=\mathcal B(G)\setminus\mathcal S(G)$, so that $\mathcal B(G)=\mathcal S(G)\sqcup\mathcal A(G)$. For an attachment $B$, we write $\mathcal Z_B(\bm0)$ for the graph lattice sum of $B$ at $\bm k=\bm0$, which is independent of any choice of terminals.
\end{definition}

\begin{theorem}[Block factorisation]
    \label{thm:block_factorization}
Let $G=(V,E,K,s,t)$ be connected, with spine $\mathcal S(G)$ and attachments $\mathcal A(G)$ as in the preceding definition. Then, for $\bm k\in\mathrm{BZ}$,
\[
\mathcal Z_G(\bm k)
=\prod_{B\in\mathcal S(G)}\mathcal Z_B(\bm k)\,
 \prod_{B\in\mathcal A(G)}\mathcal Z_B(\bm0).
\]
Only the spine factors depend on $\bm k$. In particular, when $s=t$ every factor can be evaluated at $\bm k=\bm0$ and
\[
\mathcal Z_G(\bm k)=\prod_{B\in\mathcal B(G)}\mathcal Z_B(\bm0)=\mathcal Z_G(\bm 0).
\]
\end{theorem}

\begin{proof}
We use induction on the number of blocks $|\mathcal B(G)|$. If $G$ consists of a single block, then the statement holds trivially. Otherwise the tree $T(G)$ has a leaf block $B$, meeting the rest of $G$ at a single cutvertex $p$. Write $G=G'\cup_p B$. The cutvertex lemma at $p$ then gives
\[
{\mathcal Z}_G(\bm k)={\mathcal Z}_{G'}(\bm k)\,{\mathcal Z}_B ,
\]
where the split-off factor ${\mathcal Z}_B$ equals ${\mathcal Z}_B(\bm k)$, with terminals $(p_{m-1},t)$, if $B$ is the spine endpoint containing $t$ (symmetrically for $s$), and equals ${\mathcal Z}_B(\bm0)$ if $B$ is an attachment. The reduced graph $G'$ has one fewer block, with $\mathcal S(G')=\mathcal S(G)\setminus\{B\}$ and $\mathcal A(G')=\mathcal A(G)\setminus\{B\}$, and inherits the terminals $s$ and, if $B=B_m$ was peeled, $p_{m-1}$ in place of $t$. By induction, we obtain the desired factorization.
\end{proof}

After having established that graph lattice sums can be factorized into blocks, we briefly comment on how articulation vertices are discovered and how the block-cut tree is constructed. Given a connected multigraph $G=(V,E,K,s,t)$, its articulation vertices and the biconnected blocks can be found in linear time and storage using the classical depth-first-search algorithm by Hopcroft and Tarjan \cite{Tarjan1972,HopcroftTarjan1973}, implemented for instance in NetworkX. Note here that, as vertex biconnectivity does not depend on edge multiplicities, the decomposition of the multigraph corresponds to that of the underlying simple graph. The block-cut tree is then assembled by connecting each block to the articulation vertices it contains. The terminals are mapped to their tree nodes and the unique tree path between them, forming the spine, is extracted by breadth-first search in linear time. The complete combinatorial layer for the block-cut decomposition is linear in time and space.

A direct consequence of the block-factorization theorem is that any tree graph can be written in terms of the lattice Fourier transform of the interaction kernel. For the particular case of power-law interactions, this reduces to the Epstein zeta function. The following section discusses these and other elementary blocks.

\section{Evaluation of basic blocks}
\label{sec:basic_blocks}

In this section, we discuss the evaluation of graph lattice sums and graph zeta functions for the two most basic, and, in practice, most often occuring blocks, that arise after block factorization. These are
\begin{enumerate}
    \item bridges (two nodes connected by a finite number of edges),
    \item cycles without terminals.
\end{enumerate}
For typical quantum lattice models, such as the LRTFIM, these blocks cover the large majority of all blocks. In addition, for power-law interactions, they belong to the hardest objects to compute with direct methods, as the error under lattice discretization falls off arbitrarily slowly, as the interaction exponent approaches the system dimension. Fortunately, these objects can be computed to full precision with the procedures shown below. From the block-factorization \cref{thm:block_factorization} then follows that any graph, whose blocks are of this form, can be computed with linear numerical complexity.

A central object is the lattice Fourier transform of the interaction kernel. We begin by defining the standard Fourier transform and the lattice Fourier transform.
\begin{definition}[Fourier transform and lattice Fourier transform]
    Let $f\in L^1(\mathds R^d)$. We then define its Fourier transform as 
    \[
    (\mathcal F f)(\bm \xi) = \hat f(\bm \xi) = \int_{\mathds R^d} e^{-2\pi i \bm x\cdot \bm \xi} f(\bm x)\,\mathrm d \bm x,\quad \bm \xi \in \mathds R^d.
    \]
    For a Bravais lattice $\Lambda\subseteq \mathds R^d$, consider a function $f\in \ell^{1}(\Lambda)$. Its lattice Fourier transform is defined as 
    \[
    (\mathcal F_\Lambda f)(\bm k)=\hat f(\bm k)=\sum_{\bm x\in \Lambda} e^{-2\pi i \bm x\cdot \bm k} f(\bm x),\quad \bm k\in \mathrm{BZ}.
    \]
    By duality, both the Fourier transform and the lattice Fourier transform extend to tempered distributions \cite{hoermander1966introduction}.
\end{definition}
If the kernel $K$ has compact support or decays superalgebraically, such that it can be sufficiently well approximated on a truncated grid, then its lattice Fourier $\hat K$ transform is analytic on the Brillouin zone and grid values can be extracted with log-linear complexity through use of the Fast Fourier Transforms (FFT).

The situation, however, becomes significantly more complicated, if the kernel includes tails that only decay at an algebraic rate, i.e., for graph zeta functions.
The prototypical case here is the regularized power law kernel $\mathcal K_\nu$ with exponent $\nu\in \mathds C$ appearing in quantum spin models with power-law interactions such as the LRTFIM.  The resulting lattice Fourier transform is called the Epstein zeta function \cite{epstein1903theorieI,epstein1903theorieII}.
Initially introduced by Paul Epstein in 1903, it forms the natural generalization of the Riemann zeta function to oscillatory lattice sums in higher dimensions  while preserving a functional equation. We have recently discussed the analytic properties and the efficient computation of the Epstein zeta function in \cite{buchheit2024epstein}, together with an implementation in the publicly available high-performance C library \href{https://github.com/epsteinlib/epsteinlib/epsteinlib}{EpsteinLib}.
\begin{definition}[Epstein zeta function]
\label{def:epstein}
    Let $\Lambda$ be a $d$ Bravais lattice, $\bm k\in \mathds R^d$, and $\nu\in \mathds C$. The Epstein zeta function is then defined as
    \[
    Z_{\Lambda,\nu}(\bm k) = \sum_{\bm x\in \Lambda} e^{-2\pi i \bm x\cdot \bm k}\,\mathcal K_\nu(\bm x),\quad \mathrm{Re}(\nu)>d.
    \]
    The Epstein zeta function can be meromorphically continued to $\nu\in \mathds C$.
\end{definition}
The Epstein zeta function is the prototypical lattice Fourier transform of a homogeneous kernel. When additionally including higher-order derivatives in $\bm k$, as discussed in \cite{buchheit2026zeta}, general anisotropies can be described, allowing for the evaluation of lattice Fourier transforms of arbitrary homogeneous kernels. Lattice Fourier transforms of general interaction kernels can, therefore, be described by Epstein zeta functions, potentially including derivatives, describing long-range tails, plus a rapidly decaying Fourier series. 

It is important to note that the Epstein zeta function exhibits power-law singularities at the reciprocal lattice points $\Lambda^\ast = A^{-T}\mathds Z^d$, that require special attention. The following Lemma discusses the holomorphic extension of the function in both arguments, it's singularity structure, and the removal of singularities. It is a combined version of the results from \cite[Theorems 2.14, 2.18]{buchheit2024epstein}.
\begin{lemma}[Holomorphy and singularity structure of the Epstein zeta function]
    \label{lem:holomorphy}
    Define the set $D_L\subseteq \mathds C^d$ as the  following intersection of $d$-dimensional complex cones with origins at $L\subseteq \mathds R^d$,
$$
D_L=\{\bm u\in\mathds C^d:|\Re{\bm u}-\bm z|>|\Im{\bm u}|\ \forall \bm z\in L\}.
$$
\begin{enumerate} \item The Epstein zeta function can be holomorphically extended to  
$$
(\nu,\bm k)\in 
\mathds C\times D_{\Lambda^*}.
$$
For $\bm k\in\Lambda^*$, the Epstein zeta function is holomorphic in $\nu\in\mathds C\setminus\{d\}$ with a simple pole in $\nu=d$.  

\item The singularity of the Epstein zeta function at $\bm k=\bm 0$ can be isolated as
    \[
    Z_{\Lambda,\nu}(\bm k) = Z_{\Lambda,\nu}^{\mathrm{reg}}(\bm k)+\hat s_\nu(\bm k)/V_\Lambda,
    \]
    where $Z_{\Lambda,\nu}^{\mathrm{reg}}$ is analytic on $(\mathds R^d\setminus \Lambda^\ast)\cup \{\bm 0\}$ and where $\hat s_\nu(\bm k)$ is the distributional Fourier transform of $s_\nu(\bm x)=\vert \bm x\vert^{-\nu}$, which reads for $\nu\in \mathds C\setminus (2\mathds N +d)$,
    \[
    \hat s_\nu(\bm k) = c_\nu \vert \bm k\vert^{\nu-d},\qquad c_\nu=\pi^{\nu-d/2}\frac{ \Gamma ((d-\nu)/2)}{\Gamma(\nu/2) }.
    \]
    In case that $\nu=d+2n$, $n\in \mathds N$, the Fourier transform is uniquely defined up to a polynomial of order $2n$. We adopt the choice,
\[\hat s_{d+2n}(\bm y)= \frac{\pi^{n+d/2}}{\Gamma(n+d/2)}\frac{(-1)^{n+1}}{n!} ( \pi \bm k^2 )^{n} \log (\pi  \bm k^{2}).
    \]
\item For $\nu\in d+2 \mathds N_0$, the regularised Epstein zeta function can be holomorphically extended to  $\bm k \in D_{\Lambda^\ast\setminus\{\bm 0\}}$,
and, outside these particular $\nu$ values, the funtion extends holomorphically to
\[
(\nu,\bm k) \in \big(\mathds C\setminus (d+2\mathds N_0)\big)\times D_{\Lambda^\ast\setminus\{\bm 0\}}.
\]
\end{enumerate}
\end{lemma}  

Note that an additional lattice shift vector can be included in the definition of the Epstein zeta function to account for lattices with multiple atoms per unit cell, see for instance \cite{buchheit2024epstein}. 

Given an analytically known lattice Fourier transform of the interaction kernel, such as for graph zeta functions, bridge blocks are available as analytic objects, computable via FFT for compactly supported kernels $a(\bm x)$, accompanied by EpsteinLib evaluation for the Epstein zeta functions for the power-law contributions $\mathcal K_\nu$.
\begin{lemma}[Analytic block I: Bridge]
    \label{lem:bridge}
    Let $G=(\{1,2\},E,K,1,2)$ with $V=\{1,2\}$ and $E$ a non-zero set of edges connecting $1$ and $2$. The graph lattice sum then reads
    \[
    \mathcal Z_G(\bm k) = \mathcal F_{\Lambda} \bigg(\prod_{e\in E} \tilde K_e\bigg)(\bm k),
    \]
    where $\tilde K_e=K_e$ if $(e^-,e^+)=(1,2)$ and $\tilde K_e=K_e(-\bm \cdot)$ for the reverse orientation $(e^-,e^+)=(2,1)$. 
    In the case of graph zeta functions with single power-law kernels $K_e=\mathcal K_{\nu_e}$, one obtains
    \[
    \mathcal \zeta_G(\bm k)=Z_{\Lambda,\mu}(\bm k), \quad \mu=\sum_{e\in E}\nu_e.
    \]
    The case of equal terminals $s=t$ (no momentum dependency) is equivalently obtained by setting $\bm k=\bm 0$ above.
\end{lemma}
\begin{proof}
    Reduce the multi-graph to a simple graph by \cref{lem:merge} on edge merging and set $p=s=1$ and $\bm x_p=\bm 0$ by \cref{lem:choice_of_x}. The reduced kernel is then $\prod_{e\in E} \tilde K_e$ and the graph lattice sum corresponds by definition to its lattice Fourier transform. The equivalent result for the power-law kernel $\mathcal K_{\nu}$ is obtained after noticing that
    \[
    \mathcal K_{\nu_1}\mathcal K_{\nu_2} = \mathcal K_{\nu_1+\nu_2},
    \]
    where edge orientation is irrelevant, yielding the result.
\end{proof}
This result readily generalizes to general graph zeta functions, as shown in \ref{thm:convolution}.

The Epstein zeta function has recently been used by some of the authors to efficiently compute many-body lattice sums   \cite{buchheit2025epstein_method}. The resulting method currently serves as the foundation for an ongoing investigation into the influence of many-body interactions on the stability of crystal lattices in theoretical chemistry, with first results in \cite{robles2025exact}. 

In \cite[Theorem~2.6]{buchheit2025epstein_method}, we have shown that graph zeta function for circle graphs without momentum (equal terminals) can be computed to full precision at linear numerical cost with $\mathcal O(|V|-1)$. This representation directly generalizes to general summable interaction kernels and provides us with the second analytically available block.
\begin{corollary}[Analytic block II: Cycle without momentum]
    \label{cor:analytic_block_II}
Let $G=(V,E,K)$ be a simple circle graph with all edges oriented either clockwise or anticlockwise. The graph lattice sum then admits the representation
\[
{\mathcal Z}_G= V_\Lambda \int_{\mathrm{BZ}}\prod_{e\in E} {\hat { K}}_e(\bm k)\,\mathrm d \bm k.
\]
\end{corollary}
The proof for the extension of the above theorem for general kernel functions, required for arbitrary graph zeta functions, follows in complete analogy to the power-law case in \cite[Theorem~2.6]{buchheit2025epstein_method}. Of course, the theorem generalizes to arbitrary edge orientations by replacing kernels $K_e$ by $K_e(-\cdot)$, if required.

For graph zeta functions, where the lattice Fourier transform is singular at $\bm k=\bm 0$, this integral can be computed efficiently by separating the integration domain into pyramids and then performing a Duffy transform on each subdomain \cite{duffy1982quadrature}. For compactly supported interaction kernels, a simple Gauss quadrature is sufficient as the lattice Fourier transform is analytic.

\begin{table}[t]
\centering
{\setlength{\tabcolsep}{4pt}\scriptsize
\begin{tabular}{@{}r rrrr rrrr@{}}
\toprule
& \multicolumn{4}{c}{$0$qp} & \multicolumn{4}{c}{$1$qp (finite $\bm k$)}\\
\cmidrule(lr){2-5}\cmidrule(lr){6-9}
$r$ & graphs & blocks & unique & reuse\,[\%]
    & graphs & blocks & unique & reuse\,[\%]\\
\midrule
 1 & ---     & ---      & ---     & ---  & 1       & 1        & 1      & 0.0 \\
 2 & 1       & 1        & 1       & 0.0  & 2       & 3        & 2      & 33.3\\
 3 & 1       & 1        & 1       & 0.0  & 4       & 7        & 4      & 42.9\\
 4 & 3       & 4        & 3       & 25.0 & 10      & 21       & 7      & 66.7\\
 5 & 3       & 4        & 4       & 0.0  & 24      & 55       & 13     & 76.4\\
 6 & 10      & 17       & 9       & 47.1 & 65      & 166      & 25     & 84.9\\
 7 & 15      & 26       & 13      & 50.0 & 176     & 476      & 55     & 88.4\\
 8 & 44      & 86       & 29      & 66.3 & 513     & 1\,479   & 132    & 91.1\\
 9 & 89      & 181      & 50      & 72.4 & 1\,535  & 4\,571   & 365    & 92.0\\
10 & 254     & 548      & 123     & 77.6 & 4\,806  & 14\,775  & 1\,068 & 92.8\\
11 & 633     & 1\,406   & 267     & 81.0 & 15\,541 & 48\,446  & 3\,423 & 92.9\\
12 & 1\,900  & 4\,333   & 752     & 82.6 & ---     & ---      & ---    & --- \\
13 & 5\,450  & 12\,600  & 2\,005  & 84.1 & ---     & ---      & ---    & --- \\
\midrule
total & 8\,403 & 19\,207 & 2\,602 & 86.5 & 22\,677 & 70\,000 & 3\,423 & 95.1\\
\bottomrule
\end{tabular}\vspace*{.1cm}}
\caption{Block census of the perturbative corpora for the long-range transverse-field Ising model. At each order $r$, the corpus graphs decompose into biconnected blocks, of which the unique column counts those distinct up to isomorphism, retaining edge multiplicities. The reuse column describes the percentage of block evaluations served from cache. The $1$qp corpus is evaluated at finite $\bm k$, where spine blocks are additionally distinguished by their terminal roles and momentum. The total row uses the globally distinct block count.}
\label{tab:block-census}
\end{table}

Before moving on to the evaluation of more general blocks, it is useful to briefly evaluate the statistics of blocks appearing in the 0qp and 1qp corpora for the LRTFIM, shown to low-order in \cref{fig:tfim_o7}, which will serve as the primary benchmarks for this work. We will focus on state-of-the-art studies for the LRFTIM using order $r=13$ pCUT expansions for the 0qp and order $r=11$ expansions for the 1qp sector. The related graph statistics are summarized in \cref{tab:block-census}. For 0qp, all graphs appearing at order $r$ are at most $(r-1)d$ dimensional lattice sums. Up to order 13, the 0qp corpus consists of $8\,403$ graphs (all without terminals, hence no $\bm k$ dependency) decomposable into $19\,207$ biconnected blocks. Of all blocks, $7\,902$ ($41.1\,\%$) are simple bridges (analytic block type I, computable via \cref{lem:bridge}), while $5\,611$ blocks ($29.2\,\%$) are cycles (analytic block type II, computable via \cref{cor:analytic_block_II}). Therefore, the analytic objects discussed in this section already account for $70.4\,\%$ of all occuring blocks in the 0qp corpus. 

The 1qp corpus till order 11 consists of $22\,677$ graphs, with summation dimension at order $r$ at most $rd$, of which only $4\,898$ have equal terminals, everything else depending on momentum $\bm k$. Of the total number of $70\,000$ blocks, $44\,439$ ($63.5\,\%$) are bridges and $14\,273$ ($20.4\,\%$) are simple cycles, leading to a total of $83.9\,\%$ of analytically accessible blocks. Also note that these analytic blocks and graphs built from them are computable to full precision at linear cost in the number of nodes $|V|$, in contrast to exponential scaling for direct summation approaches. In addition, they belong to the set of least connected graphs, which converge slowest under direct summation approaches. The method provided in this section solves this issue.

Block-factorization has an important second consequence, besides the exponential reduction in computational cost for many individual graphs, such as trees. When evaluating not just single graphs, but a whole corpus from \cref{def:corpus}, necessary for extracting physical quantities, the number of rendundant graphs is very small, yet the number of recurring blocks is high. Therefore, we apply block-caching, meaning that all blocks across the corpus are evaluated only once and then reused. This compresses the $19\,207$ occurrences of the 0qp corpus to $2\,574$ distinct blocks and the $70\,000$ occurrences of the 1qp corpus to $2\,868$, corresponding to a reuse factor of $7.5$ and $24.4$, respectively. The resulting reuse is highly non-uniform. A bridge is characterised by its combined interaction kernel alone, so the $7\,902$ bridges of the 0qp corpus reduce to $6$ distinct objects and the $44\,439$ bridges of the 1qp corpus to $11$. Meanwhile, the $5\,611$ and $14\,273$ cycles reduce to $115$ and $211$. Therefore, after caching, the analytic block types therefore account for only $4.7\,\%$ and $7.7\,\%$ of the blocks that need to be computed. The next section focusses on the category of graphs that appears dominantly after block caching.

\section{Construction of series-parallel blocks}
\label{sec:sp-construction}

In the previous sections, we have discussed the decomposition of graph lattice sums into blocks and the subsequent evaluation of the most basic, and most often occuring blocks. In this section, we show that, using the lattice Fourier transform of the kernel, or the Epstein zeta function in the power-law case, as the base object, the most-important  and most often occuring remaining blocks can be built. It turns out that this task can be achieved from two binary operations, pairwise multiplication and convolution of objects. To this end, we first define the periodic convolution on the torus.

\begin{definition}
    Let $f,g:\mathrm{BZ}\to \mathds C$ be integrable and periodic. We then define the periodic convolution on the reciprocal unit cell as
    \[
    (f\ast g)(\bm k) = V_\Lambda \int_{\mathrm{BZ}} f(\bm k-\bm p) g(\bm p)\,\mathrm d \bm p.
    \]
\end{definition}

We then notice that the convolution of two lattice Fourier transform of kernels in the wavevector argument corresponds to pointwise multiplication of the kernels.
\begin{corollary}
\label{cor:convolution}
Let $K_1,K_2\in\ell^1(\Lambda)$. Then
\[
\hat K_1\ast \hat K_2 = \mathcal F_\Lambda (K_1 K_2).
\]
    In particular, for $\nu_1,\nu_2\in \mathds C$ with $\mathrm{Re}(\nu_i)>d$, $i=1,2$, we have
    \begin{align*}
    Z_{\Lambda,\nu_1}\ast Z_{\Lambda,\nu_2} &= Z_{\Lambda,\nu_1+\nu_2},
    \end{align*}
    which can be meromorphically extended to $\nu_i\in \mathds C$.
    Furthermore, with $1$ the constant function, we have for $\nu\in \mathds C$,
    \[
        1\ast Z_{\Lambda,\nu}=0.
    \]
\end{corollary}
\begin{proof}
    Both sums in the convolution converge absolutely and we can exchange the order of summation and integration, yielding
    \[
    (\hat K_1\ast \hat K_2)(\bm k)=\sum_{\bm x_1,\bm x_2\in \Lambda}K_1( \bm x_1)K_2(\bm x_2) e^{-2\pi i \bm k\cdot \bm x_1} V_\Lambda \int_{\mathrm{BZ}} e^{-2\pi i \bm k\cdot(\bm x_2-\bm x_1)}\,\mathrm d \bm k.
    \]
    The first result then follows after noting that for lattice vectors $\bm x_1,\bm x_2\in \Lambda$,
    \[
    V_\Lambda \int_{\mathrm{BZ}} e^{-2\pi i \bm k\cdot(\bm x_2-\bm x_1)}\,\mathrm d \bm k = \delta_{\bm x_1,\bm x_2},
    \]
    with $\delta_{\bm x,\bm y}$ the Kronecker delta. The result for the Epstein zeta function follows as a particular case for the interaction kernel $K_j=\mathcal K_{\nu_j}$. Furthermore,
    \[
     1\ast Z_{\Lambda,\nu} = \sum_{\bm x\in \Lambda\setminus\{\bm 0\}} \frac{1}{\vert \bm x\vert^\nu} \delta_{\bm 0,\bm x} = \bm 0,
    \]
    as the origin is excluded from the sum due to the regularization of the kernel. Both results extend meromorphically to $\nu_i\in \mathds C$ by the identity theorem.
\end{proof}

We then show that the set of graph-lattice sums is closed under pointwise multiplication and convolution. In particular, starting from bridges as base objects, a large set of more complicated graphs can be obtained. We first begin by defining the two binary operations and then relate their action to the corresponding graph composition rules.

\begin{definition}[Graph operations]
\label{def:graph_definitions}
 Let ${\mathcal Z}_{G_1},\mathcal Z_{G_2}:\mathrm{BZ}\to\mathds C$
be two graph lattice sums on a Bravais lattice $\Lambda$.  
We define two binary operation and one unary operation, producing a new graph lattice sum
$\mathcal Z_G:\mathrm{BZ}\to\mathds C$.

\begin{enumerate}
    \item \emph{Pointwise multiplication:}
    \[
      \mathcal Z_G(\bm k)=\mathcal Z_{G_1}(\bm k)\,\mathcal Z_{G_2}(\bm k),
    \]
    \item \emph{Convolution:}
    \[
      \mathcal Z_G(\bm k)=(\mathcal Z_{G_1}\ast \mathcal Z_{G_2})(\bm k),
    \]
    \item \emph{Zero-momentum evaluation:}
    \[
      \mathcal Z_G(\bm k)\leftarrow \mathcal Z_G(\bm 0).
    \]
\end{enumerate}
\end{definition}
The following theorem related these binary operations directly to two well-known graph-theoretic composition rules, namely serial and parallel composition at the terminals, see \cite[Sec.~8.3]{Bodlaender1998} for a review.
\begin{theorem}[Graph composition from graph operations]
\label{thm:zeta_vs_graph_ops_weighted}
Consider two graph lattice sums $\mathcal Z_{G_1}$ and $\mathcal Z_{G_2}$ 
with connected two--terminal weighted multigraphs
\[
G_i=(V_i,E_i,K_i,s_i,t_i),\quad i=1,2,
\]
such that, after relabeling, $V_1\cap V_2=\varnothing$.
 The operations of Definition~\ref{def:graph_definitions}
then realize the standard two--terminal graph compositions as follows.

\begin{enumerate}
\item[\textup{(i)}] \textup{(Series composition $\Leftrightarrow$ pointwise multiplication)}
Form the series composition $G=G_1\circ G_2$ by identifying $t_1$ with
$s_2$ and declaring $(s_1,t_2)$ as terminals,  taking the disjoint union of
edge multisets with
exponents inherited from $G_1$ and $G_2$. Then,
\[
\mathcal Z_{G}(\bm k)=\mathcal Z_{G_1}(\bm k) \mathcal Z_{G_2}(\bm k).
\]
\item[\textup{(ii)}] \textup{(Parallel composition $\Leftrightarrow$ convolution)}
Form the parallel composition $G=G_1\parallel G_2$ by identifying the respective terminal $s_1=s_2$ and $t_1=t_2$,
 again taking the disjoint union of edge multisets with inherited exponents. Then, 
\[
\mathcal Z_{G}(\bm k)=(\mathcal Z_{G_1}\ast\mathcal Z_{G_2})(\bm k).
\]
\item[\textup{(iii)}] \textup{($1$--sum attachment $\Leftrightarrow$
multiplication by a zero--momentum evaluation)}
First forget the terminals of $G_2$, regarding it as an unpointed weighted
multigraph.  Fix a vertex $v\in V(G_1)$ and a vertex $w\in V(G_2)$, and form
$G$ as the $1$--sum obtained by identifying $v$ with $w$.  
The terminals of $G$ are thus those of $G_1$.
Then, 
\[
\mathcal Z_G(\bm k)=\mathcal Z_{G_1}(\bm k)\mathcal Z_{G_2}(\bm 0).
\]
\end{enumerate}
\end{theorem}
\begin{proof}
    (i) We start from the defining lattice sums in Definition~\ref{def:graph_definitions}. 
By Lemma~\ref{lem:choice_of_x}, we may fix an arbitrary vertex position in each graph without changing the value of the sum. 
For $G_1$ we fix $t_1$ and for $G_2$ we fix $s_2$, setting
$\bm x_{t_1}=\bm 0$ and $\bm x_{s_2}=\bm 0$.
This identifies the vertices $t_1$ and $s_2$, thus the defining step of the series composition.

Forming the product of the two graph lattice sums and using absolute summability, the lattice sums may be combined. 
The exponential factors then satisfy
\[
e^{-2\pi i(\bm x_{t_1}-\bm x_{s_1})\cdot\bm k}
\,e^{-2\pi i(\bm x_{t_2}-\bm x_{s_2})\cdot\bm k}
\underset{t_1=s_2}{=}
e^{-2\pi i(\bm x_{t_2}-\bm x_{s_1})\cdot\bm k},
\]
so the terminals of the resulting graph are $(s_1,t_2)$.

Moreover, by construction of the series composition, each edge of $G$
belongs either to $E_1$ or to $E_2$ and connects only vertices of the corresponding graph.
After fixing the identified vertex $t_1=s_2$ to the origin,
the product of edge weights factorizes as
\[
\prod_{e\in E_1}
   {K}_{e}\big(\delta \bm x_{e}\big)
\prod_{e\in E_2}
   {K}_{e}\big(\delta\bm x_{e}\big)=
\prod_{e\in E(G)}
   {K}_{e}\big(\delta\bm x_{e}\big).
\]
Hence, the combined lattice sum coincides with the graph lattice sum of the series composition $G$.
    
    (ii) Using Lemma~\ref{lem:choice_of_x}, we omit the vertices $s_1$ and $s_2$ from summation and can thus set them equal with  $\bm x_{s_1}=\bm x_{s_2}=\bm 0$. When performing the convolution, notice that the only factors that carry a $\bm k$ dependence are the exponentials. After exchanging the convolution with the sum, which is possible by Fubini due to absolute convergence, we have
    \[
    e^{-2\pi i \bm x_{t_1}\cdot (\bm \cdot)} \ast e^{-2\pi i \bm x_{t_2}\cdot (\bm \cdot)} = e^{-2\pi i \bm x_{t_1}\cdot (\bm \cdot)}\delta_{\bm x_{t_1},\bm x_{t_2}},
    \]
    which sets the vertices $t_1$ and $t_2$ equal while achieving the correct form of the exponential for the graph $G$. Thus, the operation sets $s_1=s_2$ and $t_1=t_2$, which is equivalent to parallel composition.

    (iii) As before, use Lemma~\ref{lem:choice_of_x} to remove the vertex $v$ from $G_1$ and $w$ from $G_2$ from summation, setting $\bm x_{v}=\bm x_{w}=\bm 0$. We have thus identified $v$ by $w$, which corresponds to the 1-sum. Set $\bm k=0$ in $\mathcal Z_{G_2}$, thereby forgetting the terminals. The product of the graph lattice sums then directly yields the graph lattice sum for $G$, with the terminals identified as those of $G_1$. 
\end{proof}

The following example illustrates, using the composition rules from \cref{thm:zeta_vs_graph_ops_weighted}, how more densely connected graph lattice sums can be constructred using iterated multiplications and convolutions of bridges. Note, in particular, how the reordering of operations changes the arrangement of the terminals.

\begin{example}
    For bridges $G_1,\dots,G_5$ with even kernels $K_1,\dots,K_5$, the operation
    \[
    \mathcal Z_{G} =  \Big((\mathcal Z_{G_1}\mathcal Z_{G_2})\ast (\mathcal Z_{G_3}\mathcal Z_{G_4})\Big)\ast \mathcal Z_{G_5}
     \]
     yields the circle graph with $4$ nodes and an additional connecting line with kernel $K_5$ with terminals colored in orange as shown below in (a). 
Altering the order of convolutions can be used to map the terminals to different nodes. For instance,
 \[
    \mathcal Z_{G} =  \Big((\mathcal Z_{G_1}\mathcal Z_{G_2})\ast \mathcal Z_{G_5}\Big) \mathcal Z_{G_3}\Big)\ast \mathcal Z_{G_4},
\]
leads to the same graph but with $\bm k$ associated with the $K_4$ edge, see (b).

     \begin{center}
     \begin{tikzpicture}
  \tikzset{
    mynode/.style={circle, draw, minimum size=5pt, inner sep=0pt},
    terminal/.style={mynode, fill=orange},
  }

  \node[terminal] (a) at (0:1.2cm) {};
  \node[mynode]   (b) at (90:1.2cm) {};
  \node[terminal] (c) at (180:1.2cm) {};
  \node[mynode]   (d) at (270:1.2cm) {};

 \draw (a) -- node[above right]  {$K_2$} (b);
  \draw (b) -- node[above left]   {$K_1$} (c);
  \draw (c) -- node[below left]   {$K_3$} (d);
  \draw (d) -- node[below right]  {$K_4$} (a);
  \draw (a) -- node[above]        {$K_5$} (c);

  \node[anchor=north west, xshift=-20pt, yshift=2pt] at (current bounding box.north west) {(a)};
\end{tikzpicture}
\qquad
    \begin{tikzpicture}
  \tikzset{
    mynode/.style={circle, draw, minimum size=5pt, inner sep=0pt},
    terminal/.style={mynode, fill=orange},
  }

  \node[terminal] (a) at (0:1.2cm) {};
  \node[mynode]   (b) at (90:1.2cm) {};
  \node[mynode]   (c) at (180:1.2cm) {};
  \node[terminal] (d) at (270:1.2cm) {};

 \draw (a) -- node[above right]  {$K_2$} (b);
  \draw (b) -- node[above left]   {$K_1$} (c);
  \draw (c) -- node[below left]   {$K_3$} (d);
  \draw (d) -- node[below right]  {$K_4$} (a);
  \draw (a) -- node[above]        {$K_5$} (c);

    \node[anchor=north west, xshift=-20pt, yshift=2pt] at (current bounding box.north west) {(b)};

\end{tikzpicture}
\end{center}
The graph (a) occurs in the 1qp corpus of the LRTFIM at order $r=5$ with prefactor $-71/8$, see \cref{fig:tfim_o7} (b).
\end{example}

We will now classify the set of all graph lattice sums that can be generated from the operations in \cref{def:graph_definitions}. The correct graph-theoretic framework is given by the two-terminal series-parallel graphs, introduced by Duffin \cite{Duffin1965}, where we follow the notation from \cite[Sec.~8.3]{Bodlaender1998}.

\begin{definition}[Series-parallel graphs]
    \label{def:sp_graphs}
    A two-terminal multi-graph $(V,E,s,t)$ is called series-parallel if it can be formed by the following rules:
    \begin{enumerate}
        \item A graph with only two distinct nodes $s,t$ and an edge $(s,t)$ is series-parallel.
        \item Let $G_j=(V_j,E_j,s_j,t_j)$, $j\in\{1,2\}$ be two series-parallel graphs. Then, after taking the disjoint union of $G_1$ and $G_2$,
        \begin{enumerate}
            \item the series composition $G=G_1\circ G_2$ is obtained by identifying $s_2=t_1$ and setting $(s_1,t_2)$ as the new terminals,
            \item the parallel composition $G=G_1 \parallel G_2$ is obtained by identifying $s_1=s_2$ and $t_1=t_2$. 
        \end{enumerate}
        \item A graph without terminals, or equivalently with equal terminals $s=t$, is called series-parallel if it can be equipped with two distinct terminals such that the resulting two-terminal graph is series-parallel.
    \end{enumerate}
\end{definition}

The following theorem then identifies the set of graph lattice sums that can be generated from \cref{def:graph_definitions}. They are exactly those of the series-parallel two-terminal multigraphs. After block-factorization from \cref{thm:block_factorization}, this extends to any graph whose spine blocks are series-parallel with respect to their routed terminals and whose attachment blocks are series-parallel for some choice of terminals.

\begin{theorem}[Construction of series-parallel graph lattice sums]
\label{thm:tw2_zeta}
Let $G=(V,E,K,s,t)$ be a series-parallel two-terminal multigraph with kernels
$K_e\in\ell^1(\Lambda)$ for all $e\in E$. Its graph lattice sum
can be generated, starting from the graph lattice sums of its $\vert E\vert$
edges, viewed as bridges with terminals their endpoints, using exactly
\begin{enumerate}
    \item $\vert V\vert-2$ pairwise multiplications,
    \item $\vert E\vert-\vert V\vert+1$ convolutions,
\end{enumerate}
leading to a total of $\vert E\vert-1$ operations. If $s=t$, one additional zero-momentum evaluation is required.
\end{theorem}
\begin{proof}
    Assume $s\neq t$. From \cref{def:sp_graphs}, we know that the graph $G$ is obtainable from pairwise series and parallel compositions of its edges, viewed as bridges with endpoint terminals. We may thus record the resulting decomposition as a binary tree whose $\vert E\vert$ leaves are the basic bridges, one for each $e\in E$, and whose inner nodes define the compositions (series or parallel). Each inner node is therefore associated with a two-terminal multigraph assembled from the edges of its leaves, and the root yields $G$ itself.
    
    We construct the graph lattice sum from the leaves upwards, where the graph lattice sum for each leaf is the lattice Fourier transform of the associated kernel $\hat K_e$, where the edge is oriented as required by the composition and $K_e$ is replaced by $K_e(-\,\cdot\,)$ on a reversed edge. Series composition of graphs is then identified with multiplication and parallel composition of graphs by convolution of graph lattice sums as shown in \cref{thm:zeta_vs_graph_ops_weighted}. Summability of $K_e$ results in continuity and boundedness of $\hat K_e(\bm k)$ on $\mathrm{BZ}$, implying integrability. As these properties are preserved under multiplication and convolution, all graph lattice sums occurring as sub-products are continuous and integrable, hence the convolutions are well-defined. We can therefore iterate the operations up to the root, with value $\mathcal Z_G$. The case of equal terminals corresponds to zero momentum evaluation. It is obtained by choosing an arbitrary pair of terminals that admits a series-parallel decomposition, which exists by \cref{def:sp_graphs} and evaluating at $\bm k=\bm 0$ as the final operation.

Regarding the counts, first notice that both compositions combine two graphs into a single one whose edge set is the disjoint union of the two edge sets. Starting from the $\vert E\vert$ single edges of $G$ and ending with $G$ itself, each composition reduces the number of objects by one, so that the total number of pairwise operations equals $\vert E\vert-1$. Let $n_{1}$ be the number of multiplications and $n_2$ the number of convolutions. Starting from a total of $2|E|$ nodes of the leaves, each multiplication identifies one pair of nodes while each convolution identifies two pairs of nodes. Hence
\[
\vert V\vert =2\vert E\vert -n_1 -2n_2=2\vert E\vert -2(n_1 +n_2)+n_1,
\]
and with $n_1+n_2=|E|-1$, we find $n_1=\vert V\vert -2$ and $n_2=\vert E\vert-1-n_1=\vert E\vert -\vert V\vert +1$. 
\end{proof}

The construction underlying the proof of the above theorem directly provides the foundation for an algorithm. Starting from the full graph, we repeatedly merge parallel edges into single edges, corresponding to convolution, and eliminate non-terminal nodes with 2 neighbors, corresponding to pointwise multiplication, until only the terminals with a single edge remain. Of the two operations, the convolutions dominate the numerical cost, whose efficient evaluation is discussed in the next section, while the combinatorial layer is again negligible. The proof shows that the reduction process never stalls, when the conditions of the theorem are fulfilled. Additionally, failure of the reduction for general blocks directly shows that block requires a different evaluator, which is discussed in \cref{sec:tensor}.

We finally want to point out that the class of series-parallel graphs can be brought into connection with a well-known classification concept for graphs called treewidth, originally introduced by Halin \cite{halin1976s} and independently rediscovered by Robertson and
Seymour \cite{RobertsonSeymour1986}. For a review, see \cite{CyganFKLMPPS15}. In the following, we follow the notation from \cite{Bodlaender1998}. 

\begin{definition}[Tree decomposition and treewidth]
    \label{def:tree_decomp}
Let $G=(V,E)$ be a finite simple graph. A tree decomposition of $G$ is a pair $\big(\{X_i\}_{i \in I}, T=(I,F)\big)$ with $\{X_i\}_{i\in I}$ a family of subsets of $V$, referred to as bags, one for each node of $T$, and $T$ a tree with edges $F$, such that the following conditions hold:
\begin{enumerate}
    \item Each node of $V$ is included in at least one of the bags: $\bigcup_{i\in I}X_i=V$.
    \item For all edges $e\in E$, there exists an $i\in I$, such that $\{e^+, e^-\}\subseteq X_i$.
    \item For all $i,j,k\in I$, if $j$ lies on a path between $i$ and $k$, then $X_i\cap X_k\subseteq X_j$.
\end{enumerate}
The width of a tree decomposition is $\max_{i\in I}(|X_i|-1)$ and the treewidth $\mathrm{tw}(G)$ of a graph is the minimum width of all possible tree decompositions. We identify the treewidth of a multi graph with the treewidth of the underlying simple graph, obtained by joining all parallel edges while disregarding edge orientations.
\end{definition}

The following corollary shows that a two-terminal block being series-parallel is equivalent to the block, augmented by an edge formed by the terminals, having $\mathrm{tw}\le 2$. 
\begin{corollary}[Equivalence of series-parallel and $\mathrm{tw}\le 2$  classification]
\label{cor:sp_tw_equivalence}
Let $G=(V,E,s,t)$ be a block. For $s\neq t$, define the augmented graph $\tilde G=G+(s,t)$, and $\tilde G=G$ otherwise. Then $(G,s,t)$ is series-parallel if and only if $\mathrm{tw}(\tilde G)\le2$ where for $s=t$ this is understood for some choice of terminals.
\end{corollary}
\begin{proof}
The graph $\tilde G$ is again a block, as a cutvertex $v$ of $\tilde G$ would disconnect $\tilde G-v$ and therefore also its spanning subgraph $G-v$. Its only biconnected component being $\tilde G$ itself, $\mathrm{tw}(\tilde G)\le2$ holds if and only if $\tilde G$ is series-parallel, by~\cite[Thm.~42]{Bodlaender1998}. In case that $s=t$, this is the assertion. For $s\neq t$, the latter is equivalent to $(G,s,t)$ being series-parallel by~\cite[Cor.~3.4]{RoyleSokal2015}.
\end{proof}

\begin{remark}
Whether a block is series-parallel depends on the position of its terminals,
and does not follow from $\mathrm{tw}(G)\le 2$ alone. The graph (a), shown below, with the
terminals marked in orange is
series-parallel, its graph lattice sum being constructed as indicated below.
The graph in (b) is the same block with the two non-adjacent nodes as
terminals. Its augmented graph is the complete graph $K_4$ shown in (c), so
that $\mathrm{tw}(\tilde G)=3$ and its graph lattice sum cannot be constructed
from the operations in \cref{def:graph_definitions}. The graph in (c) is
excluded for every choice of terminals, as already $\mathrm{tw}(G)=3$. Blocks
of this type are evaluated by the tensor-network method of \cref{sec:tensor}.
\end{remark}

\begin{center}
\makebox[\textwidth][c]{%
\begin{tikzpicture}[baseline={(0,0)}]
  \tikzset{
    mynode/.style={circle, draw, minimum size=5pt, inner sep=0pt},
    terminal/.style={mynode, fill=orange},
  }

  \node[terminal] (n0)   at (0:0.95cm) {};
  \node[mynode]   (n90)  at (90:0.95cm) {};
  \node[terminal] (n180) at (180:0.95cm) {};
  \node[mynode]   (n270) at (270:0.95cm) {};

  \draw (n0)   -- node[above right] {\small$K_2$} (n90);
  \draw (n90)  -- node[above left]  {\small$K_1$} (n180);
  \draw (n180) -- node[below left]  {\small$K_3$} (n270);
  \draw (n270) -- node[below right] {\small$K_4$} (n0);
  \draw (n0)   -- node[above]       {\small$K_5$} (n180);

  \node[anchor=north] at (0,-1.5)
    {\small$\mathcal Z_{G} = \Big((\mathcal Z_{G_1}\mathcal Z_{G_2})\ast
      (\mathcal Z_{G_3}\mathcal Z_{G_4})\Big)\ast \mathcal Z_{G_5}$};
  \node[anchor=north west] at (current bounding box.west |- 0,1.45) {(a)};
\end{tikzpicture}\hfill
\begin{tikzpicture}[baseline={(0,0)}]
  \tikzset{
    mynode/.style={circle, draw, minimum size=5pt, inner sep=0pt},
    terminal/.style={mynode, fill=orange},
  }

  \node[mynode]   (n0)   at (0:0.95cm) {};
  \node[terminal] (n90)  at (90:0.95cm) {};
  \node[mynode]   (n180) at (180:0.95cm) {};
  \node[terminal] (n270) at (270:0.95cm) {};

  \draw (n0)   -- (n90);
  \draw (n90)  -- (n180);
  \draw (n180) -- (n270);
  \draw (n270) -- (n0);
  \draw (n0)   -- (n180);

  \node[anchor=north] at (0,-1.5) {\small$\mathrm{tw}(\tilde G)=3$};
  \node[anchor=north west] at (current bounding box.west |- 0,1.45) {(b)};
\end{tikzpicture}\hfill
\begin{tikzpicture}[baseline={(0,0)}]
  \tikzset{mynode/.style={circle, draw, minimum size=5pt, inner sep=0pt}}

  \node[mynode] (n0)   at (0:0.95cm) {};
  \node[mynode] (n90)  at (90:0.95cm) {};
  \node[mynode] (n180) at (180:0.95cm) {};
  \node[mynode] (n270) at (270:0.95cm) {};

  \draw (n0)   -- (n90);
  \draw (n90)  -- (n180);
  \draw (n180) -- (n270);
  \draw (n270) -- (n0);
  \draw (n0)   -- (n180);
  \draw (n90)  -- (n270);

  \node[anchor=north] at (0,-1.5) {\small$\mathrm{tw}(G)=3$};
  \node[anchor=north west] at (current bounding box.west |- 0,1.45) {(c)};
\end{tikzpicture}}
\end{center}

In this section, we have established how general series-parallel blocks can be constructed based on multiplications and convolutions of kernel lattice Fourier transforms. The computation of the arising convolutions is however nontrivial, when the kernel function includes algebraic tails, as these lead to singularities at $\bm k=\bm 0$ for the lattice Fourier transform that need to be correctly captured across repeated convolutions and multiplications. An adaptive convolution strategy can succeed for a low number of iterated convolution operation, yet retaining full precision is challenging. In the following, we solve this issue by reformulating the two operation in an efficiently computable semi-analytic algebra that permits analytical evaluation of the convolutions at linear cost in the number of operations and log-linear cost in the number of discretization parameters.

\section{Semi-analytic algebra for graph products and convolutions}
\label{sec:algebra}

In the previous section, we have seen that from repeated convolutions and products of bridges general graph lattice sums for series-parallel blocks can be constructed. From \cref{thm:tw2_zeta}, we know that each series-parallel block can be created from precisely $\vert E\vert-1$ operations, starting from simple bridges. 
While the combinatorial procedure on how to combine the basic objects has been settled in the last section, the stable and efficient numerical evaluation of the repeated convolutions remains challenging.  
While simple Fourier-based approaches allow for stable evaluations of repeated operations in the case of compactly supported or rapidly decaying kernels $K_e$, where the Fourier series can be truncated, it is a nontrivial task to account for the power-law tails in the kernel of graph zeta functions in \cref{def:graph_zeta_functions}. 
These correspond to non-analyticities of the graph zeta function at $\bm k=\bm 0$, the simplest example being the $|\bm k|^{\nu-d}$ singularity of the Epstein zeta function $Z_{\Lambda,\nu}(\bm k)$, see \cref{lem:holomorphy}\,(2). Adaptive integration methods can be applied, yet scale in general quadratically, as FFT is no longer available. Also, avoiding loss of precision after multiple convolution operations is challenging. In this section, we demonstrate a method that allows for a numerically efficient and precise treatment of convolutions and multiplications for general kernels with power-law tails under repeated multiplications and convolutions, even for kernels that include tails with exponents close to the lattice dimension $d$. This is achieved by casting the leading non-analyticities of graph zeta functions in terms of Epstein zeta functions. The remaining regular part is written in terms of a rapidly decaying Fourier series. The resulting semi-analytical structure closes under convolution and multiplication, forming an algebra that  allows for a precise and fast evaluation at linear scaling in the number of operations and log-linear scaling in the number of discretization points. 

We begin by introducing a semi-analytic structure for kernels and graph lattice sums. Note that a slightly more general kernel compared to the one used for graph zeta functions in \cref{def:graph_zeta_functions} is necessary.
\begin{definition}[Semi-analytic kernels and graph lattice sums]
\label{def:semianalytic}
We call a kernel $K\in\ell^1(\Lambda)$ semi-analytic if it admits a
representation
\[
K(\bm x)=a(\bm x)+\sum_{j=1}^{N}b_j\,\mathcal K_{\nu_j}(\bm x),
\qquad N\in\mathds N_0,
\]
with $b_j,\nu_j\in\mathds C$, $\mathrm{Re}(\nu_j)>d$, and even coefficients $a(\bm x)\in\mathds C$ obeying the bound $|a(\bm x)|\le C|\bm x|^{-(d+2)}$ for all $\bm x\in\Lambda\setminus\{\bm 0\}$ and some $C>0$. We note that this representation is not unique and operations within this section will act on a given representation. Finally, a graph lattice sum $\mathcal Z_G$ is called semi-analytic, if it is the lattice Fourier transform of a semi-analytic kernel, namely
\[
\mathcal Z_G(\bm k)=\sum_{\bm x\in\Lambda}a(\bm x)e^{-2\pi i\bm x\cdot\bm k}
+\sum_{j=1}^{N}b_j\,Z_{\Lambda,\nu_j}(\bm k),
\]
with $Z_{\Lambda,\nu}$ the Epstein zeta function.
\end{definition}
Numerically, we can store the function as the triple $(a,\bm b,\bm\nu)$, with $\bm b,\bm\nu\in\mathds C^N$, where $a$ is truncated to the balanced cell
\[
\Lambda_n=A\,\{-\lceil n/2\rceil+1,\ \dots,\ \lfloor n/2\rfloor\}^d,
\]
which includes one representative of each site of the torus $\Lambda/(n\Lambda)$ from \cref{def:quantum-model} with $n^d$ points in total. The Fourier series of $a$ can then be evaluated on the momentum grid $\mathrm{BZ}_n=(\Lambda^\ast/n)/\Lambda^\ast$ through FFT.

This structure, of course, includes all simple regularized power-laws combined with compactly supported or rapidly-decaying kernels. In particular, it includes the kernels of graph zeta functions. Provided a graph lattice sum in the structure of \cref{def:semianalytic}, the resulting function can be efficiently evaluated using EpsteinLib for the power-law tails and FFT after truncating the rapidly decaying Fourier series, which is meaningful due to the decay condition on $a(\bm x)$.

We first notice that the semi-analytic structure allows to convolve two semi-analytic graph lattice sums, while retaining the exact singularities in $\bm k$. This rest on the basic, yet powerful, property that convolutions of Epstein zeta functions are Epstein zeta functions again with a shifted exponent. Meanwhile, convolutions of Epstein zeta functions with rapidly decaying Fourier series yield Fourier series with faster decay, while convolutions of the two short-ranged parts remain short-ranged.

\begin{theorem}[Convolution of semi-analytic graph lattice sums]
\label{thm:convolution}
Let $\mathcal Z_{G_1},\mathcal Z_{G_2}$ be semi-analytic as in \cref{def:semianalytic}, with representations $(a_\eta,\bm b_\eta,\bm\nu_\eta)$ and $\bm b_\eta,\bm\nu_\eta\in\mathds C^{N_\eta}$, $\eta\in\{1,2\}$, with $\mathrm{Re}(\nu_{i\eta})>d$ for all $i,\eta$. Let in addition $|a_\eta(\bm x)|\le C_\eta|\bm x|^{-\gamma_\eta}$ on $\Lambda\setminus\{\bm 0\}$ with $\gamma_\eta>d$ for $C_\eta>0$. Then $\mathcal Z_{G_1}\ast\mathcal Z_{G_2}$ is semi-analytic, with
\[
(\mathcal Z_{G_1}\ast\mathcal Z_{G_2})(\bm k)
=\sum_{\bm x\in\Lambda}a_3(\bm x)\,e^{-2\pi i\bm x\cdot\bm k}+\sum_{i=1}^{N_1}\sum_{j=1}^{N_2}b_{i1}b_{j2}\,
Z_{\Lambda,\nu_{i1}+\nu_{j2}}(\bm k),
\]
with Fourier coefficients
\[
a_3(\bm x)=a_1(\bm x)a_2(\bm x)
+a_1(\bm x)\sum_{j=1}^{N_2}b_{j2}\,\mathcal K_{\nu_{j2}}(\bm x)
+a_2(\bm x)\sum_{i=1}^{N_1}b_{i1}\,\mathcal K_{\nu_{i1}}(\bm x) .
\]
The exponents of the Epstein zeta terms then satisfy $\mathrm{Re}(\nu_{i1}+\nu_{j2})>2d>d$. The regular part $a_3$ is even and obeys $|a_3(\bm x)|\le C_3|\bm x|^{-\gamma_3}$ on $\Lambda\setminus\{\bm 0\}$ for some $C_3>0$ with
\[
\gamma_3=\min\big(\gamma_1+\gamma_2,\;\gamma_1+\min_j\mathrm{Re}(\nu_{j2}),\;
\gamma_2+\min_i\mathrm{Re}(\nu_{i1})\big),
\]
with minima over empty index sets taken as $\infty$. In particular, $\gamma_3\ge d+2$ if $\gamma_1,\gamma_2\ge d+2$, so that semi-analyticity from \cref{def:semianalytic} is preserved.
\end{theorem}
\begin{proof}
The convolution is the lattice Fourier transform of the pointwise product of the kernels,
$\mathcal Z_{G_1}\ast\mathcal Z_{G_2}=\mathcal F_\Lambda(K_1K_2)$. Inserting
$K_\eta=a_\eta+\sum_i b_{i\eta}\mathcal K_{\nu_{i\eta}}$ and using
$\mathcal K_{\nu_{i1}}\mathcal K_{\nu_{j2}}=\mathcal K_{\nu_{i1}+\nu_{j2}}$ results in
\[
K_1K_2=a_3+\sum_{i=1}^{N_1}\sum_{j=1}^{N_2}b_{i1}b_{j2}\,
\mathcal K_{\nu_{i1}+\nu_{j2}}
\]
yielding $a_3$ as stated, and $\mathcal F_\Lambda\mathcal K_\nu=Z_{\Lambda,\nu}$ by
\cref{def:epstein}. The real part of every new exponent is strictly larger than the real parts of the exponents involved. Hence, the bound on the exponent in \cref{def:semianalytic} is preserved.

For the decay of the Fourier series, bound the first summand of $a_3$ on $\Lambda\setminus\{\bm 0\}$ by
\[
|a_1(\bm x)a_2(\bm x)|\le C_1C_2|\bm x|^{-(\gamma_1+\gamma_2)},
\]
and the second and third as follows
\[
\bigg|a_\eta(\bm x)\sum_{i=1}^{N_{\mu}} b_{i\mu}\mathcal K_{\nu_{i\mu}}(\bm x)\bigg|
\le C_\eta\sum_{i=1}^{N_{\mu}}|b_{i\mu}|\,|\bm x|^{-\gamma_\eta-\mathrm{Re}(\nu_{i\mu})},
\quad \eta\neq\mu.
\]
This yields the stated bound on $\gamma_3$. Evenness of $a_3$ follows from evenness of $a_1$, $a_2$, and $\mathcal K_\nu$.
\end{proof}

While the closure of semi-analyticity under convolution follows directly from the pointwise product of the kernels, the situation under pointwise multiplication of graph lattice sums is more involved.  The following theorem shows that, indeed, semi-analyticity is also preserved under multiplication of graph lattice sums. The central idea is to analytically determine the leading order singularities of the product, casting them again as associated Epstein zeta functions, while transferring the remaining terms with faster decay into the Fourier series. 

\begin{theorem}[Multiplication of semi-analytic graph lattice sums]
\label{thm:multiplication}
Let $\mathcal Z_{G_1},\mathcal Z_{G_2}$ be semi-analytic as in \cref{def:semianalytic}, with representations $(a_\eta,\bm b_\eta,\bm\nu_\eta)$ and $\bm b_\eta,\bm\nu_\eta\in\mathds C^{N_\eta}$, $\eta\in\{1,2\}$, where $\mathrm{Re}(\nu_{i\eta})>d$ and $\nu_{i1},\nu_{j2}\notin d+2\mathds N_0$. Let the regular parts satisfy $|a_\eta(\bm x)|\le C_\eta|\bm x|^{-\gamma_\eta}$ on $\Lambda\setminus\{\bm 0\}$ with $\gamma_\eta>d$. Then $\mathcal Z_{G_1}\mathcal Z_{G_2}$ is semi-analytic, and
\[
\begin{aligned}
\mathcal Z_{G_1}(\bm k)\,\mathcal Z_{G_2}(\bm k)
&=\sum_{\bm x\in\Lambda}a_3(\bm x)\,e^{-2\pi i\bm x\cdot\bm k}+\sum_{i=1}^{N_1}\sum_{j=1}^{N_2} b_{i1}b_{j2}
\frac{c_{\nu_{i1}}\,c_{\nu_{j2}}}{V_\Lambda\,c_{\nu_{i1}+\nu_{j2}-d}}\,
Z_{\Lambda,\nu_{i1}+\nu_{j2}-d}(\bm k)\\
&\quad+\mathcal Z_{G_2}(\bm 0)\sum_{i=1}^{N_1}b_{i1}\,Z_{\Lambda,\nu_{i1}}(\bm k)
+\mathcal Z_{G_1}(\bm 0)\sum_{j=1}^{N_2}b_{j2}\,Z_{\Lambda,\nu_{j2}}(\bm k),
\end{aligned}
\]
where $1/c_{\nu_{i1}+\nu_{j2}-d}$ is set to $0$ when
$\nu_{i1}+\nu_{j2}-d\in d+2\mathds N_0$. The exponents of the Epstein terms again satisfy
$\mathrm{Re}(\nu_{i1}+\nu_{j2}-d)>d$. The regular part $a_3$ is even and obeys the bound
$|a_3(\bm x)|\le C_3|\bm x|^{-\gamma_3}$ on $\Lambda\setminus\{\bm 0\}$ for some
$C_3>0$, with
\[
\gamma_3=\min\Big(\gamma_1,\;\gamma_2,\;\min_{i}\mathrm{Re}(\nu_{i1})+2,\;\min_{j}\mathrm{Re}(\nu_{j2})+2\Big),
\]
with minima over empty index sets taken as $\infty$. In particular, $\gamma_3\ge d+2$ if $\gamma_1,\gamma_2\ge d+2$, so that semi-analyticity from \cref{def:semianalytic} is preserved.
\end{theorem}
\begin{proof}
We begin with a preparatory step, decomposing each Epstein zeta function in the two graph lattice sums into its regular part and a Riesz kernel,
$Z_{\Lambda,\nu}=Z^{\mathrm{reg}}_{\Lambda,\nu}+\hat s_\nu/V_\Lambda$ with
$\hat s_\nu(\bm k)=c_\nu|\bm k|^{\nu-d}$ the distributional Fourier transform of $\vert\bm \cdot\vert^{-\nu}$ on $\mathds R^d$, see \cref{lem:holomorphy}. Note that $Z^{\mathrm{reg}}_{\Lambda,\nu}$ is analytic in a complex
neighbourhood of $\bm k=\bm 0$ and $Z_{\Lambda,\nu}$ is analytic on
$\mathrm{BZ}\setminus\Lambda^\ast$. We can now separate each graph lattice sum $\mathcal Z_{G_\eta}$ into singular terms at $\bm k=\bm 0$ plus a remainder, 
\[
\mathcal Z_{G_\eta}=\frac{1}{V_\Lambda}S_\eta+R_\eta
\]
with
\[
S_\eta=\sum_{i=1}^{N_\eta} b_{i\eta}\,\hat s_{\nu_{i\eta}},
\qquad
A_\eta=\sum_{i=1}^{N_\eta} b_{i\eta}\,Z^{\mathrm{reg}}_{\Lambda,\nu_{i\eta}},
\qquad
R_\eta=A_\eta+\hat a_\eta,
\]
with the Fourier series $\hat a_\eta(\bm k)= (\mathcal F_\Lambda a_\eta)(\bm k)$.
Note that both $A_\eta$ and $\hat a_\eta$ are even, the former as $-\Lambda=\Lambda$, the latter by semi-analyticity. Therefore, also $R_\eta$ is even.

In the next step, fix a radially symmetric cutoff function $\chi\in C^\infty(\mathrm{BZ},\mathds R_{\ge0})$ with $\chi=1$ on $|\bm k|<r/2$ and $\chi=0$ on $|\bm k|>r$. Choose the radius $r>0$ small enough, such that both $r<\rho_{\Lambda^\ast}$ with $\rho_{\Lambda^\ast}=\min_{\bm q\in\Lambda^\ast\setminus\{\bm 0\}}|\bm q|$, and such that $\chi$ vanishes in an environment of $\partial\mathrm{BZ}$.  By \cref{lem:holomorphy}, $Z^{\mathrm{reg}}_{\Lambda,\nu_{i\eta}}$ is then analytic on a complex ball of radius
$(\rho_{\Lambda^\ast}-r)/\sqrt2$ about every point of $\mathrm{supp}\,\chi$. Split the product of the two graph lattice sums into $\mathcal Z_{G_1} \mathcal Z_{G_2}=\mathcal Z_{G_1} \mathcal Z_{G_2}(1-\chi)+\mathcal Z_{G_1} \mathcal Z_{G_2}\chi$, the first term periodic and as regular as $\hat a_\eta$, the second carrying the localized power-law singularities at $\bm k=\bm 0$.

(\emph{Step 1:\,Regular part}) We begin the discussion with the regular part. Expand $\mathcal Z_{G_1}\mathcal Z_{G_2}(1-\chi)$ into the four products of Epstein and Fourier parts. The term involving products of Epstein zeta functions is a product of analytic functions on $\mathrm{supp}(1-\chi)$ multiplied by a smooth periodic cutoff. 
Therefore, it is smooth and periodic on all of $\mathrm{BZ}$, thus exhibiting superalgebraically decaying Fourier coefficients. 

The Fourier coefficients of the two mixed terms are given by the lattice convolution of a superalgebraically decaying function with $a_\eta$, hence decaying asymptotically as $|\bm x|^{-\gamma_\eta}$. 
Finally, the Fourier coefficients of $\hat a_1 \hat a_2$ are $\sum_{\bm y\in \Lambda}a_1(\bm y)a_2(\bm x-\bm y)$. Split the sum at $|\bm y|=|\bm x|/2$. For $|\bm y|\le |\bm x|/2$, bound $|a_2(\bm x-\bm y)|\le C_2(|\bm x|/2)^{-\gamma_2}$ by the inverse triangle inequality. For $|\bm y|>|\bm x|/2$, bound $|a_1(\bm y)|\le C_1 (|\bm x|/2)^{-\gamma_1}$. The absolute value of the sum is bounded by the $\ell^1$ norm of the remaining summand. This yields
\[
\sum_{\bm y\in \Lambda}|a_1(\bm y) a_2(\bm x-\bm y)|\le \Vert a_1\Vert_{\ell^1}C_2(|\bm x|/2)^{-\gamma_2}+\Vert a_2\Vert_{\ell^1}C_1(|\bm x|/2)^{-\gamma_1}\le C |\bm x|^{-\min(\gamma_1,\gamma_2)},
\]
for some $C>0$. All of this is absorbed into $a_3$.

(\emph{Step 2:\,Singular part}) We then proceed with the localized singular part. On $\mathrm{supp}\,\chi$,
\[
\mathcal Z_{G_1}\mathcal Z_{G_2}
=\frac{S_1S_2}{V_\Lambda^{2}}
+\frac{S_1R_2+S_2R_1}{V_\Lambda}
+R_1R_2,
\]
yielding three groups requiring separate analysis. 
For the first group, we have
\[
\frac{S_1S_2}{V_\Lambda^{2}}=\sum_{i=1}^{N_1}\sum_{j=1}^{N_2}b_{i1}b_{j2} \frac{\hat s_{\nu_{i1}}\hat s_{\nu_{j2}}}{V_\Lambda^{2}}.
\]
For $\nu\notin d+2\mathds N_0$, we have $\hat s_\nu=c_\nu|\bm k|^{\nu-d}$, avoiding potential logarithmic terms. By the restriction on $\nu_{i\eta}$, the resulting product is therefore again a power-law,
\[
\frac{\hat s_{\nu_{i1}}\hat s_{\nu_{j2}}}{V_\Lambda^{2}}=\frac{c_{\nu_{i1}}c_{\nu_{j2}}}{V_\Lambda^{2}}\,|\bm k|^{\nu_{i1}+\nu_{j2}-2d}=\frac{c_{\nu_{i1}}c_{\nu_{j2}}}{V_\Lambda\,c_{\nu_{i1}+\nu_{j2}-d}}\frac{\hat s_{\nu_{i1}+\nu_{j2}-d}}{V_\Lambda}.
\]
After writing $\hat s_\mu/V_\Lambda=Z_{\Lambda,\mu}-Z^{\mathrm{reg}}_{\Lambda,\mu}$, we note that $Z^{\mathrm{reg}}_{\Lambda,\mu}\chi$ is smooth and periodic, which can therefore be absorbed in the Fourier series. This provides the double sum of $N_1N_2$ Epstein zeta functions. Here and below, $Z_{\Lambda,\mu}\chi$ differs from $Z_{\Lambda,\mu}$ by the smooth periodic function $Z_{\Lambda,\mu}(1-\chi)$, which is absorbed into the Fourier series as well. If $\nu_{i1}+\nu_{j2}-d=d+2n$, the product is a polynomial in $\bm k$ and hence analytic, and its product with $\chi$ can therefore be completely absorbed into the Fourier series. Correspondingly, the coefficient $c_{\nu_{i1}+\nu_{j2}-d}$ exhibits a pole there, so that with the convention $1/c_{\nu_{i1}+\nu_{j2}-d}=0$ no Epstein term is generated.

For the second group, $\mathrm{Re}(\nu)>d$ gives $\hat s_\nu(\bm 0)=0$, yielding
$R_\mu(\bm 0)=\mathcal Z_{G_\mu}(\bm 0)$. For $\eta\neq\mu$, apply the split
\[
S_\eta R_\mu=\mathcal Z_{G_\mu}(\bm 0)\,S_\eta
+S_\eta\big(A_\mu-A_\mu(\bm 0)\big)
+S_\eta\big(\hat a_\mu-\hat a_\mu(\bm 0)\big).
\]
The first term contributes
$\mathcal Z_{G_\mu}(\bm 0)\sum_i b_{i\eta}Z_{\Lambda,\nu_{i\eta}}$. For $(\eta,\mu)=(1,2)$ and $(2,1)$, we obtain the two single sums in the theorem statement. In the second term, we use that $A_\mu$ is analytic and even. Hence, its
first-order Taylor coefficient vanishes and Cauchy estimates on the ball of radius
$(\rho_{\Lambda^\ast}-r)/\sqrt2$ give $A_\mu-A_\mu(\bm 0)=\mathcal O(|\bm k|^{2})$ on $\mathrm{supp}\,\chi$. Each summand is then $|\bm k|^{\nu_{i\eta}-d}\chi$ times a smooth function vanishing to second order at $\bm 0$. Expanding this function into homogeneous polynomials shows that the Fourier coefficients of the summand are of order $|\bm x|^{-(\mathrm{Re}(\nu_{i\eta})+2)}$.  The third term needs to be estimated on the coefficient side. Up to superalgebraically decaying corrections, the coefficients of $\hat s_\nu\chi/V_\Lambda$ are equal to $\mathcal K_\nu$, and $\hat a_\mu(\bm 0)=\sum_{\bm y\in \Lambda}a_\mu(\bm y)$. Therefore, the coefficients of $S_\eta(\hat a_\mu-\hat a_\mu(\bm 0))\chi/V_\Lambda$ are, up to corrections of order $|\bm x|^{-\gamma_\mu}$,
\[
\sum_{i=1}^{N_\eta} b_{i\eta}\sum_{\bm y\in\Lambda}
\Big(\mathcal K_{\nu_{i\eta}}(\bm x-\bm y)-\mathcal K_{\nu_{i\eta}}(\bm x)\Big)
a_\mu(\bm y).
\]
As before, split the inner sum at $|\bm y|=|\bm x|/2$. On $|\bm y|\le|\bm x|/2$, the terms in brackets obey
\[\mathcal K_{\nu_{i\eta}}(\bm x-\bm y)-\mathcal K_{\nu_{i\eta}}(\bm x)=-\bm y\cdot\nabla\mathcal K_{\nu_{i\eta}}(\bm x)
+\mathcal O\big(|\bm y|^{2}|\bm x|^{-\mathrm{Re}(\nu_{i\eta})-2}\big).\] The linear term cancels because the summation region is symmetric and $a_\mu$ is even, and $\sum_{|\bm y|\le R}|\bm y|^{2}|a_\mu(\bm y)| =\mathcal O\big(C_\mu R^{(d+2-\gamma_\mu)_+}\big)$, with a logarithm in place of the power in case that $\gamma_\mu=d+2$. 
This region therefore contributes the exponent $\mathrm{Re}(\nu_{i\eta})+\min(2,\gamma_\mu-d)$. On $|\bm y|>|\bm x|/2$ the two terms are bounded separately, by $C_\mu(|\bm x|/2)^{-\gamma_\mu}\|\mathcal K_{\nu_{i\eta}}\|_{\ell^1}$ and by $|\bm x|^{-\mathrm{Re}(\nu_{i\eta})}\sum_{|\bm y|>|\bm x|/2}|a_\mu(\bm y)|$, of exponents $\gamma_\mu$ and $\mathrm{Re}(\nu_{i\eta})+\gamma_\mu-d$ respectively, where $\|\mathcal K_\nu\|_{\ell^1}=Z_{\Lambda,\mathrm{Re}(\nu)}(\bm 0)$.

For the third group, $A_1A_2\chi$ is smooth and periodic, $A_\eta \hat a_\mu\chi$ has
coefficients given by a convolution of a superalgebraically decaying sequence with
$a_\mu$, and $\hat a_1 \hat a_2$ was treated in Step~1. Therefore, all decay as
$|\bm x|^{-\min(\gamma_1,\gamma_2)}$ or faster.

Taking the minimum over all exponents yields $\gamma_3$. Step~1 and the third group of Step~2 give $\min(\gamma_1,\gamma_2)$. In the second group of Step~2, the second term gives $\min_i\mathrm{Re}(\nu_{i\eta})+2$, and the third term gives $\gamma_\mu$ and $\mathrm{Re}(\nu_{i\eta})+\min(2,\gamma_\mu-d)$. The latter equals $\mathrm{Re}(\nu_{i\eta})+2$ for $\gamma_\mu\ge d+2$ and exceeds $\gamma_\mu$ otherwise. All other terms are smooth and periodic. The logarithm at $\gamma_\mu=d+2$ does not enter the bound, since $|\bm x|^{-\mathrm{Re}(\nu_{i\eta})-2}\log|\bm x|$ decays faster than $|\bm x|^{-\gamma_\mu}$.
Finally, every contribution is a product or convolution of even functions, as $\chi$ is radial, so $a_3$ is even and $\mathcal Z_{G_1}\mathcal Z_{G_2}$ is semi-analytic in the sense of \cref{def:semianalytic}.
\end{proof}

Using \cref{thm:convolution,thm:multiplication}, multiplications and convolutions can be carried out within a semi-analytical framework. However, applying this procedure requires an additional postprocessing step, as otherwise an explosion of the number of Epstein zeta terms with the number of operations would occur. Starting with $N_1$ and $N_2$ Epstein terms, a convolution results in $N_1 N_2$ Epstein terms while a multiplication yields $N_1 N_2 +N_1+N_2$ terms. While terms with duplicate exponents can be merged by adding their prefactors, this is not sufficient to reduce the overall increase in computationally expensive Epstein terms. In addition, cancellation error occurs when applying graph multiplication to Epstein zeta functions with $\mathrm{Re}(\nu_i)\gg d$. As both graph multiplication and convolution raise the real part of the exponents, this effect always occurs after a sufficiently large number of operations. Both challenges are overcome by defining a cutoff exponent $\nu_\mathrm{max}=d+\sigma_\mathrm{max}$, above which Epstein zeta functions are moved into the Fourier series. 

For any semi-analytic kernel $K$ and $\sigma_\mathrm{max}\ge2$, compressing the computationally expensive Epstein terms with $\mathrm{Re}(\nu_j)>\nu_\mathrm{max}$ by moving them into the Fourier series leaves the kernel unchanged while preserving semi-analyticity, as $\mathrm{Re}(\nu_j)\ge d+2$. The following lemma bounds the error of a subsequent truncation of the Fourier series to the balanced cell $\Lambda_n$.

\begin{lemma}[Truncation error]
\label{lem:compression}
Consider a kernel $K$ with $|K(\bm x)|\le c \vert \bm x\vert^{-\gamma}$, with $\gamma>d$, for $\bm x\neq \bm 0$ and $c>0$. Then, for $n \in \mathds N$, $n\ge 4$,
\[
\left \vert \sum_{\bm x\in \Lambda\setminus \Lambda_n} e^{-2\pi i \bm x\cdot \bm k} K(\bm x) \right \vert \le c \,C_{\Lambda,\gamma}\, n^{-(\gamma-d)},
\]
uniformly in $\bm k\in \mathrm{BZ}$ for some $C_{\Lambda,\gamma}>0$ depending only on $\Lambda$ and $\gamma$.
\end{lemma}
\begin{proof}
    Consider $\bm m \in \mathds Z^d$. Then $A\bm m\not\in \Lambda_n$ implies $\vert \bm m\vert_\infty\ge \lceil n/2\rceil$. Further, $\vert A\bm m\vert\ge \alpha \vert \bm m\vert_{\infty}$ with $\alpha$ the smallest singular value of $A$. The sum is then uniformly bounded in $\bm k$ by
    \[
    c\,\alpha^{-\gamma}\sum_{\substack{\bm m\in \mathds Z^d\\ \vert \bm m\vert_\infty\ge \lceil n/2\rceil}}\vert \bm m\vert_\infty^{-\gamma}.
    \]
    The shell with $\vert \bm m\vert_\infty = j$, including $(2j+1)^d-(2j-1)^d$ points, lies on the faces of the associated cube with $2d$ faces, each with $(2j+1)^{d-1}$ points. Hence, the number of points in the shell is bounded by
    \[
    (2j+1)^d-(2j-1)^d\le  2d (2j+1)^{d-1}\le 2d\,(3j)^{d-1}.
    \]
    Using that $t\mapsto t^{-(\gamma-d+1)}$ is decreasing, together with $\gamma>d$, and  $\lceil n/2 \rceil-1\ge n/4$ allows us to bound the sum by 
    \[
    \sum_{\substack{\bm m\in \mathds Z^d\\ \vert \bm m\vert_\infty\ge \lceil n/2\rceil}}\vert \bm m\vert_\infty^{-\gamma}\le 2d\,3^{d-1}\int_{n/4}^\infty t^{-(\gamma-d+1)}\,\mathrm d t= \frac{2d\,3^{d-1}}{\gamma-d}(n/4)^{-(\gamma-d)},
    \]
    which yields the stated bound.
\end{proof}

The algebra defined by \cref{thm:convolution,thm:multiplication} allows us to generate arbitrary series-parallel graph lattice sums for blocks with $|E|$ edges from the same number of bridges, if their kernels admit the semi-analytic form in \cref{def:semianalytic}. The following remark summarizes the resulting scaling of numerical cost with the number of operations, focussing on the case of uniform power-law kernels, which is of high importance for physical applications, i.e. in the LRTFIM.

\begin{remark}[Cost of series-parallel graph zeta functions]
\label{rem:sp_cost}
Consider power-law kernels $K_e=\mathcal K_\nu$ on all edges, $\sigma=\mathrm{Re}(\nu)-d>0$, and compress at
$\nu_{\mathrm{max}}=d+\sigma_{\mathrm{max}}$, $\sigma_{\mathrm{max}}\ge2$, after each of the
$|E|-1$ operations of \cref{thm:tw2_zeta}. A term combining $m$ edges, either through multiplication or convolution, exhibits exponents
$m\nu-jd$ with $j<m$, of real part at least $d+m\sigma$, so only terms with
$m\le\sigma_{\mathrm{max}}/\sigma$ survive. After merging equal exponents their number is
bounded independently of the graph. For a path with $\sigma=1.3$ and
$\sigma_{\mathrm{max}}=4$, for instance, only $\nu$, $2\nu-d$ and $3\nu-2d$ are kept, however long the
path. Exponents at and near $d+2\mathds N$, corresponding to logarithmic singularities appearing at $\bm k=\bm 0$ in the Epstein zeta function, are treated in the following remark. When storing regular parts on the truncated lattice
$\Lambda_n$, each convolution is a pointwise product of $n^d$ kernel values and each
multiplication one fast Fourier transform, so $\zeta_G$ requires
$\mathcal O(|E|\,n^d\log (n^d))$ operations.
\end{remark}

In \cref{thm:multiplication}, we have excluded exponents $\nu\in d+2\mathds N$, where the singularity of the Epstein zeta function at $\bm k=\bm 0$ exhibits an additional logarithm factor. The following remark discusses how this particular case is stably handled.

\begin{remark}[Avoidance of poles]
\label{rem:pole_family}
Exponents in $d+2\mathds N$, where $c_\nu$ has a pole and $\hat s_\nu$ is logarithmic by \cref{lem:holomorphy}, are excluded in \cref{thm:multiplication}, as they would require generalized Epstein zeta functions with additional powers of logarithms in the kernel. For uniform exponents, such exponents arise for isolated $\sigma$ only. A term $b\mathcal K_\nu$ with such an exponent is compressed into the regular part before a multiplication, which leaves the kernel unchanged and preserves semi-analyticity, as $\mathrm{Re}(\nu)\ge d+2$. A product exponent $\nu_{i1}+\nu_{j2}-d$ landing on $d+2\mathds N$ carries the coefficient zero in \cref{thm:multiplication}. No term is lost by this convention, as the product of the two singular parts is then a polynomial in $\bm k$, which is included in the Fourier series.
\end{remark}

After discussing the cost of constructing series-parallel graphs from single edges, we comment on the achieved accuracy when truncating the short-range part to a finite $\Lambda_n$.

\begin{remark}[Accuracy]
    \label{rem:accuracy}
Multiplications of graph zeta functions by \cref{thm:multiplication} are exact on $\mathrm{BZ}_n$, as the product is sampled there. However, the regular part of the product is obtained from these samples and therefore contains the coefficients outside of $\Lambda_n$ folded onto $\Lambda_n$. This causes an error of order $n^{-\gamma}$ in every subsequent convolution, with $\gamma$ the decay exponent from \cref{thm:multiplication}. Two kinds of regular parts are instead truncated to $\Lambda_n$, namely the regular part produced by a convolution and a term $b_j\mathcal K_{\nu_j}$ moved into the regular part through compression. Their samples miss the tail of the Fourier series, which is of order $n^{-(\gamma'-d)}$ at $\bm k=\bm 0$ by \cref{lem:compression}, where $\gamma'$ is the decay exponent from \cref{thm:convolution} or the exponent of the compressed term, and smaller away from $\bm k=\bm 0$, where the tail oscillates. A two-connected block closes with a convolution, since a series composition at the root would create a cut vertex. The truncated samples therefore enter only through a multiplication, which averages them over $\mathrm{BZ}_n$, and the truncation costs order $n^{-\gamma'}$ as well. For uniform exponents $\nu=d+\sigma$ with $\sigma_{\mathrm{max}}\ge\sigma+2$ and no exponent in $d+2\mathds N$, $\gamma$ is at least $\nu+2$ and $\gamma'$ exceeds $\nu+2$, so the error of a series-parallel block scales as $n^{-(d+\sigma+2)}$, see \cref{fig:sp_algebra_error_n}. A term compressed by \cref{rem:pole_family} has $\gamma'=d+2m$ for $\nu=d+2m$ and costs order $n^{-(d+2m)}$, see panel (b) of \cref{fig:sp_algebra_error_n} at $\sigma=1/2$. All errors are uniform in $\bm k$.
\end{remark}

We note that the error scaling predicted in the above remark is also observed numerically in \cref{fig:sp_algebra_error_n} in the later \cref{sec:benchmarks}.

The compression in \cref{rem:pole_family} avoids exponents on $d+2\mathds N$, but exponents close to it require care as well. For such a term, the cross terms of \cref{thm:multiplication} carry the prefactor $c_\nu$, whose gamma factor grows like the inverse distance to the pole, and the Fourier series has to compensate them. Both parts are then large and cancel, which amplifies the errors of \cref{rem:accuracy} by this factor. The implementation therefore compresses such a term when the distance is below $10^{-2}$ and its coefficient dominates the representation. A product exponent near $d+2\mathds N$ has a small coefficient, proportional to the distance by \cref{thm:multiplication}, and is left in place. If it lands within $10^{-4}$ of $d+2\mathds N$, it is placed at a distance of $10^{-6}$ above the pole with the corresponding coefficient, so that the polynomial of \cref{rem:pole_family} is carried by an Epstein term rather than by the truncated Fourier series, where it would cancel against the analytic part.

Combined, \cref{thm:convolution,thm:multiplication,lem:compression} transform the combinatorial procedure from \cref{sec:sp-construction} into a numerical algorithm for stably computing graph lattice sums and graph zeta functions for general series parallel blocks operating entirely on triples $(a,\bm b,\bm\nu)$. Within the elimination procedure, parallel edges are eliminated with convolution and nodes of degree 2 by multiplication, yielding the graph lattice sum of the reduced graph at numerical accuracy. Meanwhile, the compression after every operation keeps the number of Epstein zeta terms bounded to small numbers even for large operations counts by discarding terms with exponents larger than a threshold value into the Fourier series, while retaining the leading order singularities in $\bm k$ analytically. Importantly, it not only keeps the representation bounded and numerically tractable, but also avoids operations that are numerically harmful, since Epstein zeta functions with $\mathrm{Re}(\nu)\gg d$
approach the constant $\mathcal K_\nu$-sum over the shortest lattice vectors which would lead to cancellation error. 

Combining the results of the previous three section, any graph lattice sum for any graph whose blocks are either bridges or series-parallel can be computed, corresponding to blocks with $\mathrm{tw}(\tilde G)\le 2$ by \cref{cor:sp_tw_equivalence}, with $G$ the augmented block with the edge between terminals added. These blocks form the largest part of physical corpora. For the 0qp corpus  of the LRTFIM till order $r=13$, the $70.35\,\%$ analytic blocks, computable by \cref{sec:basic_blocks}, are accompanied by remaining $23.34\,\%$ (4\,483) additional series-parallel blocks, leading to a total of $93.70\,\%$ blocks that can be computed by the methods discussed so far. For the 1qp corpus till $r=11$, the $83.87\,\%$ of analytic blocks are accompanied by $14.09\,\%$ (9\,863) of series-parallel blocks, leading to an even higher total of $97.96\,\%$ of computable blocks. The remaining $6.30\,\%$ of blocks for 0qp and $2.95\,\%$ for 1qp are blocks with genuine treewidth $\mathrm{tw}(\tilde G)\ge 3$ that require a different evaluator, which is discussed in the next section.

\section{Tensor-network bucket elimination for highly connected graphs}
\label{sec:tensor}

The previous sections have presented an efficient method for evaluating the graph lattice sum for any graph that consists of blocks with $\mathrm{tw}(\tilde G)\le 2$. Albeit rare, blocks with $\mathrm{tw}(\tilde G)\ge 3$ do occur in high-order physical corpora and can in general not be neglected if precise results are sought. For these blocks, series-parallel reduction stalls and is unable to create the graph from two operations alone. 

The first important fact to notice is that, while series-parallel reduction does not reduce such a block to a single bridge outright, the work of the preceding sections remains highly useful even at larger treewidth. The reason is that we can still apply the series-parallel reduction, replacing any series-parallel part of the block by a single bridge with effective kernel $K^{\mathrm{eff}}_e$. Once the reduction stalls, every node other than the terminals has degree at least three. This does not only reduce numerical work by reducing the number of node positions to be summed over, but also improves accuracy under discretization.  The resulting sum converges more rapidly, due to the number of edges connecting to a node (again, excluding terminals) being at least $3$, and only a suitable discretization scheme needs to be found. We evaluate the arising sums efficiently through tensor network bucket elimination \cite{Dechter1999}, a special case of the sum-of-products framework of \cite{AjiMcEliece2000}. There the numerical cost is no longer exponential in the number of nodes $|V|$, which reaches $12$ for the LRTFIM corpora, but only in the block's treewidth, which never exceeds $4$ for any graph in the corpora. Note that tensor network contraction has been used for simulating quantum computations, i.e.~in \cite{MarkovShi2008}. An important ingredient, which separates our application from standard cases, is that the edge kernels of graph lattice sums only depend on relative distances and are thus convolution operators, which removes up to two powers from the numerical scaling. In total, the reduction in complexity renders all remaining blocks within the LRTFIM corpora computable to high precision on a standard desktop machine, even for lattices in $d=3$ spatial dimensions and interaction exponents close to $d$.

\begin{definition}[Discretized graph lattice sums]
\label{def:discretized}
Consider a block $G=(V,E,K,s,t)$ with pinning node $p$ and let $n\in\mathds N$. We consider two discretizations of the graph lattice sum in \cref{def:graph-lattice-sum}, a periodic scheme and a box truncation. Both replace $\Lambda$ by the balanced cell $\Lambda_n$ from \cref{sec:algebra}, which we recall as
\[
\Lambda_n=A\,\{-\lceil n/2\rceil+1,\ \dots,\ \lfloor n/2\rfloor\}^d,
\]
so that each vertex ranges over $n^d$ positions and each residue class modulo $n\Lambda$ occurs exactly once. They differ only in how edge differences $\delta \bm x_e$ are formed. In both cases, the discretized graph lattice sum reads \[
\mathcal Z^{(n)}_{G}(\bm k)=\sum_{\{\bm x_v\in \Lambda_n\}_{v\neq p}}
e^{-2\pi i\,\delta\bm x^{(n)}_{(s,t)}\cdot\bm k}\prod_{e\in E}K_e(\delta\bm x^{(n)}_e),
\]
where we choose $\bm x_p=\bm 0$. The periodic scheme reduces the differences modulo $n\Lambda$ back into $\Lambda_n$, hence $\delta \bm x_e^{(n)}\in \Lambda_n$, whereas the box truncation takes the original definition $\delta \bm x^{(n)}_e=\delta\bm x_e=\bm x_{e^+}-\bm x_{e^-}\in\Lambda$.
\end{definition}
As already established before, the direct evaluation of the above sum is prohibitively expensive and only possible for the smallest blocks, small space dimension $d$, and rapidly decaying kernels. Its summand, however, is a product of tensors, one for each edge and each carrying the positions of two vertices as its indices. Such a product, together with the prescription of which indices to sum over, is a tensor network and its tensors are referred to as factors in the literature. We can then perform the sum one index at a time. One collects all tensors carrying the chosen index, multiplies them, and then sums over the index. This leads to a new tensor, that carries the remaining indices of the collected factors. If we iterate this procedure until no index is left, the sum is evaluated and the cost in memory and time is not governed by $\vert V\vert$, but by the largest number of indices an intermediate tensor carries. The associated method is called tensor network bucket elimination, developed by Rina Dechter in \cite{Dechter1999}. The
following definition provides a formal definition of the required tensor network concepts, following \cite{Dechter1999,kschischang2001factor,koller2009probabilistic}.

\begin{definition}[Tensor-network representation of graph lattice sums]
\label{def:factors}
For a scope $U\subseteq V$, let $\mathcal X_U
    =
    \left\{
        \bm x_U:U\to\Lambda_n
    \right\}$, with $\bm x_v=\bm x_U(v)$,
be the configuration space of the vertices in $U$. 
A factor is a function $\phi:\mathcal X_U\to\mathds C$ with
$\operatorname{scope}(\phi)=U$, and a tensor network is a finite collection of
such factors. Thus, $\phi$ can be identified as a tensor of order $|U|$, with one mode of length
$n^d$ for each vertex in its scope. A two-point factor with scope $\{u,v\}$ is
called translation invariant if it depends only on the relative displacement
$\delta\bm x^{(n)}_{(u,v)}$ as in \cref{def:discretized}.

For each edge $e\in E$, set $U_e=\{e^-,e^+\}$
and define the edge factor $\phi_e:\mathcal X_{U_e}\to\mathds C$ by
\[
    \phi_e(\bm x_{U_e})
    =
    K_e\bigl(\delta\bm x_e^{(n)}\bigr).
\]
Likewise, set $U_{\bm k}=\{s,t\}$
and define the momentum factor
$\phi_{\bm k}:\mathcal X_{U_{\bm k}}\to\mathds C$ by
\[
    \phi_{\bm k}(\bm x_{U_{\bm k}})
    =
    e^{-2\pi i\,\delta\bm x^{(n)}_{(s,t)}\cdot\bm k}.
\]
The initial factor family of a discretized graph lattice sum is therefore
\[
    \mathcal P_1
    =
    \{\phi_e\}_{e\in E}\cup\{\phi_{\bm k}\},
\]
and the graph lattice sum admits the tensor-network representation
\[
    \mathcal Z_G^{(n)}(\bm k)
    =
    \sum_{\bm x\in\mathcal X_V,\ \bm x_p=\bm 0}
    \prod_{\phi\in\mathcal P_1}
    \phi\bigl(\bm x_{\operatorname{scope}(\phi)}\bigr).
\]
\end{definition}

The factors of the tensor network can now be collected into so-called buckets and contracted by summing over the vertex positions in a suitable elimination order. This reduces the computational complexity from exponential in the number of vertices to exponential in the width of the elimination ordering, whose minimum is related to the treewidth of the block \cite{Dechter1999}.

\begin{definition}[Buckets and elimination orderings]
\label{def:buckets}
Let $\pi=(v_1,\ldots,v_m)$, with $m=|V|-1$,
be a permutation of $V\setminus\{p\}$, called an elimination ordering.
Starting from the initial factor family $\mathcal P_1$ of
\cref{def:factors}, define recursively for $j=1,\ldots,m$ the bucket
\[
    \Phi_j
    =
    \left\{
        \phi\in\mathcal P_j:
        v_j\in\operatorname{scope}(\phi)
    \right\}
\]
as well as its remaining scope
\[
    S_j
    =
    \left(
        \bigcup_{\phi\in\Phi_j}
        \operatorname{scope}(\phi)
    \right)
    \setminus\{v_j\}.
\]
Eliminating $v_j$ replaces all factors in $\Phi_j$ by the factor
\[
    \tau_j(\bm x_{S_j})
    =
    \sum_{\bm x_{v_j}\in\Lambda_n}
    \prod_{\phi\in\Phi_j}
    \phi\bigl(\bm x_{\operatorname{scope}(\phi)}\bigr),
    \qquad
    \operatorname{scope}(\tau_j)=S_j,
\]
and yields the factor family
\[
    \mathcal P_{j+1}
    =
    \bigl(\mathcal P_j\setminus\Phi_j\bigr)
    \cup\{\tau_j\}.
\]
The width $w(\pi)$ of the elimination ordering $\pi$ is
\[
    w(\pi)
    =
    \max_{1\le j\le m}|S_j|.
\]
\end{definition}

Classical bucket elimination requires $\mathcal O(|\Lambda_n|^{w+1})$ arithmetic operations and $\mathcal O(|\Lambda_n|^{w})$ memory, with $w$ is the largest scope met during the elimination, which is bounded by the treewidth of the graph formed by the scopes. The special structure appearing in of graph lattice sums allows us to reduce the scaling exponent of memory by one order and the exponent pf work by up to two orders. The first ingredient is translational invariance of the edge kernels, which permits specific elimination steps to be carried out using FFT, removing one unit from the exponent of the related elimination step. Importantly, it also allows us to obtain the values at the whole $\bm k$ grid at the cost of a single $\bm k$ value (neglecting locarithmic scaling corrections).

\begin{lemma}[Convolutional elimination]
\label{lem:peeling}
Consider step \(j\) of the elimination in \cref{def:buckets}. Suppose that
there exists \(u\in S_j\) such that every factor in \(\Phi_j\) containing
\(u\) is a translation-invariant two-mode factor with scope
\(\{u,v_j\}\). Let
\[
    \Phi_j^{(u)}
    =
    \bigl\{
        \phi\in\Phi_j:
        u\in\operatorname{scope}(\phi)
    \bigr\},
    \qquad
    W
    =
    S_j\setminus\{u\},
\]
and write
\[
    \phi(\bm x_u,\bm x_{v_j})
    =
    K_\phi(\bm x_u-\bm x_{v_j}),
    \qquad
    K
    =
    \prod_{\phi\in\Phi_j^{(u)}}K_\phi,
    \qquad
    \psi
    =
    \prod_{\phi\in\Phi_j\setminus\Phi_j^{(u)}}\phi.
\]
Then \(\psi\) is independent of \(\bm x_u\), and the elimination of \(v_j\)
is, fibrewise in \(\bm x_W\), the cyclic convolution
\[
    \tau_j(\bm x_u,\bm x_W)
    =
    \bigl(\psi(\,\cdot\,,\bm x_W)\star K\bigr)(\bm x_u).
\]
It can therefore be evaluated using
\[
    \mathcal O\!\left(
        \bigl(|\Phi_j|+\log|\Lambda_n|\bigr)
        |\Lambda_n|^{|S_j|}
    \right)
    \quad\text{operations and}\quad
    \mathcal O\bigl(
        |\Lambda_n|^{|S_j|}
    \bigr)
    \quad\text{memory}.
\]
\end{lemma}

\begin{proof}
By assumption, the product of all factors in \(\Phi_j^{(u)}\) depends only
on \(\bm x_u-\bm x_{v_j}\), while no factor in
\(\Phi_j\setminus\Phi_j^{(u)}\) depends on \(\bm x_u\). Hence
\[
    \tau_j(\bm x_u,\bm x_W)
    =
    \sum_{\bm x_{v_j}\in\Lambda_n}
    \psi(\bm x_{v_j},\bm x_W)
    K(\bm x_u-\bm x_{v_j}),
\]
where differences are taken modulo \(n\Lambda\). This is a cyclic
convolution on \(\Lambda_n\) for every fixed \(\bm x_W\).

There are \(|\Lambda_n|^{|W|}\) fibres, and each convolution requires
\(\mathcal O(|\Lambda_n|\log|\Lambda_n|)\) operations using FFT.
Since \(|W|=|S_j|-1\), their total cost is
\(\mathcal O(|\Lambda_n|^{|S_j|}\log|\Lambda_n|)\). Forming \(K\) and
\(\psi\) requires
\(\mathcal O(|\Phi_j|\,|\Lambda_n|^{|S_j|})\) operations. The convolution
output and the factors needed to form it contain at most
\(|\Lambda_n|^{|S_j|}\) entries, which proves the stated bounds.
\end{proof}

The second ingredient is that translation invariance is not restricted to the
edge kernels, but is a general property of all occuring factors. Every factor generated by the elimination of \cref{def:buckets} is invariant
under a common shift of all positions in its scope by $\bm a\in \Lambda_n$. It is therefore
determined by an array with one mode fewer than its scope suggests and the full factor can be obtained through reindexing.

\begin{theorem}[Numerical cost of tensor network bucket elimination]
Consider the torus discretization for a block $G$ with $\mathrm{tw}(G)\ge 2$. Set $N=|\Lambda_n|$ and $\mathrm{tw}=\mathrm{tw}(\tilde G)$, with $\tilde G=G+(s,t)$ for $s\neq t$ and $\tilde G=G$ otherwise. Set the pin $p=s$ and choose an elimination ordering $\pi$ of $V\setminus\{p\}$ satisfying $|S_j|\le\mathrm{tw}$ for every $j$ and eliminating $t$ last when $s\ne t$. We can then compute $\mathcal Z_{G}^{(n)}(\bm k)$, either for a single momentum or on the full momentum grid $\mathrm{BZ}_n$, using $\mathcal O \big(N^{\mathrm{tw}}\big)$ operations and $\mathcal O \big(N^{\mathrm{tw}-1}\big)$ memory. If every step at maximum width is convolutional in the sense of \cref{lem:peeling}, then scaling of arithmetic cost reduces to $\mathcal O \big(N^{\mathrm{tw}-1}\log N\big)$.
\end{theorem}
\begin{proof}
    Choose a tree decomposition of $\tilde G$ of width $\mathrm{tw}$. Then, there exists at least one bag including both $s$ and $t$, since they either share an edge or are equal. Define such a node as the root. We then eliminate from the leaves towards the root, removing one by one vertices in leaves that are not included in the parent nodes, then removing the leaf once this procedure stalls. Note by \cref{def:tree_decomp}\,(3) that such vertices can not be present in any other bag. When only the root remains, eliminate its remaining vertices, leaving $s$ pinned and, when $s\ne t$, eliminating $t$ last. The largest scope appearing is thus $|S_j|\le \mathrm{tw}$, with associated work scaling as $\mathcal O\big(N^{\mathrm{tw}+1}\big)$ and memory scaling as $\mathcal O\big(N^{\mathrm{tw}}\big)$.

    Translational invariance reduces the exponents by one unit. At the start, all factors are invariant under addition of $\bm a \in \Lambda_n$ to all vertex positions, as they only depend on relative displacements modulo $n\Lambda$ and this property is preserved under elimination. Each factor with nonempty scope $U$ is determined by $N^{|U|-1}$ values with one position fixed to zero. Likewise, evaluating $\tau_j$ with one position in $S_j$ fixed to zero requires $N^{|S_j|-1}$ sums of length $N$. Since $G$ is fixed, this gives $\mathcal O(N^{\mathrm{tw}})$ work and $\mathcal O(N^{\mathrm{tw}-1})$ memory. For a convolutional step with $|S_j|\ge2$, fixing one position in $S_j\setminus\{u\}$ as in \cref{lem:peeling} leaves $N^{|S_j|-2}$ convolutions of length $N$. Thus, if every step at maximum width is convolutional, the total work reduces to $\mathcal O(N^{\mathrm{tw}-1}\log N)$. Steps with $|S_j|=1$ cost $\mathcal O(N)$ directly. Accumulating sums and processing convolutions one at a time preserves the memory bound.

    If $s=t$, the result is independent of $\bm k$. Otherwise, the momentum factor enters only the last bucket, so all preceding steps are independent of $\bm k$. The product of the other remaining factors at $\bm x_s=\bm 0$ defines $\psi$, and the final contraction is
    \[
    \mathcal Z_{G}^{(n)}(\bm k)=\sum_{\bm x_t \in \Lambda_n} \psi(\bm x_t) e^{-2\pi i \bm x_t\cdot \bm k},
    \]
    which can be evaluated via FFT at $N\log N$ cost, compared to cost $N$ for a fixed momentum evaluation. These costs fit the stated bounds as $\mathrm{tw}\ge2$.
\end{proof}
For state-of-the-art corpora of the LRTFIM (order $r=13$ for 0qp, order $r=11$ for 1qp), the optimal bound is obtained for more than $99\,\%$ of graphs. Importantly, the highest augmented treewidth obtained equals $\mathrm{tw}=4$ and all of these cases obey the optimal scaling. Thus, scaling of work has been reduced from $\mathcal O\big(N^{|V|-1}\big)$ to  $\mathcal O(N^3\log N)$ even for the most difficult graphs in the corpora. In addition, the above theorem show that the evaluation of the whole Brillouin zone grid is available at the same cost than a single momentum evaluation.

Each method in the preceding section forms a natural generalization of its predecessors and consumes them them as a internal subroutines. After block-factorization, we have first identified analytic building blocks, formed by bridges and cycles. Through series-parallel reduction, this foundation can be used to create all series-parallel blocks, where the leading algebraic tails of the arising edge kernels can be determined at log-linear cost. Finally, as soon as series-parallel reduction stalls, then the tensor network bucket elimination is used. It is important to note that the reduction step should not be skipped, as otherwise, the resulting tensor contractions converge significantly more slowly. After reduction, all non-terminal nodes are at least 3-connected, and a small lattice grid can be used to resolve the blocks to equivalent precision as their series-parallel counterparts.

Using this hierarchy of methods, the full corpora of the LRFTIM can be computed for the full $\bm k$ grid within minutes on a standard laptop at better precision than leading Monte Carlo approaches, which requires 24-hours on a cluster for a single $\bm k$ value \cite{adelhardt2024monte}. The framework is further applicable to any gapped quantum lattice systems, famous examples being the XY or Heisenberg model. We provide a detailed account of the runtime and accuracy of our method in the next section. 

\section{Numerical Results}
\label{sec:benchmarks}

In this section, we benchmark in detail the accuracy and runtime of our method and their scaling with the numerical discretization parameters, the most important one being the size of the discretized lattice $|\Lambda_n|=n^d=N$. As the arising lattice sums are both high-dimensional and, for small exponents, slowly converging, computing reliable benchmark values for single graphs, as well as for realistic 0qp and 1qp corpora, is a challening task. We begin our analysis by analytically available examples for single graphs, focussing on the hard case of power-law edge kernels $\mathcal K_\nu$ with exponents close to the system dimension. We further benchmark more complicated graphs against direct summation at larger exponents, where the resulting sums can still be computed within reasonable timeframes. In order to reduce the computational complexity, while keeping results bit-identical to iterated summation, we apply tensor network bucket elimination using a box truncation combined with Richardson extrapolation. We further provide a detailed analysis of the scaling of the runtime of our method with grid size. In particular, we demonstrate exponential reduction in complexity for series-parallel blocks compared to direct summation, obtaining linear scaling in $|V|$ instead. For dense cores with $\mathrm{tw}(G)\ge 3$, we show that scaling of numerical work never exceeds $|\Lambda_n|^3$ for the whole LRTFIM corpus, where we can further reduce the number of points per dimension $n$ compared to series parallel blocks, as the resulting sums are rapidly convergent after series-parallel reduction. For the realistic case of whole corpora, we focus on the most relevant case of LRTFIM, and compare our results to published Monte Carlo values, which we can reproduce completely within minutes on standard desktop hardware, compared to typical runtimes of 24 hours on a cluster for a Monte Carlo run. 

\subsection{Analytic and numerical benchmarks for series-parallel graphs}

We begin our analysis with the most often occuring and, under direct discretization, most slowly convergent blocks. Here, we first note that analytic blocks, bridges and simple cycles, can be computed to full precision using the method for Epstein zeta functions \cite{buchheit2024epstein} and the method for many-body zeta functions in \cite{buchheit2025epstein_method} at linear scaling with the number of nodes. The remaining slowly convergent sums are then series parallel blocks evaluated through the semi-analytic zeta algebra.

We first use the circle zeta function for a cycle with $L$ nodes and all edge exponents equal to $\nu$ as a benchmark for the series-parallel algebra. For a cycle with $L$ nodes, and terminals associated with two neighboring nodes, the graph zeta function reads
\[
\zeta_G(\bm k) = \Big(\prod_{j=1}^n Z_{\Lambda,\nu}\Big)\ast Z_{\Lambda,\nu}(\bm k).
\]

\begin{figure}
    \centering
    \includegraphics[width=.8\textwidth]{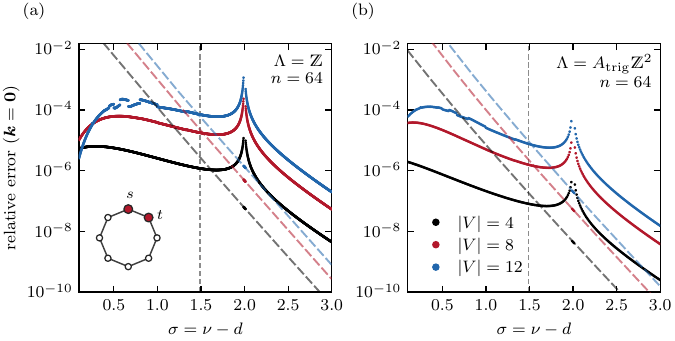}
    \caption{Relative error of the zeta algebra for a cycle with $|V|=4,8,12$ nodes, neighboring terminals, and power-law edge kernels $\mathcal K_\nu$ with equal exponents $\nu=d+\sigma$ at $\bm k=\bm 0$, compared with the analytic benchmark value at $\bm k=\bm 0$. The dashed lines display the value for $\sigma_\mathrm{max}=0$, where only a truncated Fourier series is used. Panel (a) shows the integer lattice $\Lambda$, while (b) displays the results for the trigonal lattice with $n=64$ points per dimension, respectively. The dashed line displays the conservatively chosen cutoff value in $\sigma$, after which $\sigma_\mathrm{max}=0$ is applied.}
    \label{fig:zeta_algebra_accuracy}
\end{figure}

For small values of $\nu$, the associated sum is slowly convergent and direct evaluation approaches scale exponentially with $|V|$. As an analytical benchmark, we use that the value of this graph zeta function at $\bm k=\bm 0$ equals the many-body zeta function from \cite{buchheit2026zeta}, which can be computed at linear cost in $|V|$ to full precision even for $\nu\to d$. In \cref{fig:zeta_algebra_accuracy} we display the relative error of the zeta algebra, measured against the exact reference, at $\bm k=\bm 0$ for (a) the integer lattice $\Lambda=\mathds Z$ and (b) the trigonal lattice $\Lambda=A_\mathrm{trig}\mathds Z^2$,
with 
\[
A_\mathrm{trig}=\begin{pmatrix} 1 & 1/2 \\[2pt] 0 & \sqrt3/2\end{pmatrix},
\]
as a function of the reduced exponent $\sigma=\nu-d$. In both cases, a momentum grid with $n=64$ points per dimension was used. The black dots show the case $|V|=4$, the red dots correspond to $|V|=8$ and the blue dots show $|V|=12$, corresponding to $(|V|-1)d$-dimensional sums each. The solid lines show the zeta algebra for $\sigma_{\mathrm{max}}=4$, where Epstein zeta functions with exponents $\nu<d+\sigma_\mathrm{max}$ are retained, compared with the case $\sigma_{\mathrm{max}}=0$, where only a truncated Fourier series is used.
A qualitatively similar error behavior occurs for both lattice cases. In the regime $0<\sigma<3/2$, the algebra provides a stable error in $\sigma$, while the error of a simple Fourier approach degrades as $\nu\to d$. For $\sigma>3/2$ retaining the singularities improve the results compared to the Fourier approach, as the sums are already quickly convergent, while the gamma pole of the Fourier transform of $|\bm \cdot |^{-\nu}$ at $\nu=d+2$ leads to a slight loss of precision. For that reason, we switch to the Fourier discretization for $\sigma>3/2$, keeping the stable evaluation for small interaction exponents while avoiding cancellation error around gamma poles. Importantly, only a slight loss in precision is observed as $L$ increases, providing us with access to graphs with very high numbers of nodes.

\begin{figure}
    \centering 
    \includegraphics[width=.8\textwidth]{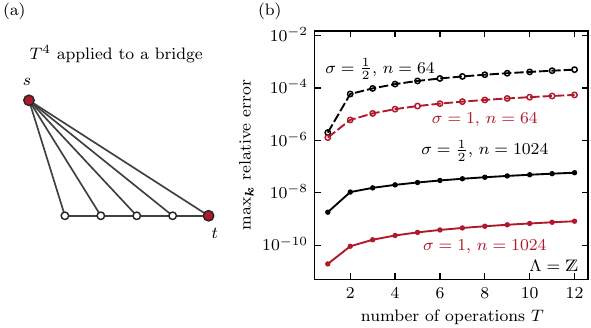}
    \caption{(a) Graph obtained from a single bridge by applying the operation $T$, consisting of a multiplication followed by a convolution, four times. (b) Maximum relative error of the associated graph zeta function $\zeta_G(k)$ for $\Lambda=\mathds Z$ on the grid $k\in \mathrm{BZ}_{64}$ for $\sigma=1/2$ (black) and $\sigma=1$ (red). Open points correspond to the zeta algebra with $n=64$ while solid points correspond to the finer discretization $n=1024$. The reference is the Richardson-extrapolated truncated sum over the box $[-L,L]$ with $L=10\,000,12\,000,\dots,22\,000$, evaluated using bucket elimination at the 64 momentum points. The error increases only linearly with the number of operations.}
    \label{fig:Titer}
\end{figure}

We proceed with a more complicated series-parallel graph. Defining the operation $T$ on a graph zeta function via 
\[
(T \zeta_G)=(\zeta_G Z_{\Lambda,\nu})\ast Z_{\Lambda\nu}(\bm k),
\]
we consider the graph $G$ obtained by iterated applications $T$ to the bridge $Z_{\Lambda,\nu}$, corresponding to a serial and a parallel composition each. The resulting graph after 4 operations is displayed in \cref{fig:Titer}(a). We now measure the maximum relative error on the integer lattice $\Lambda=\mathds Z$ over the whole momentum grid $\mathrm{BZ_n}$ as a function of the number of operations $T$ in panel (b). The open dots correspond to $n=64$, while the filled dots show the finer discretization $n=1024$. Meanwhile, black points correspond to $\sigma=1/2$, while red dots denote $\sigma = 1$. A reliable reference on the whole grid is obtained through tensor network bucket elimination using the box truncation from \cref{def:discretized}, which is bit-identical to direct summation over the truncated lattice $\{-L,\dots,L\}$, yet has a scaling exponent that only depends on the treewidth of the graph, and not on the number of nodes. We further apply a ladder $L=10\,000,12\,000,\dots,22\,000$ and use Richardson extrapolation to further reduce truncation error, leading to a reliable reference value. We find that the error increases only weakly linearly with the number of operations. A slightly larger increase of error between operations number $1$ and $2$ compared to the remaining data points is observed. We find that the method reproduces the complete momentum grid precisely with controlled error increase even for high summation dimensions. Note that the largest value $L=12$ corresponds to a $22$-dimensional slowly convergent sum. Importantly, the algebra precisely resolves the singularities in the momentum $k$ even after 12 consecutive multiplications and singular convolutions.

\begin{figure}
    \centering 
    \includegraphics[width=.8\textwidth]{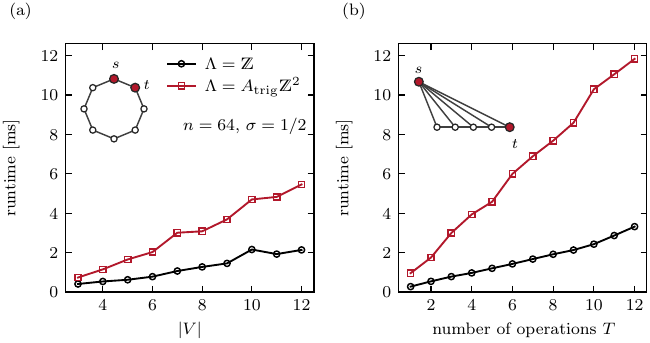}
    \caption{Runtime in milliseconds for evaluating graph zeta functions $\zeta_G$ on a complete $n^d$ momentum grid for $n=64$ and $\sigma=1/2$ on a single core Apple M1 Max core. Panel (a) shows the runtime for a cycle with neighboring terminals as a function of the number of nodes $|V|$, while panel (b) shows the runtime for the graph obtained from repeated applications of $T$, as a function of the operation count. The black dots show $\Lambda=\mathds Z$, while the red squares display the results for the trigonal lattice. The runtime increases only linearly, compared to exponential increase for direct summation.}
    \label{fig:runtime}
\end{figure}

We then proceed with an analysis of the scaling of runtime with the number of nodes in the graph. Using the zeta algebra in \cref{sec:algebra}, we predict that the scaling should reduce from $N^{|V|-1}$ to $|V|N\log N$ for all series-parallel graphs. This is confirmed numerically for all series-parallel blocks in the LRTFIM. We here analyze two already introduced graphs, namely the circle graph with neighboring terminals from \cref{fig:zeta_algebra_accuracy} and the graph obtained by repeated applications of $T$ from \cref{fig:Titer}. In \cref{fig:runtime}\, we display the single core runtime on an Apple M1 Max processor for computing the graphs on an $n^d$ momentum grid with $n=64$. Black dots correspond to the integer lattice, while red squares denote the 2D hexagonal lattice. Panel (a) displays the runtime of the circle graph as a function of the number of nodes $|V|$, whereas panel (b) shows the runtime for computing the assembled graph as a function of number of applied operations $T$. In both cases, runtime increases only linearly, confirming the theoretical prediction. The larger runtime in (b) is due to more complicated operations being applied, consisting of both a multiplication and a convolution within the algebra. 

\begin{figure}
    \centering 
    \includegraphics[width=.8\textwidth]{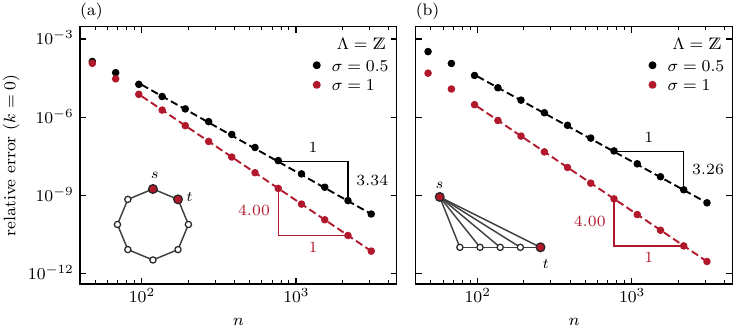}
    \caption{Relative error of the graph zeta function with edge exponent $\nu=d+\sigma$ at $\bm k=\bm 0$ obtained from the zeta algebra for $\Lambda=\mathds Z$ as a function of the number of discretization points $n$ for (a) the cycle graph with $|V|=8$ nodes and (b) the graph obtained by applying $T^4$ to a bridge. The black dots show $\sigma=0.5$, while the red dots display $\sigma = 1$. The function is evaluated at $n=\lfloor 48\cdot 2^{j/2}\rceil$ with $\lfloor \cdot \rceil$ the nearest integer and $j=0,1,\dots,12$. The dashed lines show a linear fit against the logarithmized data, with resulting fit exponents shown as captions to the triangle. The scaling is in good agreement with the theoretical bound $\vert n\vert^{-(d+\sigma+2)}$ from \cref{rem:accuracy}.}
    \label{fig:sp_algebra_error_n}
\end{figure}

We finally investigate the dependency of the relative error of the semi-analytic zeta algebra on the number of discretization points per dimension $n$ in \cref{fig:sp_algebra_error_n} for the cycle with $|V|=8$ nodes in (a) and the graph obtained from applying $T^4$ to a bridge in (b). We here choose $\Lambda=\mathds Z$ and zero momentum and provide the results for $\sigma=0.5$ (black) and $\sigma =1$ (red). The reference value in (a) is obtained from the the closed-form circle zeta function, while a tensor network bucket elimination with box truncation on $[-L,L]$ accompanied by Richardson extrapolation on the ladder with half-widths $L=20\,000,24\,000,\dots,36\,000$ was used for panel (b). From a linear fit of the logarithmized data, we find that the error scales approximately as the power-law
\[
E_\mathrm{rel}(n) \sim |n|^{-(d+\sigma+2)},
\]
with fitted exponents provided in the plot. This error scaling is consistent with the theoretical prediction in \cref{rem:accuracy}.

\subsection{Benchmarks for $\mathrm{tw}\ge 3$ graphs evaluated through the tensor-network}

\begin{figure}
    \centering 
    \includegraphics[width=.8\textwidth]{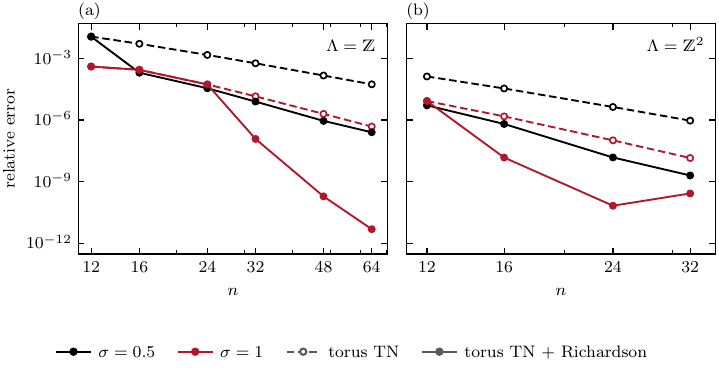}
    \caption{Relative error of tensor network bucket elimination, using torus truncation, for the $K_4$ graph and (a) $\Lambda=\mathds Z$ and (b) $\Lambda=\mathds Z^2$ as a function of the number of grid points per dimension $n$ for $\sigma=1/2$ (black) and $\sigma=1$ (red). The open points (dashed lines) show the bare torus tensor network values, while the solid points display the optimized value using a Richardson ladder.}
    \label{fig:tensor_torus}
\end{figure}

Testing the accuracy of the tensor network for highly connected graphs is a non-trivial task, as the arising sums are high-dimensional objects with only few analytic references available. A comparison with truncated direct summation on a box $A\{-L,\dots L\}^d$ is desirable, yet the rise in dimension restricts its use to graphs with low node counts $|V|$, low spatial dimensions $d$, and small box size $L$. The hurdle of establishing a reliable reference is overcome by a two stage strategy. We first verify that direct nested summation yields exactly the same results as a tensor network with box-truncation. We then establish the box truncation combined with Richardson extrapolation for large $L$ as our reference, as already done in the previous subsection. This is then compared against the torus truncation, used in \cref{sec:tensor}, which permits the reduction of complexity through use of FFTs and provides access to the full momentum grid at the cost of a single momentum evaluation. 

Following this strategy, we have verified that the tensor-network with box-truncation coincides with an explicit nested lattice sum sharing no
code with it, for the graphs $K_4$ (4 nodes, all connected with edges having equal exponent) and $K_5$ (5 connected nodes) on $\mathds{Z}$ and $\mathds{Z}^2$ at $\sigma=0.5$, $1$ and $2$, over box half-widths up to $L=28$, the largest covering $2.1\times10^{8}$ explicit terms.
The largest relative deviation found is $5.3\times10^{-15}$.

In \cref{fig:tensor_torus}, we benchmark the error of the tensor network bucket elimination for the graph $K_4$, using torus truncation, against a large $L$ tensor network with box truncation, equal to machine precision to a truncated sum, including Richardson extrapolation. Panel (a) displays the case of the integer lattice $\Lambda=\mathds Z$, whereas (b)shows the square lattice $\Lambda=\mathds Z^2$. Black points correspond to $\sigma=1/2$, whereas red points denote $\sigma=1$. Open points display the bare torus tensor network value, whereas the solid points include an additional Richardson ladder. Note that deviations in the error reduction due to the extrapolation due to the particular choice of rungs. An analytic treatment of the leading order truncation tails as in the zeta algebra in \cref{sec:algebra} would remove the need for extrapolation. This is, however, nontrivial due to the more complicated singularity structure and will therefore be the topic of future work.

\subsection{Comparison against published Monte Carlo data}

\begin{table}[t]
\centering
\setlength{\tabcolsep}{2.4pt}
\scriptsize
\providecommand{\an}[1]{\textcolor{black!45}{#1}}
\begin{tabular}{@{}llcrrrrrrrrrrrrrc@{}}
\toprule
& & & \multicolumn{13}{c}{order $r$} & \\
\cmidrule(lr){4-16}
Series & $\sigma$-range & Ref. & $1^{\mathrm{A}}$ & $2^{\mathrm{A}}$ & $3^{\mathrm{A}}$ & 4 & 5 & 6 & 7 & 8 & 9 & 10 & 11 & 12 & 13 & $\max_r$ \\
\midrule
\multicolumn{17}{@{}l}{$\Lambda=\mathds{Z}$}\\[1pt]
0qp & $0.5$--$3$ (4) & \cite{Langheld2022} & --- & --- & \an{3.5} & 1.4 & 0.4 & 1.5 & 1.8 & 1.1 & 1.5 & 1.3 & 1.2 & 1.4 & 1.5 & \textbf{1.8} \\
1qp, $k=0$ & $\tfrac23$--$9$ (13) & \cite{Fey2020Diss} & --- & \an{0.0} & \an{0.5} & 0.3 & 0.4 & 0.3 & 0.2 & 0.2 & 0.2 & --- & --- & --- & --- & \textbf{0.4} \\
1qp, $k=0$ & $0.5$--$3$ (4) & \cite{Langheld2022} & --- & --- & \an{2.5} & 1.8 & 0.6 & 1.2 & 2.0 & 1.5 & 1.3 & 0.9 & 1.8 & --- & --- & \textbf{2.0} \\
1qp, $k=1/2$ & $0.5$--$6$ (9) & \cite{Fey2020Diss} & --- & --- & \an{0.3} & 0.4 & 0.7 & 0.3 & 0.3 & 0.2 & 0.2 & --- & --- & --- & --- & \textbf{0.7} \\
1qp, $k=1/2$ & $0.5$--$4$ (8) & \cite{adelhardt2024monte} & --- & \an{2.4} & \an{1.9} & 2.3 & 1.6 & 2.1 & 1.3 & 2.0 & 1.8 & 1.4 & --- & --- & --- & \textbf{2.3} \\
\midrule
\multicolumn{17}{@{}l}{$\Lambda=\mathds{Z}^2$}\\[1pt]
0qp & $0.5$--$8$ (10) & \cite{Fey2020Diss} & --- & \an{0.2} & \an{0.2} & 0.3 & 0.3 & 0.3 & 0.2 & 0.4 & 0.2 & \an{30.9} & --- & --- & --- & \textbf{0.4} \\
0qp & $0.5$--$8$ (6) & \cite{AdelhardtPrivComm} & --- & \an{3.5} & \an{1.4} & 2.0 & 1.6 & 2.7 & 2.1 & 2.1 & 1.9 & 0.7 & 0.9 & 2.1 & 1.0 & \textbf{2.7} \\
1qp, $\bm k=\bm 0$ & $0.75$--$8$ (11) & \cite{fey2019quantum} & --- & \an{0.3} & \an{0.4} & 0.2 & 0.7 & 0.2 & 0.2 & 0.3 & 0.2 & --- & --- & --- & --- & \textbf{0.7} \\
1qp, $\bm k=\bm 0$ & $0.5$--$8$ (5) & \cite{AdelhardtPrivComm} & --- & \an{1.5} & \an{0.9} & 0.9 & 1.1 & 2.6 & 1.2 & 1.8 & 2.1 & 3.1 & 1.8 & --- & --- & \textbf{3.1} \\
1qp, $(1/2,1/2)^T$ & $1$--$8$ (7) & \cite{fey2019quantum} & --- & \an{0.1} & \an{0.2} & 0.3 & 0.4 & 0.2 & 0.2 & 0.2 & 0.1 & --- & --- & --- & --- & \textbf{0.4} \\
\midrule
\multicolumn{17}{@{}l}{$\Lambda=A_{\mathrm{trig}}\mathds{Z}^2$}\\[1pt]
0qp & $0.5$--$8$ (8) & \cite{Fey2020Diss} & --- & \an{0.2} & \an{0.2} & 0.2 & 0.2 & 0.2 & 0.2 & 0.3 & 0.1 & \an{22.8} & --- & --- & --- & \textbf{0.3} \\
1qp, $\bm\Gamma$ & $0.5$--$6$ (8) & \cite{fey2019quantum} & \an{0.2} & \an{0.2} & \an{0.9} & 1.2 & 0.4 & 0.2 & 0.5 & 0.2 & 0.2 & --- & --- & --- & --- & \textbf{1.2} \\
1qp, $\bm K$ & $1.5$--$6$ (5) & \cite{fey2019quantum} & \an{0.3} & \an{0.1} & \an{0.2} & 0.1 & 0.1 & 0.2 & 0.3 & 0.1 & 0.2 & --- & --- & --- & --- & \textbf{0.3} \\
\midrule
\multicolumn{17}{@{}l}{$\Lambda=\mathds{Z}^3$}\\[1pt]
1qp, $\bm k=\bm 0$ & $0.5$--$7$ (8) & \cite{Fey2020Diss} & \an{0.4} & \an{0.4} & \an{0.6} & 0.5 & 0.5 & 0.2 & 0.2 & 0.5 & 0.2 & --- & --- & --- & --- & \textbf{0.5} \\
1qp, $(1/2,1/2,1/2)^T$ & $0.5$--$7$ (8) & \cite{Fey2020Diss} & \an{0.2} & \an{0.2} & \an{0.2} & 0.2 & 0.2 & 0.1 & 0.1 & 0.2 & 0.2 & 0.3 & --- & --- & --- & \textbf{0.3} \\
\bottomrule
\vspace*{.1cm}
\end{tabular}
\caption{
Maximum absolute error of the graph-zeta series coefficients for the 0qp and 1qp LRTFIM corpus
series against published Monte Carlo values,
$\max_{\sigma}\,|c_r-c_r^{\mathrm{MC}}|/\sigma_{\mathrm{MC}}$, per order $r$,
over all exponents $\sigma$ available for the respective series. The second column gives their
range and number. Here,
$r^{\mathrm{A}}$ denotes analytic orders, at which every graph of the corpus reduces to a closed form
and the graph-zeta value is exact, measuring thus only the Monte Carlo error. Grey values are
excluded from $\max_r$. These include the analytic orders, and the top order of the $0$-quasiparticle tables of
\cite{Fey2020Diss}, which are known to be underconverged. Over the $753$ remaining
coefficients the deviation never exceeds $3.1\,\sigma_{\mathrm{MC}}$, with median
$0.12\,\sigma_{\mathrm{MC}}$ and $90\%$ within $1\,\sigma_{\mathrm{MC}}$. The graph-zeta values were
obtained using $n=2048$ grid points per dimension for $d=1$, $n=48$ for $d=2$, and $n=10$ for $d=3$. The complete collection of graph-zeta values needed for the comparison were obtained with GZL within $5$ minutes on 8 cores of an Apple M1 Max processor. In comparison, a cluster Monte Carlo run requires approximately $10^4$--$10^5$ core-hours for a single series evaluation.}
\label{tab:mc_comparison}
\end{table}

\begin{figure}
    \centering
    \includegraphics[width=0.8\textwidth]{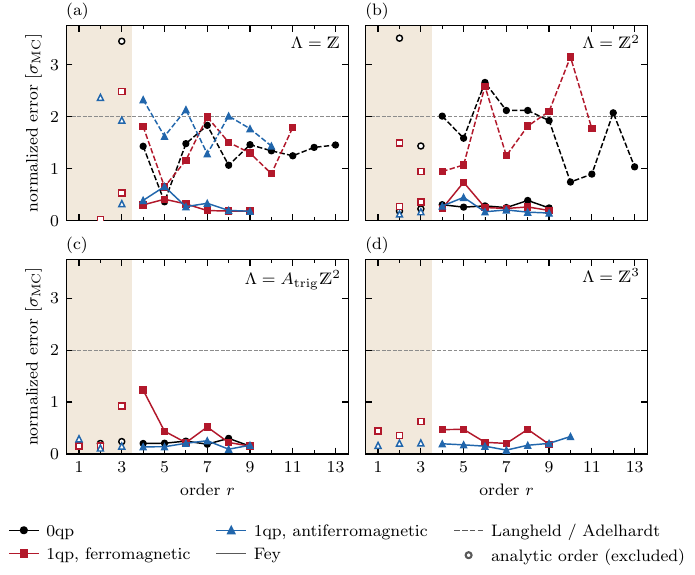}
    \caption{Absolute error of the graph zeta series coefficient $c_r$ at order $r$ against published Monte Carlo results for the LRTFIM as in \cref{tab:mc_comparison} in units of the Monte Carlo standard deviation $\sigma_\mathrm{MC}$. Black points display 0qp, red dots ferromagnetic 1qp, and blue dots  antiferromagnetic 1qp for (a) $\Lambda=\mathds Z$, (b) $\Lambda=\mathds Z^2$, (c) the trigonal lattice, and (d) $\Lambda=\mathds Z^3$. The shaded region on the left highlights orders $r\le 3$, for which all graphs admit analytic closed forms. For these cases, the points measure exclusively the error of the Monte Carlo summation.}
    \label{fig:mc_comparison}
\end{figure}

We finally provide an extensive comparison of our method against state-of-the-art published Monte Carlo results for high-order series expansions of the well-known long-range transverse-field Ising model (LRTFIM) for the ground state energy density $\varepsilon (\lambda)$  (0qp) and the dispersion relation $\omega(\lambda,\bm k)$ (1qp). In total, we compare $114$ perturbative series including $998$ series coefficients. Geometries include chains ($\Lambda=\mathds Z$), square and trigonal lattices, as well as the three-dimensional cubic lattice $\Lambda=\mathds Z^3$, and large ranges of the reduced interaction exponent $\sigma$ are explored, including values close to the lattice dimension. Different wavevectors $\bm k$ at the minimum of the ferromagnetic dispersion relation ($\lambda>0$) and the minimum of the antiferromagnetic dispersion relation ($\lambda<0$) are explored.

The Monte Carlo data has been published in the following sources. The square and triangular lattice 1qp gap series coefficients can be found in the Supplemental Material of Ref.~\cite{fey2019quantum} (Tables~I--IV), contributing $31$ series and $261$ coefficients. Further, $56$ series ($463$ coefficients) are taken from Appendix~F of
Ref.~\cite{Fey2020Diss}, including the chain gap series at $k=0$ and $k=1/2$ (Tables~F.2
and~F.3), the square and triangular lattice 0qp series (Tables~F.4 and~F.7), and
the cubic lattice 1qp series at momenta $\bm k=\bm 0$ and $\bm k=(1/2,1/2,1/2)^T$ (Tables~F.10
and~F.11). Note that the same appendix reprints the two-dimensional 1qp series of
Ref.~\cite{fey2019quantum} (Tables~F.5, F.6, F.8 and~F.9). The high-order chain 0qp
series to order 13 and the ferromagnetic chain 1qp series to order 11 stem from
Ref.~\cite{Langheld2022}, with data published on Zenodo \cite{Langheld2022Data}, contributing $8$ series (80 coefficients). We are grateful to P.~Adelhardt for providing us with an additional $19$ series ($194$ coefficients) for the chain at wavevector at the minimum of the antiferromagnetic dispersion relation and the square lattice up to order 13 for 0qp and order~11 for 1qp \cite{AdelhardtPrivComm}. Derived quantities from this data have been published in \cite{adelhardt2024monte,PhysRevB.102.174424}.

The results of the comparison are presented in \cref{tab:mc_comparison}. Here, the graph-zeta values were obtained using $n=2048$ momentum grid points per dimension for $d=1$, $n=48$ points for $d=2$, and $n=10$ points for $d=3$. The left column shows the computed series, grouped by the underlying lattice, including 0qp/1qp sectors and different momenta. The next column details the range of $\sigma$-values used, giving the total count in brackets, followed by the reference to the data source. After that, we provide the maximum error of the series coefficient at order $r$ (reaching order 13 for 0qp and order 11 in case of 1qp), obtained with GZL, at given momentum $\bm k$ in units of the Monte Carlo standard deviation $\sigma_{\mathrm{MC}}$, defined as
\[
E_r=\max_{\sigma}\Big\vert c_r(\bm k)-c_r^{\mathrm{MC}}(\bm k)\Big \vert/\sigma_{\mathrm{MC}}.
\]
Note here that the low orders $r=1,2,3$ consist entirely of graphs that admit analytic forms (either bridges or cycles without momenta) from \cref{sec:basic_blocks}, which evaluate to machine precision. Therefore, these columns measure exclusively the error of Monte Carlo summation, which is marked in the table by an index $r^
\mathrm{A}$ and by greying out the respective error values. Note that we have included only $\sigma$ values with $\sigma\ge 1/2$ in these measurements. Monte Carlo points for $\sigma = 0.1$ exist, yet here, Monte Carlo does not provide faithful results anymore due to extremely slow convergence of the sums. This is signalled by large deviations already in the analytic orders. Further, the top order for the square and trigonal lattices from \cite{Fey2020Diss} are excluded from our comparison (values greyed out), as they are suspected to be defective. This is based on two points. First, for the square lattice, the two available Monte Carlo runs of that series agree
to within $0.3\,\sigma_{\mathrm{MC}}$ at order 9, yet differ from one another by $25\,\sigma_{\mathrm{MC}}$ at order 10. Second, the given uncertainties in \cite{Fey2020Diss} degrade by a factor of $75$ between order 9 and order 10, providing a strong indication for an underconverged result of the cluster run. Meanwhile, the graph zeta results were converged and did not change under further increase of the discretization parameters.

We now provide a combined error analysis. Excluding the greyed-out cases (analytic rows and the $17$ coefficients of the two underconverged reference top orders), we obtain a maximum error of $3.1\,\sigma_\mathrm{MC}$, with median $0.12\,\mathrm{\sigma_\mathrm{MC}}$. Including all data, $89\,\%$ of the coefficients are reproduced within $1\,\sigma_{\mathrm{MC}}$ and $97\,\%$ within $2\,\sigma_{\mathrm{MC}}$. We therefore conclude that, within the statistical uncertainties of the Monte Carlo reference data, the
graph-zeta values reproduce the Monte Carlo series.

In \cref{fig:mc_comparison}, we provide a plot of the contents of \cref{tab:mc_comparison}, displaying the normalized error $E_r$ as a function of the order $r$ for (a) the integer lattice, (b) the square lattice, (c) the trigonal lattice, and (d) the cubic lattice. Here, the black dots correspond to 0qp, the red dots to ferromagnetic 1qp, and the blue dots to antiferromagnetic 1qp. Meanwhile, solid lines and points refer to the Fey thesis as the data source, whereas the dashed line corresponds to data from Langheld and Adelhardt. The dashed horizontal line shows the $2\,\sigma_{\mathrm{MC}}$ boundary, whereas the shaded region on the left displays the analytic orders, where graph-zeta is analytic and only the error of the Monte Carlo reference is measured. The two underconverged top orders of the chain and trigonal lattice ground state energy density are excluded from the plot, as discussed before. 

We close this session with a short discussion on runtime needed to reproduce the Monte Carlo reference data. A Monte Carlo run underlying Refs.~\cite{fey2019quantum,Fey2020Diss,
Langheld2022,adelhardt2024monte} is of cluster scale, with a single series at a single momentum
requiring about $24\,\mathrm{h}$ on $72$ cores for approximately $20$ seeds, resulting in roughly
$3\times10^4$ core-hours. Meanwhile, all $114$ series (998 coefficients) in this comparison were computed within $5$ minutes on eight cores of an Apple M1 Max processor, including the full $n^d$ momentum grid. On a single core, the 3D series at $n=10$ (1qp, order 10) requires approximately half a minute for the full momentum grid with $|\Lambda_n|=10^3$ grid points.

\section{Physical application: Anomalous dispersion relation in a 3D quantum antiferromagnet with long-range interactions}
\label{sec:lrtfim}

In the last section, we have demonstrated agreement of our method with analytical and numerical benchmarks as well as with an extensive collection of published Monte Carlo data. In this section, we will derive physical results in a regime, where no Monte Carlo data is available anymore. Namely, we will show a finely resolved dispersion relation of a 3D LRTFIM on the cubic lattice $\Lambda=\mathds Z^3$ for a power-law interaction with $\sigma=1$ ($\nu=4$) close to the system dimension $d$. As Monte Carlo runs require typically a full day on a cluster for a single momentum evaluation, resolving the dispersion relation on a complete momentum grid is impossible. Meanwhile, the graph-zeta method produces the whole momentum grid at the same numerical cost than the single evaluation, as it is compatible with FFT. Following \cref{def:corpus}, the single-quasiparticle
dispersion relation is the corpus sum series
\[
  \omega(\bm k,\lambda)=\sum_{r=0}^{r_\mathrm{max}}c_r(\bm k)\,\lambda^{r},
\]
in units of the bare gap, where $c_0=1$ and with the remaining coefficients obtained via the graph-zeta method,
\[
  c_r(\bm k)=\sum_{G\in\mathcal C_r}a_r(G)\,\zeta_{\Lambda,(G,\mathcal K_\nu)}(\bm k),\quad r>0.
\]
Here, choosing $\lambda>0$ corresponds to a ferromagnetic coupling and $\lambda<0$ to an antiferromagnet.

\begin{figure}
    \centering 
    \includegraphics[width=.7\textwidth]{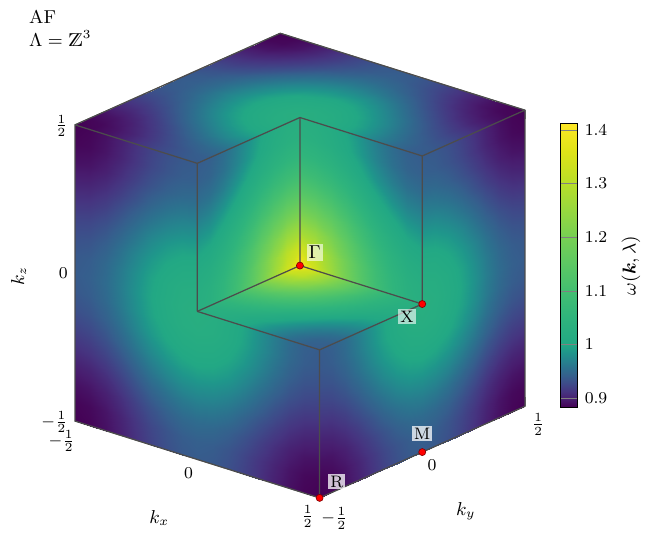}
    \caption{Three-dimensional colour plot of the dispersion relation
$\omega(\bm k,\lambda)$ of the 3D LRTFIM on the cubic
lattice $\Lambda=\mathds Z^3$ with exponent $\sigma=1$ ($\nu=d+\sigma=4$)
at antiferromagnetic coupling $\lambda=-0.03$, obtained from the order $r=10$ series
expansion (7\,136 graphs, up to $10d=30$-dimensional sums) evaluated with graph zeta functions on a momentum grid of
$N=16^3=4096$ points (colours interpolated linearly between grid
points) with energies in units of the bare gap.  The octant $k_x>0$, $k_y<0$,
$k_z>0$ is cut away to expose the planes $k_x=0$, $k_y=0$ and $k_z=0$
through $\Gamma$, where the dispersion exhibits a nonanalyticity.  The red dots mark the
high-symmetry points, in units of the reciprocal lattice vectors,
$\Gamma=(0,0,0)^T$, $\mathrm X=(\tfrac12,0,0)^T$, $\mathrm M=(\tfrac12,0,-\tfrac12)^T$
and $\mathrm R=(\tfrac12,-\tfrac12,-\tfrac12)^T$. Note that M and R are equivalent to the
standard $(\tfrac12,\tfrac12,0)^T$ and $(\tfrac12,\tfrac12,\tfrac12)^T$ by
the cubic symmetry. The evaluation of the full momentum grid on a single core of an Apple M1 Max processor required approximately 10 minutes.}
\label{fig:dispersion_3D_afm}
\end{figure}

We provide the, to the best of our knowledge, first finely resolved dispersion relation of a three-dimensional LRTFIM obtained from a high-order series expansion in \cref{fig:dispersion_3D_afm}. Here, the reduced exponent $\sigma=1$ was chosen and a perturbative series up to order $r=10$ was computed at antiferromagnetic coupling $\lambda=-0.03$. An octant was removed from the cubic Brillouin zone to expose the planes $k_x=0$, $k_y=0$, and $k_z=0$ that cut through the high-symmetry points
\[
\Gamma=(0,0,0)^T,~\mathrm X=(\tfrac12,0,0)^T,~ \mathrm M=(\tfrac12,0,-\tfrac12)^T,~ 
\text{and}~\mathrm R=(\tfrac12,-\tfrac12,-\tfrac12)^T
\]
which are marked in the plots as red dots.  The global maximum of the dispersion is reached at the $\Gamma$ point, with a local minum at $\mathrm{M}$ and the global minimum at $\mathrm{R}$.

A cut along along the standard momentum path connecting the high-symmetry points is provided in \cref{fig:dispersion_path} for $\lambda=-0.03$, the choice of \cref{fig:dispersion_3D_afm}, as well as two other couplings $\lambda=-0.02$ and $\lambda=-0.01$. The cut clearly shows a non-analyticity of the dispersion relation at the $\Gamma$ point, scaling as $\vert\bm k \vert$. This non-analyticity is recovered analytically from simple Epstein zeta bridges, which include singularities of the form $c_\nu |\bm k|^{\sigma}$, see \cref{lem:holomorphy}\,(b), which results in the observed behavior at $\sigma=1$. The presence of such non-analyticities is known as anomalous dispersion. It constitutes a direct fingerprint of a long-range interaction tail and is sought after in real materials. In a current experiment, a dipolar XY Rydberg atom simulator resolves the
dispersion of elementary excitations of a two-dimensional long-range XY model, recovering, for ferromagnetic couplings, a non-linear small-$\bm k$ dispersion attributable to the dipolar tail in \cite[Fig.~2]{chen2025spectroscopy}. In solid state systems, the dispersive excitations of the triangular-lattice quantum Ising magnet KTmSe$_2$ have been resolved by inelastic neutron scattering in \cite{zheng2023exchange}. In our companion paper \cite{duft2026}, evidence is presented that its dispersion can be explained by nearest-neighbor exchange interactions accompanied by long-range dipolar interactions, based on high-order series expansions evaluated with the GZL library \cite{gzl2026}.

\begin{figure}
    \centering 
    \includegraphics[width=\textwidth]{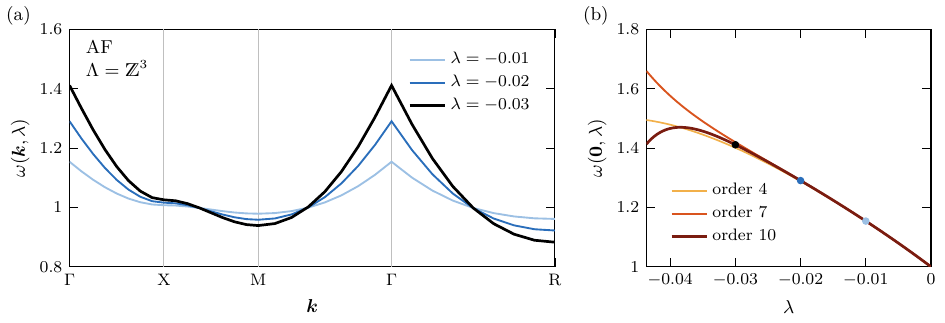}
    \caption{(a) Cut along the high-symmetry path of the dispersion relation $\omega(\bm k,\lambda)$ for the 3D antiferromagnetic LRTFIM with $\Lambda=\mathds Z^3$ with $\sigma=1$ for different choices of the perturbation $\lambda\in \{-0.01,-0.02,-0.03\}$ with $\lambda=-0.03$ corresponding to \cref{fig:dispersion_3D_afm}. Panel (b) shows displays the dispersion at the $\Gamma$ point ($\bm k=\bm 0$) as a function of $\lambda$ for different expansion orders $r\in \{4,7,10\}$. The colored dots mark the $\lambda$ values chosen in panel (a).}
    \label{fig:dispersion_path}
\end{figure}

\section{Outlook}
\label{sec:outlook}

In this work, we have provided a precise and efficient method that permits the computation of the high-dimensional lattice sums appearing in high-order series expansions of observables of gapped quantum lattice models. Due to its generality, being applicable to general $d$-dimensional lattices, general interactions and general gapped quantum models, the framework offers vast application potential. While high-order series expansions belong to the small set of methods that, in principle, can provide a reliable ground truth for an infinite quantum system, they so-far suffered from the statistical error of the required Monte Carlo summations and their associated numerical cost. This cost amounts typically to a full day on a cluster for a single momentum evaluation for the long-range TFIM. Using a combination of graph theory, generalized zeta functions, Fourier analysis of homogeneous distributions, and tensor network theory, the scaling of numerical work is reduced significantly, with the largest percentage of objects evaluating at log-linear cost, compared to exponential cost in the number of graph nodes for direct approaches. This permits to reproduce published Monte Carlo data for 1D, 2D, and 3D lattices within minutes on a laptop at improved precision. Importantly, the method provides the full momentum grid at the cost of a single $\bm k$ value, allowing to resolve the full dispersion relation. In addition, it analytic access to the leading order singularities in the wavevector argument, directly related to quantum critical exponents. The numerical method is in direct use in the companion paper \cite{duft2026}, which incorporates it into a comprehensive linked-cluster perturbative method for gapped quantum lattice models. This is applied to an in-depth study of the quantum critical properties of LRTFIM, including a reproduction of experimental measurements. In order to make the work as accessible as possible to a broad audience, we provide a high-performance implementation of the method in the open-source Graph Zeta Library (GZL) \cite{gzl2026}. It is our hope, dear reader, that this work will prove helpful to you in discovering relevant new effects in infinite quantum lattice models that would otherwise be out of reach. Future work will be dedicated to accelerating numerical convergence of the algebra further by including anisotropic generalized zeta functions, already discussed in a micromagnetics setting in \cite{buchheit2026zeta}. The method will also be adapted to multi-atomic lattices, extending its reach to a large number of experimentally available compounds. This will include frustrated systems, where Monte Carlo methods can be ineffective. Finally, extensions to other models, other perturbative approaches, to finite temperature \cite{Burkard2026A,Burkard2026B}, as well as to systems with correlated disorder are worth exploring.

\section*{Acknowledgements}
AB is grateful to Jan Koziol for the inspiring discussions that initiated this work, and to Kai Phillip Schmidt and his group, in particular Antonia Duft, Patrick Adelhardt and Jan Koziol, for their time and support and for their hard work on the physical application of the numerical method. AB would like to thank Torsten Keßler, Kirill Serkh, Manfred Sigrist, and Christian Lubich for helpful discussions.  AB is grateful to Patrick Klitzke for introducing him to AI-assisted specification-driven development, which has proven helpful in the development of the Graph Zeta Library.

AB acknowledges supported by the Klaus-Tschira Stiftung under Grant No. 00.025.2025. The authors gratefully acknowledge the scientific support and HPC
resources provided by the Erlangen National High Performance Computing Center (NHR@FAU) of the Friedrich-Alexander-Universität Erlangen-Nürnberg (FAU)
under the NHR project n101af.

Claude Code (using models Opus 4.8, Opus 5, Fable 5, Fable 5.1) was used in the development of the Graph Zeta Library and in the preparation of this manuscript, in particular for literature research, informal proof checking, data analysis, and visualization. All results were verified by the authors.

\printbibliography

@article{epstein1903theorieI,
  title={{Zur Theorie allgemeiner Zetafunctionen}},
  author={Epstein, P.},
  journal={Math. Ann.},
  volume={56},
  pages={615–644},
  year={1903},
  publisher={Springer},
  url = {https://doi.org/10.1007/BF01444309}
}

@article{epstein1903theorieII,
  title={{Zur Theorie allgemeiner Zetafunktionen. II}},
  author={Epstein, P.},
  journal={Math. Ann.},
  volume={63},
  pages={205--216},
  year={1906},
  publisher={Springer},
  url = {https://doi.org/10.1007/BF01449900}
}

@book{campa2014physics,
  title={Physics of long-range interacting systems},
  author={Campa, A. and Dauxois, T. and Fanelli, D. and Ruffo, S.},
  year={2014},
  publisher={OUP Oxford}
}

@book{gelfand1964generalizedI,
  title={Generalized Functions},
  subtitle={Volume I. Properties and Operations},
  author={Gel'fand, I. M. and Shilov, G. E.},
  publisher={Academic Press},
  year={1964}
}

@article{PhysRevB.102.174424,
  title = {Quantum criticality and excitations of a long-range anisotropic XY chain in a transverse field},
  author = {Adelhardt, P. and Koziol, J. A. and Schellenberger, A. and Schmidt, K. P.},
  journal = {Phys. Rev. B: Condens. Matter},
  volume = {102},
  issue = {17},
  pages = {174424},
  numpages = {10},
  year = {2020},
  month = {Nov},
  publisher = {American Physical Society},
  url = {https://doi.org/10.1103/PhysRevB.102.174424},
}

@article{fey2019quantum,
  title={Quantum criticality of two-dimensional quantum magnets with long-range interactions},
  author={Fey, S. and Kapfer, S. C. and Schmidt, K. P.},
  journal={Phys. Rev. Lett.},
  volume={122},
  number={1},
  pages={017203},
  year={2019},
  publisher={APS},
  url={https://doi.org/10.1103/PhysRevLett.122.017203}
}

@article{duffy1982quadrature,
  title={Quadrature over a pyramid or cube of integrands with a singularity at a vertex},
  author={Duffy, M. G.},
  journal={SIAM J. Numer. Anal.},
  volume={19},
  number={6},
  pages={1260-1262},
  year={1982},
  url={https://doi.org/10.1137/0719090}
}

@book{hoermander1966introduction,
    title = {An Introduction to Complex Analysis in Several Variables},
    author = {Hörmander, L.},
    publisher = {Van Nostrand},
    year = {1966}
}

@article{buchheit2023exact,
  title = {Exact continuum representation of long-range interacting systems and emerging exotic phases in unconventional superconductors},
  author = {Buchheit, Andreas A. and Ke\ss{}ler, Torsten and Schuhmacher, Peter K. and Fauseweh, Benedikt},
  journal = {Phys. Rev. Res.},
  volume = {5},
  issue = {4},
  pages = {043065},
  numpages = {22},
  year = {2023},
  month = {Oct},
  publisher = {American Physical Society},
  doi = {10.1103/PhysRevResearch.5.043065}
}

@article{buchheit2024epstein,
  title={Computation and properties of the Epstein zeta function with applications to quantum systems},
  author={Buchheit, Andreas A and Busse, Jonathan K and Gutendorf, Ruben},
  journal={IMA Journal of Numerical Analysis},
  pages={drag057},
  year={2026},
  publisher={Oxford University Press},
  doi={10.1093/imanum/drag057}
}

@article{adelhardt2024monte,
  title={Monte Carlo based techniques for quantum magnets with long-range interactions},
  author={Adelhardt, Patrick and Koziol, Jan A and Langheld, Anja and Schmidt, Kai P},
  journal={Entropy},
  volume={26},
  number={5},
  pages={401},
  year={2024},
  publisher={MDPI}
}

@article{robles2025exact,
  title={Exact lattice summations for Lennard-Jones potentials coupled to a three-body Axilrod--Teller--Muto term applied to cuboidal phase transitions},
  author={Robles-Navarro, Andres and Cooper, Shaun and Buchheit, Andreas A and Busse, Jonathan K and Burrows, Antony and Smits, Odile and Schwerdtfeger, Peter},
  journal={The Journal of Chemical Physics},
  volume={163},
  number={9},
  year={2025},
  publisher={AIP Publishing},
  url={https://doi.org/10.1063/5.0276677}
}

@article{buchheit2025epstein_method,
  author  = {Buchheit, Andreas A. and Busse, Jonathan K.},
  title   = {Epstein zeta method for many-body lattice sums},
  journal = {Numerische Mathematik},
  year    = {2026},
  month   = jul,
  doi     = {10.1007/s00211-026-01558-y},
}

@article{koziol2024order,
  title={Order-by-disorder and long-range interactions in the antiferromagnetic transverse-field Ising model on the triangular lattice—A perturbative point of view},
  author={Koziol, Jan Alexander and M{\"u}hlhauser, Matthias and Schmidt, Kai Phillip},
  journal={Results in Physics},
  volume={61},
  pages={107794},
  year={2024},
  publisher={Elsevier}
}

@article{adelhardt2025quantum,
  title={Quantum-critical and dynamical properties of the XXZ bilayer with long-range interactions},
  author={Adelhardt, Patrick and Duft, Antonia and Schmidt, Kai Phillip},
  journal={Physical Review B},
  volume={111},
  number={2},
  pages={024409},
  year={2025},
  publisher={APS}
}

@article{stoudenmire2012studying,
  title={Studying two-dimensional systems with the density matrix renormalization group},
  author={Stoudenmire, Edwin M and White, Steven R},
  journal={Annu. Rev. Condens. Matter Phys.},
  volume={3},
  number={1},
  pages={111--128},
  year={2012},
  publisher={Annual Reviews}
}

@inproceedings{sandvik2010computational,
  title={Computational studies of quantum spin systems},
  author={Sandvik, Anders W},
  booktitle={AIP Conference Proceedings},
  volume={1297},
  number={1},
  pages={135--338},
  year={2010},
  organization={American Institute of Physics}
}

@book{diestel2025graph,
  title={Graph Theory},
  author={Diestel, Reinhard},
  edition={6th},
  series={Graduate Texts in Mathematics},
  volume={173},
  year={2025},
  publisher={Springer},
  address={Heidelberg},
  isbn={978-3-662-70106-5},
  url={https://doi.org/10.1007/978-3-662-70107-2}
}

@article{Wegner1994,
  author  = {Wegner, Franz},
  title   = {Flow-equations for {H}amiltonians},
  journal = {Annalen der Physik},
  volume  = {506},
  number  = {2},
  pages   = {77--91},
  year    = {1994},
  doi     = {10.1002/andp.19945060203}
}

@article{Knetter2000,
  author  = {Knetter, Christian and Uhrig, G{\"o}tz S.},
  title   = {Perturbation theory by flow equations: dimerized and
             frustrated {$S=1/2$} chain},
  journal = {The European Physical Journal B},
  volume  = {13},
  pages   = {209--225},
  year    = {2000},
  doi     = {10.1007/s100510050026}
}

@article{Knetter2003,
  author  = {Knetter, Christian and Schmidt, Kai P. and
             Uhrig, G{\"o}tz S.},
  title   = {The structure of operators in effective
             particle-conserving models},
  journal = {Journal of Physics A: Mathematical and General},
  volume  = {36},
  number  = {29},
  pages   = {7889--7907},
  year    = {2003},
  doi     = {10.1088/0305-4470/36/29/302}
}

@article{Fey2016,
  author  = {Fey, Sebastian and Schmidt, Kai Phillip},
  title   = {Critical behavior of quantum magnets with long-range
             interactions in the thermodynamic limit},
  journal = {Physical Review B},
  volume  = {94},
  pages   = {075156},
  year    = {2016},
  doi     = {10.1103/PhysRevB.94.075156}
}

@article{Coester2015,
  author  = {Coester, Kris and Schmidt, Kai Phillip},
  title   = {Optimizing linked-cluster expansions by white graphs},
  journal = {Physical Review E},
  volume  = {92},
  pages   = {022118},
  year    = {2015},
  doi     = {10.1103/PhysRevE.92.022118}
}

@article{Wietek2018,
  author  = {Wietek, Alexander and L{\"a}uchli, Andreas M.},
  title   = {Sublattice coding algorithm and distributed memory
             parallelization for large-scale exact diagonalizations
             of quantum many-body systems},
  journal = {Physical Review E},
  volume  = {98},
  pages   = {033309},
  year    = {2018},
  doi     = {10.1103/PhysRevE.98.033309}
}

@article{hormann2023projective,
  title={Projective cluster-additive transformation for quantum lattice models},
  author={H{\"o}rmann, Max and Schmidt, Kai Phillip},
  journal={SciPost Physics},
  volume={15},
  number={3},
  pages={097},
  year={2023},
  doi = {10.21468/SciPostPhys.15.3.097}
}

@article{buchheit2026zeta,
  title={Zeta expansion for long-range interactions under periodic boundary conditions with applications to micromagnetics},
  author={Buchheit, Andreas Alexander and Busse, Jonathan Kaspar and Ke{\ss}ler, Torsten and Rybakov, Filipp N},
  journal={Journal of Computational Physics},
  pages={114885},
  year={2026},
  publisher={Elsevier},
  doi={10.1016/j.jcp.2026.114885}
}

@article{Tarjan1972,
  author  = {Tarjan, Robert},
  title   = {Depth-first search and linear graph algorithms},
  journal = {SIAM Journal on Computing},
  volume  = {1},
  number  = {2},
  pages   = {146--160},
  year    = {1972},
  doi     = {10.1137/0201010}
}

@article{HopcroftTarjan1973,
  author  = {Hopcroft, John and Tarjan, Robert},
  title   = {Algorithm 447: efficient algorithms for graph manipulation},
  journal = {Communications of the ACM},
  volume  = {16},
  number  = {6},
  pages   = {372--378},
  year    = {1973},
  doi     = {10.1145/362248.362272}
}

@book{CyganFKLMPPS15,
  author    = {Cygan, Marek and Fomin, Fedor V. and Kowalik, {\L}ukasz and
               Lokshtanov, Daniel and Marx, D{\'a}niel and Pilipczuk, Marcin and
               Pilipczuk, Micha{\l} and Saurabh, Saket},
  title     = {Parameterized Algorithms},
  publisher = {Springer},
  address   = {Cham},
  year      = {2015},
  isbn      = {978-3-319-21274-6},
  doi       = {10.1007/978-3-319-21275-3},
}

@article{RobertsonSeymour1986,
  author  = {Robertson, Neil and Seymour, P. D.},
  title   = {Graph minors. {II}. {A}lgorithmic aspects of tree-width},
  journal = {Journal of Algorithms},
  volume  = {7},
  number  = {3},
  pages   = {309--322},
  year    = {1986},
  doi     = {10.1016/0196-6774(86)90023-4},
}

@article{halin1976s,
  title={S-functions for graphs},
  author={Halin, Rudolf},
  journal={Journal of geometry},
  volume={8},
  number={1},
  pages={171--186},
  year={1976},
  publisher={Springer}
}

@misc{duft2026,
  author = {Duft, Antonia and Adelhardt, Patrick and Koziol, Jan Alexander
            and Buchheit, Andreas A. and Schmidt, Kai Phillip},
  title  = {Exact and fast series expansions for quantum models with long-range interactions},
  year   = {2026},
  note   = {Companion article, published in parallel with this work},
}

@software{gzl2026,
  author  = {Buchheit, Andreas A.},
  title   = {Graph Zeta Library ({GZL})},
  year    = {2026},
  version = {1.0.0, to be released soon},
  license = {AGPL-3.0-or-later},
  url     = {https://github.com/graph-zeta/gzl}
  }

@article{Bodlaender1998,
  author  = {Bodlaender, Hans L.},
  title   = {A partial $k$-arboretum of graphs with bounded treewidth},
  journal = {Theoretical Computer Science},
  volume  = {209},
  number  = {1--2},
  pages   = {1--45},
  year    = {1998},
  doi     = {10.1016/S0304-3975(97)00228-4},
}

@article{RoyleSokal2015,
  author  = {Royle, Gordon F. and Sokal, Alan D.},
  title   = {Linear bound in terms of maxmaxflow for the chromatic roots of series-parallel graphs},
  journal = {SIAM Journal on Discrete Mathematics},
  volume  = {29},
  number  = {4},
  pages   = {2117--2159},
  year    = {2015},
  doi     = {10.1137/130930133},
}

@article{Duffin1965,
  author  = {Duffin, R. J.},
  title   = {Topology of series-parallel networks},
  journal = {Journal of Mathematical Analysis and Applications},
  volume  = {10},
  number  = {2},
  pages   = {303--318},
  year    = {1965},
  doi     = {10.1016/0022-247X(65)90125-3},
}

@article{AjiMcEliece2000,
  author  = {Aji, Srinivas M. and McEliece, Robert J.},
  title   = {The Generalized Distributive Law},
  journal = {IEEE Transactions on Information Theory},
  volume  = {46},
  number  = {2},
  pages   = {325--343},
  year    = {2000},
  doi     = {10.1109/18.825794},
  issn    = {0018-9448}
}

@article{Dechter1999,
  author  = {Dechter, Rina},
  title   = {Bucket elimination: A unifying framework for reasoning},
  journal = {Artificial Intelligence},
  volume  = {113},
  number  = {1--2},
  pages   = {41--85},
  year    = {1999},
  doi     = {10.1016/S0004-3702(99)00059-4},
}

@article{MarkovShi2008,
  author  = {Markov, Igor L. and Shi, Yaoyun},
  title   = {Simulating quantum computation by contracting tensor networks},
  journal = {SIAM Journal on Computing},
  volume  = {38},
  number  = {3},
  pages   = {963--981},
  year    = {2008},
  doi     = {10.1137/050644756},
}

@article{kschischang2001factor,
  title={Factor graphs and the sum-product algorithm},
  author={Kschischang, Frank R and Frey, Brendan J and Loeliger, H-A},
  journal={IEEE Transactions on information theory},
  volume={47},
  number={2},
  pages={498--519},
  year={2001},
  publisher={IEEE},
  doi = {10.1109/18.910572}
}

@book{koller2009probabilistic,
  title={Probabilistic graphical models},
  author={Koller, Daphne and Friedman, Nir},
  volume={165},
  year={2009},
  publisher={MIT press Cambridge}
}

@article{Langheld2022,
  author  = {Langheld, Anja and Koziol, Jan A. and Adelhardt, Patrick and
             Kapfer, Sebastian C. and Schmidt, Kai Phillip},
  title   = {Scaling at quantum phase transitions above the upper critical dimension},
  journal = {SciPost Phys.},
  volume  = {13},
  pages   = {088},
  year    = {2022},
  doi     = {10.21468/SciPostPhys.13.4.088},
  eprint       = {2203.08081},
  archivePrefix = {arXiv},
  note    = {Series coefficients in the accompanying data set,
             \href{https://doi.org/10.5281/zenodo.6645107}{doi:10.5281/zenodo.6645107}},
}

@phdthesis{Fey2020Diss,
  author = {Fey, Sebastian},
  title  = {Investigation of Zero-Temperature Transverse-Field Ising Models
            with Long-Range Interactions},
  school = {Friedrich-Alexander-Universit\"at Erlangen-N\"urnberg},
  year   = {2020},
}

@unpublished{AdelhardtPrivComm,
  author = {Adelhardt, Patrick},
  title  = {Monte-Carlo series coefficients for the long-range transverse-field
            {I}sing model on the 2D square lattice},
  note   = {Private communication},
  year   = {2026},
}

@misc{Langheld2022Data,
  author    = {Langheld, Anja and Koziol, Jan A. and Adelhardt, Patrick and
               Kapfer, Sebastian C. and Schmidt, Kai Phillip},
  title     = {Raw data to ``{S}caling at quantum phase transitions above the
               upper critical dimension''},
  year      = {2022},
  publisher = {Zenodo},
  doi       = {10.5281/zenodo.6645107}
}

@article{chen2025spectroscopy,
  title={Spectroscopy of elementary excitations from quench dynamics in a dipolar XY Rydberg simulator},
  author={Chen, Cheng and Emperauger, Gabriel and Bornet, Guillaume and Caleca, Filippo and G{\'e}ly, Bastien and Bintz, Marcus and Chatterjee, Shubhayu and Liu, Vincent and Barredo, Daniel and Yao, Norman Y. and Lahaye, Thierry and Mezzacapo, Fabio and Roscilde, Tommaso and Browaeys, Antoine},
  journal={Science},
  volume={389},
  number={6759},
  pages={483--487},
  year={2025},
  publisher={AAAS},
  url={https://doi.org/10.1126/science.adn0618},
}

@article{zheng2023exchange,
  title={Exchange-renormalized crystal field excitations in the quantum Ising magnet KTmSe$_2$},
  author={Zheng, Shiyi and Wo, Hongliang and Gu, Yiqing and Luo, Rui Leonard and Gu, Yimeng and Zhu, Yinghao and Steffens, Paul and Boehm, Martin and Wang, Qisi and Chen, Gang and Zhao, Jun},
  journal={Phys. Rev. B},
  volume={108},
  number={5},
  pages={054435},
  year={2023},
  publisher={APS},
  url={https://doi.org/10.1103/PhysRevB.108.054435},
}

@article{semeghini2021probing,
  title={Probing topological spin liquids on a programmable quantum simulator},
  author={Semeghini, Giulia and Levine, Harry and Keesling, Alexander and Ebadi, Sepehr and Wang, Tout T and Bluvstein, Dolev and Verresen, Ruben and Pichler, Hannes and Kalinowski, Marcin and Samajdar, Rhine and others},
  journal={Science},
  volume={374},
  number={6572},
  pages={1242--1247},
  year={2021},
  publisher={American Association for the Advancement of Science}
}

@article{samajdar2021quantum,
  title={Quantum phases of Rydberg atoms on a kagome lattice},
  author={Samajdar, Rhine and Ho, Wen Wei and Pichler, Hannes and Lukin, Mikhail D and Sachdev, Subir},
  journal={Proceedings of the National Academy of Sciences},
  volume={118},
  number={4},
  pages={e2015785118},
  year={2021},
  publisher={National Academy of Sciences}
}

@article{Burkard2026A,
  title = {Dynamic Correlations of Frustrated Quantum Spins from High-Temperature Expansion},
  author = {Burkard, Ruben and Schneider, Benedikt and Sbierski, Bj\"orn},
  journal = {Phys. Rev. Lett.},
  volume = {136},
  issue = {5},
  pages = {056501},
  numpages = {10},
  year = {2026},
  month = {Feb},
  publisher = {American Physical Society},
  doi = {10.1103/jtjk-x2lw},
  url = {https://link.aps.org/doi/10.1103/jtjk-x2lw}
}

@article{Burkard2026B,
  title = {High-temperature series expansion of the dynamic Matsubara spin correlator},
  author = {Burkard, Ruben and Schneider, Benedikt and Sbierski, Bj\"orn},
  journal = {Phys. Rev. B},
  volume = {113},
  issue = {7},
  pages = {075102},
  numpages = {15},
  year = {2026},
  month = {Feb},
  publisher = {American Physical Society},
  doi = {10.1103/1l92-z6qd},
  url = {https://link.aps.org/doi/10.1103/1l92-z6qd}
}

\appendix

\end{document}